\documentclass[reqno]{amsart}
\usepackage{amssymb,amsmath,enumerate,amscd,mathrsfs,tikz}

\numberwithin{equation}{section}

\newtheorem{theorem}[equation]{Theorem}
\newtheorem{proposition}[equation]{Proposition}
\newtheorem{lemma}[equation]{Lemma}
\newtheorem{corollary}[equation]{Corollary}

\theoremstyle{definition}
\newtheorem{definition}[equation]{Definition}
\newtheorem{remark}[equation]{Remark}
\newtheorem{example}[equation]{Example}

\def\C{\mathbb C}
\def\E{\mathscr E}
\def\F{\mathscr F}
\def\K{\mathscr K}
\def\L{\mathscr L}
\def\N{\mathbb N}
\def\R{\mathbb R}
\def\P{\mathscr P}
\def\S{\mathscr S}
\def\T{\mathscr T}
\def\U{\mathscr U}
\def\g{\mathfrak g}
\def\h{\mathfrak h}
\def\Dom{\mathcal D}
\def\id{\textup{I}}
\def\Germs{{\mathcal P}_{\textup{sa}}}
\def\opm{\textup{op}_{\textit{M}}}

\DeclareMathOperator{\Diff}{Diff}
\DeclareMathOperator{\End}{End}
\DeclareMathOperator{\tr}{tr}
\DeclareMathOperator{\spec}{spec}
\DeclareMathOperator{\res}{res}
\DeclareMathOperator{\ran}{ran}
\DeclareMathOperator{\Span}{span}
\DeclareMathOperator{\sig}{sgn}
\DeclareMathOperator{\SF}{SF}
\DeclareMathOperator{\sym}{ \sigma\!\!\!\sigma}
\def\Hol{\mathfrak{H}}
\def\Mero{\mathfrak{M}}

\begin{document}
\title[Extensions of scaling invariant operators]{Extensions of symmetric operators that are invariant under scaling and applications to indicial operators}

\author{Thomas Krainer}
\address{Penn State Altoona\\ 3000 Ivyside Park \\ Altoona, PA 16601-3760}
\email{krainer@psu.edu}

\begin{abstract}
Indicial operators are model operators associated to an elliptic differential operator near a corner singularity on a stratified manifold. These model operators are defined on generalized tangent cone configurations and exhibit a natural scaling invariance property with respect to dilations of the radial variable. In this paper we discuss extensions of symmetric indicial operators from a functional analytic point of view. In the first, purely abstract part of this paper, we consider a general unbounded symmetric operator that exhibits invariance with respect to an abstract scaling action on a Hilbert space, and we describe its extensions in terms of generalized eigenspaces of the infinitesimal generator of this action. Among others, we obtain a Green formula for the adjoint pairing, an algebraic formula for the signature, and in the semibounded case explicit descriptions of the Friedrichs and Krein extensions. In the second part we consider differential operators of Fuchs type on the half axis with unbounded operator coefficients that are invariant under dilation, and show that under suitable ellipticity assumptions on the indicial family these operators fit into the abstract framework of the first part, which in this case furnishes a description of extensions in terms of polyhomogeneous asymptotic expansions. We also obtain an analytic formula for the signature of the adjoint pairing in terms of the spectral flow of the indicial family for such operators.
\end{abstract}


\maketitle


\section{Introduction}

\noindent
We begin by describing in informal terms the setting for the problems that this paper addresses. Let $M$ be a singular manifold. A corner singularity is a point $p$ with an open neighborhood $U(p) \subset M$ that is modeled on a neighborhood
$$
U(p_0) \subset \bigl(\overline{\R}_+ \times Y\bigr)/\bigl(\{0\} \times Y\bigr)
$$
of $p_0 = \bigl(\{0\} \times Y\bigr)/\bigl(\{0\} \times Y\bigr)$, where the link manifold $Y$ is generally also a manifold with singularities. If $Y$ is closed and compact this is the setting of a conical singularity. The inferred splitting of variables can be thought of as generalized polar coordinates, where the manifold $Y$ represents the domain for the spherical variables, and the variable $x > 0$ is the radial (distance) variable from the corner point. The space ${\mathbb R}_+ \times Y$ is a generalized tangent cone equipped with a model cone metric $dx^2 + x^2g_Y$ for some metric $g_Y$ on $Y$, and locally the map
$$
U(p_0) \setminus \{p_0\} \to U(p) \setminus \{p\}
$$
may be thought of as a generalized exponential map centered at $p_0$ from the tangent cone to the manifold. In particular, an elliptic differential operator $A$ near $p$ on $M$ can be pulled-back via this map to an elliptic differential operator near $p_0$ on the tangent cone $\R_+\times Y$, also denoted by $A$ for the moment. Write
$$
A = x^{-m}\sum\limits_{j=0}^ma_j(x,y,D_y)(xD_x)^j
$$
near $p_0$, where $m > 0$, and $a_j(x,y,D_y)$ is a differential operator on $Y$ of order $m-j$. We assume that the coefficients $a_j(x,y,D_y)$ depend continuously on $x$ up to $x=0$. The tangent cone admits an $\R_+$-action by scaling the variable $x$, and pull-back of functions with respect to this action gives rise to
\begin{equation}\label{kappaintro}
\bigl(\kappa_{\varrho}u\bigr)(x,y) = u(\varrho x,y), \quad \varrho > 0,
\end{equation}
for functions $u(x,y)$ on $\R_+\times Y$. Define
$$
A_{\wedge}u = \lim\limits_{\varrho \to \infty} \varrho^{-m}\kappa_{\varrho}^{-1}A\kappa_{\varrho}u,
$$
where $u$ has compact support in $x$. In this way we obtain the indicial operator $A_{\wedge}$, the model operator associated with $A$ at the corner singularity. The operator $A_{\wedge}$ is invariant under scaling in the sense that
\begin{equation}\label{Awedgeinvariance}
A_{\wedge} = \varrho^m\kappa_{\varrho}A_{\wedge}\kappa_{\varrho}^{-1}, \quad \varrho > 0,
\end{equation}
by the limit construction; we have
\begin{equation}\label{Awedge}
A_{\wedge} = x^{-m}\sum\limits_{j=0}^m a_j(xD_x)^j,
\end{equation}
where $a_j = a_j(0,y,D_y)$.

Our objective here is to describe $L^2$-extensions of the operator $A_{\wedge}$. The geometric $L^2$-space on $\R_+\times Y$ subject to the model metric $dx^2 + x^2g_Y$ identifies with the weighted space $x^{-\alpha}L^2_b(\R_+;L^2(Y;g_Y))$ for $\alpha = \frac{\dim Y + 1}{2}$, where $L^2_b$ denotes the $L^2$-space on $\R_+$ with respect to Haar measure $\frac{dx}{x}$. Multiplication by $x^{\alpha} : x^{-\alpha}L^2_b \to L^2_b$ is a unitary equivalence, and because the class of operators $A$ considered here and the construction leading to $A_{\wedge}$ are preserved under conjugation by arbitrary powers $x^{\alpha}$ we can base all considerations on the Hilbert space $L^2_b(\R_+;L^2(Y;g_Y))$. Observe that the scaling action \eqref{kappaintro} is unitary on this space.

If $Y$ is a stratified manifold with boundary/singular set $\Sigma \subset Y$, generalized ideal boundary conditions along the lateral boundary $\R_+\times\Sigma$ are required. Abstractly these are vanishing conditions in the form $T_{\wedge}u = 0$, where $T_{\wedge}$ ought to be thought of as the lateral model boundary condition for $A_{\wedge}$ associated to a lateral boundary condition $Tu = 0$ for $A$. The development of a robust analytic theory of elliptic ideal boundary conditions associated with singular sets at the implied level of generality remains the subject of ongoing investigations by the community of researchers working on singular PDEs, and as of yet only exists for certain configurations or certain specific geometric operators and boundary conditions (references are discussed below). We will forego this problem by treating the lateral boundary condition in functional analytic terms and consider the operator \eqref{Awedge} abstractly as an ordinary differential operator of Fuchs type on $\R_+$ as
\begin{equation}\label{Awedgeabstract}
A_{\wedge} = x^{-m}\sum\limits_{j=0}^m a_j(xD_x)^j : C_c^{\infty}(\R_+;E_1) \subset L^2_b(\R_+;E_0) \to L^2_b(\R_+;E_0),
\end{equation}
where $E_0$ and $E_1$ are Hilbert spaces such that $E_1 \hookrightarrow E_0$ is continuous and dense, and the $a_j : E_1 \to E_0$ are bounded. This matches the setting above with $E_0 = L^2(Y;g_Y)$, and $E_1$ the common domain of the indicial family
\begin{equation}\label{IndicialAwedge}
p(\sigma) = \sum\limits_{j=0}^ma_j\sigma^j : E_1 \subset E_0 \to E_0,
\end{equation}
considered here a family of unbounded operators in $E_0$. The space $E_1$ encodes the (lateral) boundary condition\footnote{As we consider the space $E_1$ to be fixed this does not include all relevant boundary conditions. For example, if $Y$ is a smooth, compact manifold with boundary, the general notion of lateral boundary conditions ought to include classical Shapiro-Lopatinsky elliptic boundary conditions on $\R_+\times\partial Y$, but fixing the space $E_1$ in this case means that the lateral model boundary condition $T_{\wedge}u = 0$ cannot differentiate in the $x$-direction.} along $\Sigma \subset Y$.

The closed extensions of the operator \eqref{Awedgeabstract} in $L^2_b(\R_+;E_0)$ are then expected to encode boundary conditions as $x \to 0$ at the corner point $p_0$ that are associated with realizations of the model operator $A_{\wedge}$ on $\R_+\times Y$ subject to previously chosen lateral boundary conditions, and it is those extensions that we are interested in here.

We will assume that \eqref{Awedgeabstract} is symmetric. The general case can often be reduced to this situation by considering instead
$$
{\mathcal A}_{\wedge} = \begin{bmatrix} 0 & A_{\wedge} \\ A_{\wedge}^{\sharp} & 0 \end{bmatrix} :
C_c^{\infty}\left(\R_+;\begin{array}{c} \tilde{E}_1 \\ \oplus \\ E_1 \end{array}\right) \subset
L^2_b\left(\R_+;\begin{array}{c} E_0 \\ \oplus \\ E_0 \end{array}\right) \to
L^2_b\left(\R_+;\begin{array}{c} E_0 \\ \oplus \\ E_0 \end{array}\right),
$$
where
$$
A_{\wedge}^{\sharp} = x^{-m}\sum\limits_{j=0}^m b_j(xD_x)^j : C_c^{\infty}(\R_+;\tilde{E}_1) \subset L^2_b(\R_+;E_0) \to L^2_b(\R_+;E_0)
$$
is the formal adjoint operator to $A_{\wedge}$ with indicial family
$$
p(\sigma^{\star})^* = \sum\limits_{j=0}^m b_j\sigma^j : \tilde{E}_1 \subset E_0 \to E_0,
$$
where $\sigma^{\star} = \overline{\sigma} - im$ is reflection about the line $\Im(\sigma) = -\frac{m}{2}$, and $\tilde{E}_1 \hookrightarrow E_0$ is assumed to be the common domain of the adjoints of \eqref{IndicialAwedge}.

Let $\Dom_{\min}$ be the domain of the closure of \eqref{Awedgeabstract}, and $\Dom_{\max}$ be the domain of the adjoint. The expectation
from known cases (particularly conical singularities) is that each $u \in \Dom_{\max}$ has a finite polyhomogeneous asymptotic expansion of the form
\begin{equation}\label{introasymptotic}
u \sim \sum\limits_{\sigma_0,j} e_{\sigma_0,j}\log^j(x)x^{i\sigma_0} \mod \Dom_{\min} \textup{ as $x \to 0$}
\end{equation}
with certain $e_{\sigma_0,j} \in E_1$ and characteristic values $\sigma_0 \in \spec_b(A_{\wedge}) \subset \C$, the boundary spectrum of $A_{\wedge}$, and the domains of extensions $\Dom_{\min} \subset \Dom \subset \Dom_{\max}$ of $A_{\wedge}$ then correspond to placing vanishing conditions on these asymptotic terms which establishes an explicit correspondence for such operators between their functional analytic extensions on the one hand, and an analytic notion of generalized boundary conditions by prescribing conditions on the asymptotic behavior of functions on the other hand.

The asymptotic expansion \eqref{introasymptotic} is to be understood as
$$
u - \omega  \sum\limits_{\sigma_0,j} e_{\sigma_0,j}\log^j(x)x^{i\sigma_0} \in \Dom_{\min},
$$
where $\omega \in C_c^{\infty}(\overline{\R}_+)$ is a cut-off function with $\omega \equiv 1$ near $x=0$. Observe that the functions that constitute this expansion for each $\sigma_0$ are generalized eigenfunctions of $\g = xD_x$ with eigenvalue $\sigma_0$, the infinitesimal generator of the radial scaling action $\kappa_{\varrho} = \varrho^{i\g}$ from \eqref{kappaintro} in $L^2_b(\R_+;E_0)$.  More precisely, after multiplying by $\omega$, these functions render generalized eigenfunctions of $\g$ modulo $C_c^{\infty}$ in $L^2_b(\R_+;E_0)$.

\medskip

\noindent
In the present paper we investigate indicial operators \eqref{Awedgeabstract} from two points of view. As only model operators are studied, we are going to shorten notation and write $A$ instead of $A_{\wedge}$ in the main body of the paper, but stay with $A_{\wedge}$ for the remainder of the introduction.

In the first, purely abstract part of this paper which spans Sections~\ref{sec-prelim} through \ref{KreinAbstract}, we consider a general unbounded symmetric operator $A_{\wedge}$ that exhibits an invariance property \eqref{Awedgeinvariance} with respect to an abstract unitary $\R_+$-action $\kappa_{\varrho} = \varrho^{i\g}$ on a Hilbert space. Assuming finite deficiency indices, we describe its extensions in terms of generalized eigenspaces of the infinitesimal generator $\g$ of this action (modulo $\Dom_{\min}$), which furnishes a correspondence between closed extensions of $A_{\wedge}$ on the one hand, and invariants of the underlying dynamics of $\kappa_{\varrho}$ in relation to the operator $A_{\wedge}$ on the other hand, which can be seen as a generalization of the notion of boundary conditions via asymptotic expansions \eqref{introasymptotic}. Among others, we obtain a Green formula for the adjoint pairing, an algebraic formula for its signature, and if $A_{\wedge}$ is semibounded we find explicit descriptions of the Friedrichs and Krein extensions. Moreover, for extensions that are invariant under $\kappa_{\varrho}$, we find an equivalent characterization of the order relation for semibounded extensions solely in terms of the boundary condition.
Much of the analysis in the first part occurs on the quotient space $\hat{\E}_{\max} = \Dom_{\max}/\Dom_{\min}$, which is finite-dimensional by assumption. The scaling action $\kappa_{\varrho}$ induces an action $\hat{\kappa}_{\varrho} = \varrho^{i\hat{\g}}$ on $\hat{\E}_{\max}$, and the invariance relation \eqref{Awedgeinvariance} implies that $\hat{\g} + i\frac{m}{2}$ is selfadjoint with respect to the indefinite Hermitian sesquilinear form on $\hat{\E}_{\max}$ that is induced by the adjoint pairing. Linear algebra for such spaces, in particular the Canonical Form Theorem in this setting \cite{GohbergLancasterRodman}, is then applied.

In the second part we consider operators $A_{\wedge}$ that are slight generalizations of the differential operators of Fuchs type \eqref{Awedgeabstract} in $L^2_b(\R_+;E_0)$ in that the power of the singular factor and the order of differentiation are allowed to be decoupled, and the dilation action $\kappa_{\varrho}u(x) = u(\varrho x)$ is considered on $L^2_b(\R_+;E_0)$. The precise assumptions on $A_{\wedge}$ and the indicial family \eqref{IndicialAwedge} are formulated in Section~\ref{IndicialOperators}. Our goal has been to impose minimal ellipticity assumptions. Aside from symmetry, all assumptions on $A_{\wedge}$ are invariant with respect to conjugating the operator by arbitrary powers $x^{\alpha}$. In particular, we are not imposing invertibility assumptions on \eqref{IndicialAwedge} along lines $\Im(\sigma) = \gamma$ for specific values of $\gamma$, but Fredholmness of \eqref{IndicialAwedge} and invertibility for large $|\Re(\sigma)| \gg 0$ with estimates are needed.
The objective is to show that these operators then directly fit into the abstract framework of the first part, which confirms that extensions are characterized by the asymptotic behavior \eqref{introasymptotic} of functions as $x \to 0$. In particular, all results of the first part apply and have a direct equivalent in terms of asymptotic expansions as $x \to 0$. The starting point for the analysis are von Neumann's formulas
$$
\Dom_{\max} = \Dom_{\min} \oplus \ker(A_{\wedge,\max} + i) \oplus \ker(A_{\wedge,\max} - i).
$$
One of the key arguments is proving that functions $u \in \ker(A_{\wedge,\max} \pm i)$ are decreasing rapidly as $x \to \infty$ (intuitively, this implies that describing $\Dom_{\max}/\Dom_{\min}$ can only be about the behavior of functions as $x \to 0$, as is expected). To achieve this, left-parametrices for $A_{\wedge,\max} \pm i$ are needed modulo remainders that produce the required rapid decay as $x \to \infty$, and we will obtain these parametrices by constructing in actuality right-parametrices for $A_{\wedge,\min} \mp i$ and then passing to adjoints. The parametrix construction requires a pseudodifferential calculus which we develop in Appendix~\ref{PseudoCalculus}.
In Section~\ref{MinimalDomain} we lay groundwork for weighted function spaces associated with $A_{\wedge}$, in particular in relation to $\Dom_{\min}$, while the main result for indicial operators that explicitly describes $\Dom_{\max}$ modulo $\Dom_{\min}$ in terms of asymptotic expansions is proved in Section~\ref{MaximalDomain} (see Theorem~\ref{DmaxThm}).

Finally, in Section~\ref{GreenIndicial} we prove that the signature of the adjoint pairing of the indicial operator $A_{\wedge}$ is given by the spectral flow of the indicial family \eqref{IndicialAwedge} along the line $\Im(\sigma) = -\frac{m}{2}$, where $p(\sigma)$ is selfadjoint and Fredholm by assumption. More precisely, each indicial root along $\Im(\sigma) = -\frac{m}{2}$ contributes a piece to the signature that is described algebraically in the first part of the paper in terms of invariants of the canonical form for the generator $\hat{\g}$ of the induced scaling action $\hat{\kappa}_{\varrho}$ on $\hat{\E}_{\max} = \Dom_{\max}/\Dom_{\min}$ and the adjoint pairing, and this piece is shown to coincide with the contribution of the spectral flow of $p(\sigma)$ across that indicial root. The proof is based on local arguments in analytic Fredholm and perturbation theory to bring $p(\sigma)$ into a normal form near an indicial root $\sigma_0$ with $\Im(\sigma_0) = -\frac{m}{2}$ from which these relations follow. The details of these arguments are relegated to Appendix~\ref{ModelPairing} and \ref{AnalyticCrossingsFiniteDimensional}. The normal form for $p(\sigma)$ also implies that semibounded indicial operators $A_{\wedge}$ always satisfy the sign condition. The sign condition is a technical property pertaining to the invariants of the canonical form of $\hat{\g}$ on $\hat{\E}_{\max}$ and the adjoint pairing, and is used in the first part of the paper to obtain complete descriptions of the Friedrichs and Krein extensions.

\medskip

\noindent
This paper addresses problems that have a long history, going back to seminal contributions by Kondrat'ev \cite{Kondratiev}, who investigated Fredholm solvability of classical boundary value problems in domains with isolated conical singularities on the boundary, and Cheeger \cite{Cheeger1979,Cheeger1983} who initiated geometric and global analysis on singular manifolds. Both contributions seeded independent developments (e.g. \cite{Dauge,KozlovMazyaRossmann,NazarovPlamenevskii} are rooted in Kondrat'ev's theory, \cite{AlbinLeichtnamMazzeoPiazzaWitt,AlbinLeichtnamMazzeoPiazzaHodge,BrueningSeeley87} draw their inspiration from Cheeger's works), which have increasingly been merging since the 1990s \cite{AmmannLauterNistor,AmmannGrosseNistor,Bohlen,GilMendoza,KrainerResolventsBVP,KrainerMendozaFirstOrder,Le97,MazzeoEdges,MazzeoVertman,NazaikinskiSavinSchulzeSternin,RBM2,SchroheSchulze1,SchroheSchulze2,SchuNH,SchulzeFields}, influenced by Melrose and Schulze. An analytic theory of solvability, regularity in Sobolev spaces with weights, and asymptotics for differential equations with unbounded operator coefficients with applications to partial differential equations in generalized cones and cylinders is developed in \cite{KozlovMazya}.

This paper is motivated by recent works \cite{AlbinGellRedman,AlbinLeichtnamMazzeoPiazzaHodge,HartmannLeschVertmanDomain,HartmannLeschVertmanAsymptotics,KrainerMendozaFirstOrder,KrainerMendozaFriedrichs,MazzeoVertman} that contribute towards developing elliptic theory for operators on incomplete stratified manifolds. The long-term goal is a robust elliptic theory of ideal boundary conditions associated with the singular strata, and the model operator level is essential for this. Specifically, our goals for this paper align with Lesch's book \cite{Le97}, and with the papers by Gil and Mendoza \cite{GilMendoza} and Coriasco, Schrohe, and Seiler \cite{CoriascoSchroheSeiler}. Both \cite{GilMendoza,Le97} investigate extensions of elliptic operators on compact manifolds with isolated conical singularities, and both consider fully general elliptic operators of any order which distinguish them from other investigations that center on operators of Dirac or Laplace type that near singularities are amenable to separation of variables and special function methods; \cite{CoriascoSchroheSeiler} extends these ideas to conic manifolds with boundary and classical boundary value problems that satisfy the Shapiro-Lopatinsky condition. In \cite{GilMendoza, Le97} methods from functional analysis, operator theory, and microlocal analysis rooted in the Mellin transform are systematically used and developed. One of the main results in \cite{GilMendoza} is an explicit description of the domain of the Friedrichs extension for general semibounded cone operators on compact manifolds, and we obtain an analogous result for the Friedrichs extension for indicial operators here, but with a different approach.

Compactness of the manifold and harnessing the standard local elliptic theory away from the singularities are essential in both \cite{GilMendoza,Le97}; arguments near the singularities are local as $x \to 0$ and generally do not transfer to noncompact tangent cone configurations $\R_+\times Y$ as $x \to \infty$. However, for conical singularities where $Y$ is closed and compact it is easy to see that uniform elliptic estimates on complete manifolds \cite{ShubinNoncompact} are applicable to $A_{\wedge}$ as $x \to \infty$ (and likewise is elliptic theory for scattering manifolds \cite{MelroseScattering,SchroheSG}), and therefore no contributions to the extensions of $A_{\wedge}$ can originate from $x \to \infty$, and the results from \cite{GilMendoza,Le97} apply in this setting. This point of view does not easily generalize to more complicated link manifolds $Y$, and since analyzing the behavior as $x \to \infty$ is essential for understanding general operators we take a different approach to this problem that leads to the pseudodifferential calculus in Appendix~\ref{PseudoCalculus}.

Our approach towards extensions in the first part is rooted in operator theory. Scaling invariance of unbounded operators under sets of unitary transformations has been used for example in singular perturbation theory; we refer to \cite{AlbeverioKurasov,HassiKuzhel,MakarovTsekanovskii} for additional information.

The spectral flow formula for the signature of the adjoint pairing for indicial operators appears to be new also for conical singularities. Generally, the spectral flow is related to the signature of crossing forms, and for real-analytic crossings this relationship has been investigated in \cite{EidamPiccione,FarberLevine,GiamboPiccionePortaluri}. Our proofs in Appendix~\ref{ModelPairing} and \ref{AnalyticCrossingsFiniteDimensional} are not reliant on these references. When combined with elliptic Fredholm theory on compact manifolds with singularities (for cases that have been sufficiently well understood), the spectral flow formula at the model operator level can be used to prove vanishing theorems for spectral flows and the index, generalizing the notion of cobordism invariance for these quantities. This relation will be pursued elsewhere.


\section{Preliminaries}\label{sec-prelim}

\noindent
Let $H$ be a separable complex Hilbert space, and let
$$
A : \Dom_c \subset H \to H
$$
be unbounded, densely defined, and symmetric. Then $A$ is closable with closure $\overline{A} = A^{**}$. Let $\Dom_{\min}$ denote the domain of the closure. Let $A^*$ be the adjoint of $A$, and let $\Dom_{\max}$ be its domain. The closed extensions of $A$ are the closed unbounded operators given by
$$
A^*\Big|_{\Dom} : \Dom \subset H \to H
$$
with domains $\Dom_{\min} \subset \Dom \subset \Dom_{\max}$ that are closed with respect to the graph norm induced by the graph inner product
$$
\langle u,v\rangle_A = \langle u,v\rangle_H + \langle A^*u,A^*v\rangle_H, \quad u,v \in \Dom_{\max}.
$$
In the sequel we will abuse notation and write
\begin{align*}
A_{\max} &= A^* : \Dom_{\max} \subset H \to H, \\
A_{\min} &= \overline{A} : \Dom_{\min} \subset H \to H, \\
A_{\Dom} &= A^*\Big|_{\Dom} : \Dom \subset H \to H
\end{align*}
for these operators.  When the domain is not explicitly specified we will simply write $A$. While this does not align with standard functional analytic conventions, it is consistent with defining extensions of (formally) symmetric differential operators $A$ in $L^2$, initially given on a space of test functions\footnote{From a functional analytic point of view, we could have instead started from a densely defined closed operator $A = A_{\max} : \Dom_{\max} \to H$ such that $A^* \subset A$, and defined $A_{\min} = A^*$. Then $A_{\min}$ is symmetric with $A_{\min}^* = A$, and all intermediate closed operators arise as restrictions of $A$.}.
We will assume throughout that $A$ has finite deficiency indices, which is equivalent to
$$
\dim\Dom_{\max}/\Dom_{\min} < \infty.
$$
In particular $A_{\Dom}$ is closed for every domain $\Dom_{\min} \subset \Dom \subset \Dom_{\max}$.

\bigskip

We furthermore assume that $H$ is equipped with a scaling action, i.e., a strongly continuous and unitary $\R_+$-action $\kappa_{\varrho} : H \to H$, $\varrho > 0$. Specifically:
\begin{enumerate}
\item For every $\varrho > 0$ we have $\kappa_{\varrho} \in \L(H)$ with $\kappa_{\varrho}^{-1} = \kappa_{\varrho}^*$.
\item We have $\kappa_{\varrho\varrho'} = \kappa_{\varrho}\kappa_{\varrho'}$ for all $\varrho,\varrho' > 0$, and $\kappa_1 = \textup{Id}$.
\item We have $\lim\limits_{\varrho \to 1}\kappa_{\varrho}u = u$ for all $u \in H$.
\end{enumerate}
Moreover, $A$ is supposed to be invariant under the scaling action in the following sense:
\begin{enumerate}
\setcounter{enumi}{3}
\item $\kappa_{\varrho} : \Dom_c \to \Dom_c$ for all $\varrho > 0$;
\item there exists $m > 0$ such that
$$
A = \varrho^m\kappa_{\varrho}A\kappa_{\varrho}^{-1} : \Dom_c \to H
$$
for all $\varrho > 0$.
\end{enumerate}

\begin{lemma}\label{Amaxmininv}
The group action $\kappa_{\varrho} : H \to H$ restricts to strongly continuous group actions
$\kappa_{\varrho} : \Dom_{\max} \to \Dom_{\max}$ and $\kappa_{\varrho} : \Dom_{\min} \to \Dom_{\min}$. We have
\begin{equation}\label{Amaxkappa}
A_{\max} = \varrho^m\kappa_{\varrho}A_{\max}\kappa_{\varrho}^{-1} : \Dom_{\max} \to H, \; \varrho > 0,
\end{equation}
and likewise for $A_{\min}$. Moreover, we have
$$
\|\kappa_{\varrho}\|_{\L(\Dom_{\max})} \leq \max\{1,\varrho^m\}, \; \varrho > 0.
$$
\end{lemma}
\begin{proof}
Let $u \in \Dom_c$ and $w \in \Dom_{\max}$ be arbitrary. Then
$$
\langle Au,\kappa_{\varrho}w \rangle = \langle \kappa^{-1}_{\varrho}Au,w \rangle = \varrho^m \langle A\kappa_{\varrho}^{-1}u,w\rangle = \varrho^m\langle \kappa_{\varrho}^{-1}u,A^*w\rangle = \langle u,\varrho^m\kappa_{\varrho}A^*w\rangle.
$$
This proves that $\kappa_{\varrho}w \in \Dom_{\max}$, and $A^*\kappa_{\varrho}w = \varrho^m\kappa_{\varrho}A^*w$ for $\varrho > 0$. Consequently, if $v \in \Dom_{\max}$, then $w = \kappa_{\varrho}^{-1}v = \kappa_{1/\varrho}v \in \Dom_{\max}$, and
$$
A^*v = A^*\kappa_{\varrho}w = \varrho^m\kappa_{\varrho}A^*w = \varrho^m\kappa_{\varrho}A^*\kappa_{\varrho}^{-1}v.
$$
This proves \eqref{Amaxkappa}. We next prove the strong continuity of $\kappa_{\varrho}$ on $\Dom_{\max}$. To that end, let $v \in \Dom_{\max}$ be arbitrary. Then $\kappa_{\varrho}v \to v$ in $H$ as $\varrho \to 1$ because of the strong continuity of $\kappa_{\varrho}$ on $H$. Likewise,
$$
A_{\max}\kappa_{\varrho}v = \varrho^m\kappa_{\varrho}A_{\max}v \to A_{\max}v
$$
as $\varrho \to 1$ in $H$, again because of the strong continuity of $\kappa_{\varrho}$ on $H$. Both combined now imply that $\kappa_{\varrho}v \to v$ as $\varrho \to 1$ in the graph norm on $\Dom_{\max}$, proving strong continuity of $\kappa_{\varrho}$ on $\Dom_{\max}$. For $u \in \Dom_{\max}$ we have
\begin{align*}
\|\kappa_{\varrho}u\|^2_A &= \|\kappa_{\varrho}u\|_H^2 + \|A_{\max}\kappa_{\varrho}u\|_H^2 
= \|\kappa_{\varrho}u\|_H^2 + \varrho^{2m}\|\kappa_{\varrho}A_{\max}u\|_H^2 \\
&= \|u\|_H^2 + \varrho^{2m}\|A_{\max}u\|_H^2  \leq \max\{1,\varrho^m\}^2\|u\|_A^2
\end{align*}
which implies the asserted norm estimate for $\kappa_{\varrho}$.

Finally, starting from the invariance of $\Dom_{\max}$ under $\kappa_{\varrho}$ and \eqref{Amaxkappa}, we obtain with the same reasoning as in the first part of this proof that the domain of the functional analytic adjoint $A_{\max}^* = A_{\min}$ is invariant under the action of $\kappa_{\varrho}$, and that $A_{\min}$ satisfies the analogue of \eqref{Amaxkappa}.
\end{proof}

\begin{definition}
A closed extension $A_{\Dom} : \Dom \subset H \to H$ with $\Dom_{\min} \subset \Dom \subset \Dom_{\max}$ is called \emph{stationary} or \emph{invariant} (with respect to $\kappa_{\varrho}$) if $\kappa_{\varrho} : \Dom \to \Dom$ for all $\varrho > 0$. In that case
$$
A_{\Dom} = \varrho^m\kappa_{\varrho}A_{\Dom}\kappa_{\varrho}^{-1} : \Dom \to H, \; \varrho > 0.
$$
By Lemma~\ref{Amaxmininv}, both $A_{\max}$ and $A_{\min}$ are invariant.
\end{definition}

\noindent
Let $\g : \Dom(\g) \subset H \to H$ be the infinitesimal generator of the group action $\kappa_{\varrho}$ on $H$. We will make the convention here that
$$
\kappa_{\varrho} = \varrho^{i\g} : H \to H, \; \varrho > 0,
$$
where
\begin{align*}
\Dom(\g) &= \{u \in H;\; {\mathbb R}_+ \ni \varrho \mapsto \kappa_{\varrho}u \in H \textup{ is differentiable}\}, \\
\g(u) &= (\varrho D_{\varrho})\kappa_{\varrho}u\Big|_{\varrho = 1} \textup{ for } u \in \Dom(\g).
\end{align*}
Note that references about $1$-parameter semigroups, including \cite{EngelNagel} referenced below, generally consider additive rather than multiplicative semigroups, and, in our notation, these references generally consider $i\g$ to be the generator rather than $\g$.

Since $\kappa_{\varrho}$ is unitary, $\g^* = \g$ is selfadjoint by Stone's Theorem. In particular, $\spec(\g) \subset \R$. The restrictions
$$
\kappa_{\varrho} : \Dom_{\max} \to \Dom_{\max},  \textup{ and }
\kappa_{\varrho} : \Dom_{\min} \to \Dom_{\min}
$$
are generated by the part of $\g$ in $\Dom_{\max}$ or $\Dom_{\min}$, respectively, i.e., the operators that act like $\g$ with domains
$$
\{u \in \Dom(\g)\cap\Dom_{\max};\; \g u \in \Dom_{\max}\} \textup{ and }
\{u \in \Dom(\g)\cap\Dom_{\min};\; \g u \in \Dom_{\min}\},
$$
see \cite[II.2.3]{EngelNagel}.

The norm estimate for $\kappa_{\varrho}$ in Lemma~\ref{Amaxmininv} (both for $1 \leq \varrho < \infty$ and for $0 < \varrho \leq 1$) in conjunction with the Hille-Yosida Generation Theorem (see \cite[II.3]{EngelNagel}) implies conclusions about the resolvent of the parts of the generator $\g$ in various subspaces of $H$. In particular, we get the following:

\begin{remark}\label{GeneratorResolvent}
\begin{enumerate}
\item For every $\sigma \in \C$ with $\Im(\sigma) \neq 0$ and every $v \in H$ the equation $(\g - \sigma)u = v$ has a unique solution $u \in \Dom(\g)$.
\item For every $\sigma \in \C$ with $\Im(\sigma) \notin [-m,0]$ and every $v \in \Dom_{\max}$ the unique solution $u \in \Dom(\g)$ to the equation $(\g - \sigma)u = v$ belongs to $\Dom_{\max}$. Moreover, if $v \in \Dom_{\min}$ then so is $u$.
\end{enumerate}
\end{remark}

\noindent
Let
$$
\hat{\E}_{\max} = \Dom_{\max}/\Dom_{\min} = \{\hat{u} = u + \Dom_{\min};\; u \in \Dom_{\max}\}
$$
be the quotient space. We will generally utilize hat-notation for objects that are associated with the quotient space. This is a Hilbert space, canonically isometric to the orthogonal complement $\E_{\max}$ of $\Dom_{\min}$ in $\Dom_{\max}$ with respect to the graph inner product induced by $A_{\max}$. The group $\kappa_{\varrho}$ induces an $\R_+$-action on the quotient
\begin{gather*}
\hat{\kappa}_{\varrho} : \hat{\E}_{\max} \to \hat{\E}_{\max}, \; \varrho > 0, \\
\hat{\kappa}_{\varrho}[u + \Dom_{\min}] = \kappa_{\varrho}u + \Dom_{\min}.
\end{gather*}
The action $\hat{\kappa}_{\varrho}$ is strongly continuous on $\hat{\E}_{\max}$ with generator
\begin{align*}
\hat{\g}(u + \Dom_{\min}) &= \g u + \Dom_{\min} \textup{ for } u+\Dom_{\min} \in \Dom(\hat{\g}), \\
\Dom(\hat{\g}) &= \{u + \Dom_{\min};\; u \in \Dom(\g)\cap\Dom_{\max},\; \g u \in \Dom_{\max}\},
\end{align*}
see \cite[II.2.4]{EngelNagel}. Because we assume $\dim\hat{\E}_{\max} < \infty$  we in fact have $\Dom(\hat{\g}) = \hat{\E}_{\max}$, and $\hat{\kappa}_{\varrho} = \varrho^{i\hat{\g}} : \hat{\E}_{\max} \to \hat{\E}_{\max}$ is uniformly continuous.

\begin{proposition}
We have $\spec(\hat{\g}) \subset \{\sigma \in \C;\; -m \leq \Im(\sigma) \leq 0\}$. If $\sigma_0 \in \spec(\hat{\g})$ with $\Im(\sigma_0) = 0$ or $\Im(\sigma_0) = -m$, then $\sigma_0$ is semisimple.

Moreover, if $\kappa_{\varrho} : H \to H$ is such that
$$
\lim\limits_{\varrho \to 0} \langle \kappa_{\varrho}u,v \rangle_H = 0
$$
for all\footnote{It suffices to check that $\lim\limits_{\varrho\to 0}\langle \kappa_{\varrho}u,u \rangle_H = 0$ for $u$ in a dense subspace of $H$, e.g., for $u \in \Dom_c$. To see this note that $T : H \times H \ni (u,v) \mapsto \bigl\{\R_+ \ni \varrho \mapsto \frac{1}{2}\bigl[\langle \kappa_{\varrho}u,v \rangle_H + \langle u,\kappa_{\varrho}v \rangle_H\bigr]\bigr\} \in (C_b(\R_+),\|\cdot\|_{\infty})$ is sesquilinear, Hermitian, and continuous. We have $T(u,u) \in C_0(\R_+)$ for $u \in \Dom_c$ by assumption and the unitarity of $\kappa_{\varrho}$, and by polarization $T(u,v) \in C_0(\R_+)$ for all $u,v \in \Dom_c$. Because $C_0(\R_+) \subset (C_b(\R_+),\|\cdot\|_{\infty})$ is closed we obtain $T(u,v) \in C_0(\R_+)$ for all $u,v \in H$ by density and continuity. The same argument shows that $L(u,v) = \frac{1}{2i}\bigl[\langle \kappa_{\varrho}u,v \rangle_H - \langle u,\kappa_{\varrho}v \rangle_H\bigr] \in C_0(\R_+)$ for all $u,v \in H$, and so $\langle \kappa_{\varrho}u,v \rangle_H = T(u,v) + iL(u,v) \in C_0(\R_+)$ for all $u,v \in H$.} $u,v \in H$, then $\spec(\hat{\g}) \subset \{\sigma \in \C;\; -m < \Im(\sigma) < 0\}$.
\end{proposition}
\begin{proof}
Let $\pi_{\min} \in \L(\Dom_{\max})$ be the $A$-orthogonal projection onto $\Dom_{\min}$, and let $\pi_{\max} = 1 - \pi_{\min}$ be the $A$-orthogonal projection onto $\E_{\max} = \Dom_{\min}^{\perp}$. For $\hat{u} = u + \Dom_{\min} \in \hat{\E}_{\max}$ we have
$$
\|\hat{\kappa}_{\varrho}\hat{u}\|_{\hat{\E}_{\max}} = \|\pi_{\max}\kappa_{\varrho}\pi_{\max}u\|_A \leq \|\kappa_{\varrho}\|_{\L(\Dom_{\max})}\|\pi_{\max} u\|_{A} \leq \max\{1,\varrho^m\}\|\hat{u}\|_{\hat{\E}_{\max}}
$$
by Lemma~\ref{Amaxmininv}. Thus $\|\hat{\kappa}_{\varrho}\| \leq \max\{1,\varrho^m\}$, which implies the first claim regarding the eigenvalues of $\hat{\g}$.

We next prove the second claim regarding the absence of eigenvalues of $\hat{\g}$ on the lines $\Im(\sigma) = 0$ and $\Im(\sigma) = -m$ under the stated additional assumption on $\kappa_{\varrho}$. To this end, observe that the unitarity of $\kappa_{\varrho}$ implies that in fact both
$$
\lim\limits_{\varrho \to 0} \langle \kappa_{\varrho}u,v \rangle_H = \lim\limits_{\varrho \to \infty} \langle \kappa_{\varrho}u,v \rangle_H = 0
$$
for all $u,v \in H$. Moreover, for $u,v \in \Dom_{\max}$ we have
\begin{align*}
\lim\limits_{\varrho \to 0} \langle \kappa_{\varrho}u,v \rangle_A &= \lim\limits_{\varrho \to 0} \Bigl[\langle \kappa_{\varrho}u,v \rangle_H + \langle A_{\max}\kappa_{\varrho}u,A_{\max}v \rangle_H\Bigr] \\
&= \lim\limits_{\varrho \to 0} \Bigl[\langle \kappa_{\varrho}u,v \rangle_H + \varrho^m\langle \kappa_{\varrho}A_{\max}u,A_{\max}v \rangle_H\Bigr] = 0, \\
\lim\limits_{\varrho \to \infty} \varrho^{-m}\langle \kappa_{\varrho}u,v \rangle_A &= \lim\limits_{\varrho \to \infty} \Bigl[\varrho^{-m}\langle \kappa_{\varrho}u,v \rangle_H + \langle \kappa_{\varrho}A_{\max}u,A_{\max}v \rangle_H\Bigr] = 0.
\end{align*}
For $\hat{u} = u + \Dom_{\min}$ and $\hat{v} = v + \Dom_{\min}$ in $\hat{\E}_{\max}$ we have
$$
\langle \hat{\kappa}_{\varrho}\hat{u},\hat{v} \rangle_{\hat{\E}_{\max}} = \langle \pi_{\max}\kappa_{\varrho}\pi_{\max}u,\pi_{\max} v \rangle_A = \langle \kappa_{\varrho}\pi_{\max}u,\pi_{\max}v \rangle_A,
$$
and thus
$$
\lim\limits_{\varrho \to 0} \langle \hat{\kappa}_{\varrho}\hat{u},\hat{v} \rangle_{\hat{\E}_{\max}} = 0 \textup{ and }
\lim\limits_{\varrho \to \infty}\varrho^{-m}\langle \hat{\kappa}_{\varrho}\hat{u},\hat{v} \rangle_{\hat{\E}_{\max}} = 0.
$$
Consequently, if $\hat{\g}\hat{u} = \sigma\hat{u}$ with $\Im(\sigma) = 0$, then $\hat{\kappa}_{\varrho}\hat{u} = \varrho^{i\sigma}\hat{u}$ and thus
$$
\|\hat{u}\|_{\hat{\E}_{\max}}^2 = \bigl|\langle \hat{\kappa}_{\varrho}\hat{u},\hat{u} \rangle_{\hat{\E}_{\max}}\bigr| \to 0 \textup{ as } \varrho \to 0.
$$
Likewise, if $\hat{\g}\hat{u} = \sigma\hat{u}$ with $\Im(\sigma) = -m$, then $\hat{\kappa}_{\varrho}\hat{u} = \varrho^{i\sigma}\hat{u}$ and
$$
\|\hat{u}\|_{\hat{\E}_{\max}}^2 = \bigl|\varrho^{-m}\langle \hat{\kappa}_{\varrho}\hat{u},\hat{u} \rangle_{\hat{\E}_{\max}}\bigr| \to 0 \textup{ as } \varrho \to \infty.
$$
\end{proof}


\section{A Green formula for the adjoint pairing}\label{GreenAbstract}

\noindent
We next consider the adjoint pairing
\begin{align*}
[\cdot,\cdot]_A &: \Dom_{\max} \times \Dom_{\max} \to \C, \\
[u,v]_A &= \frac{1}{i}\Bigl[\langle A_{\max}u,v \rangle_H - \langle u,A_{\max}v \rangle_H\Bigr].
\end{align*}
This is a Hermitian sesquilinear form (the extra $\frac{1}{i}$-term renders it Hermitian rather than skew-Hermitian), and $[u,v]_A = 0$ for all $v \in \Dom_{\max}$ if and only if $u \in \Dom_{\min}$. Thus $[\cdot,\cdot]_A$ induces a nondegenerate Hermitian sesquilinear form
\begin{gather*}
[\cdot,\cdot]_{\hat{\E}_{\max}} : \hat{\E}_{\max} \times \hat{\E}_{\max} \to \C, \\
[\hat{u},\hat{v}]_{\hat{\E}_{\max}} = [u,v]_A
\end{gather*}
for $\hat{u} = u + \Dom_{\min}$ and $\hat{v} = v + \Dom_{\min}$.
Every domain $\Dom_{\min} \subset \Dom \subset \Dom_{\max}$ is determined by its projection
\begin{equation}\label{DProjection}
\hat{\E}_{\Dom} = \{\hat{u} = u + \Dom_{\min};\; u \in \Dom\} \subset \hat{\E}_{\max}.
\end{equation}
We have $\bigl(A_{\Dom}\bigr)^* = A_{\Dom^{[\perp]}}$, where
$$
\hat{\E}_{\Dom^{[\perp]}} = \hat{\E}_{\Dom}^{[\perp]} = \{\hat{v} \in \hat{\E}_{\max};\; [\hat{u},\hat{v}]_{\hat{\E}_{\max}} = 0 \;\forall \hat{u} \in \hat{\E}_{\Dom}\}.
$$
In particular, $A_{\Dom} = A_{\Dom}^*$ is selfadjoint if and only if
$$
\hat{\E}_{\Dom}^{[\perp]} = \hat{\E}_{\Dom}  \subset \bigl(\hat{\E}_{\max},[\cdot,\cdot]_{\hat{\E}_{\max}}\bigr)
$$
is a Lagrangian subspace.

\begin{remark}
We will use several results about finite-dimensional complex vector spaces $(V,[\cdot,\cdot])$ equipped with \emph{indefinite} Hermitian sesquilinear forms $[\cdot,\cdot] : V \times V \to \C$ below. We refer to \cite{Bognar,EverittMarkus,GohbergLancasterRodman} as general references. The invariants of $(V,[\cdot,\cdot])$ are $(m_0,m_+,m_-) \in \N_0^3$, where
\begin{align*}
m_0 &= \dim \{v \in V;\; [v,w]=0 \; \forall w \in V\}, \\
m_+ &= \max\bigl[\{\dim U_+;\; U_+ \subset V,\; [u,u] > 0 \; \forall 0 \neq u \in U_+\}\cup\{0\}\bigr], \\
m_- &= \max\bigl[\{\dim U_-;\; U_- \subset V,\; [u,u] < 0 \; \forall 0 \neq u \in U_-\}\cup\{0\}\bigr].
\end{align*}
We have $m_0 + m_+ + m_- = \dim V$. The \emph{signature} of $(V,[\cdot,\cdot])$ is
$$
\sig(V,[\cdot,\cdot]) = m_+ - m_-.
$$
For any basis $\{v_1,\ldots,v_n\}$ of $V$ consider the matrix $G = \bigl([v_i,v_j]\bigr)_{i,j=1}^n \in M_n(\C)$. Then $G^* = G$, and $m_0$ is the multiplicity of the eigenvalue $0$ of $G$, while $m_{\pm}$ is the total number (including multiplicities) of the positive and negative eigenvalues of $G$, respectively. $(V,[\cdot,\cdot])$ is nondegenerate if $m_0 = 0$.

A subspace $U \subset (V,[\cdot,\cdot])$ is isotropic if $U \subset U^{[\perp]}$, where
$$
U^{[\perp]} = \{v \in V;\; [v,u] = 0 \; \forall u \in U\},
$$
and Lagrangian if $U = U^{[\perp]}$.
If $(V,[\cdot,\cdot])$ is nondegenerate then Lagrangian subspaces $U$ exist if and only if $\sig(V,[\cdot,\cdot]) = 0$. In particular, $V$ must then be even-dimensional and $\dim U = \frac{\dim V}{2}$.
\end{remark}

\begin{lemma}\label{GroupInvariance}
We have
$$
[\hat{\kappa}_{\varrho}\hat{u},\hat{\kappa}_{\varrho}\hat{v}]_{\hat{\E}_{\max}} = \varrho^{m}[\hat{u},\hat{v}]_{\hat{\E}_{\max}}, \; \varrho > 0,
$$
and
$$
[(\hat{\g}+i\tfrac{m}{2})\hat{u},\hat{v}]_{\hat{\E}_{\max}} = [\hat{u},(\hat{\g}+i\tfrac{m}{2})\hat{v}]_{\hat{\E}_{\max}}.
$$
\end{lemma}
\begin{proof}
We have
\begin{align*}
i[\hat{\kappa}_{\varrho}\hat{u},\hat{\kappa}_{\varrho}\hat{v}]_{\hat{\E}_{\max}} &=
\langle A_{\max}\kappa_{\varrho}u,\kappa_{\varrho}v\rangle - \langle \kappa_{\varrho}u, A_{\max}\kappa_{\varrho}v\rangle \\
&= \varrho^m\bigl[
\langle \kappa_{\varrho} A_{\max}u,\kappa_{\varrho}v\rangle - \langle \kappa_{\varrho}u, \kappa_{\varrho}A_{\max}v\rangle\bigr] \\
&= \varrho^m\bigl[\langle A_{\max}u,v\rangle - \langle u, A_{\max}v\rangle\bigr] \\
&= i\varrho^m[\hat{u},\hat{v}]_{\hat{\E}_{\max}},
\end{align*}
proving the first claim. Consequently,
$$
[\hat{u},\hat{v}]_{\hat{\E}_{\max}} = [\varrho^{-m/2}\hat{\kappa}_{\varrho}\hat{u},\varrho^{-m/2}\hat{\kappa}_{\varrho}\hat{v}]_{\hat{\E}_{\max}}, \; \varrho > 0,
$$
and thus
\begin{align*}
0 &= (\varrho D_{\varrho})[\hat{u},\hat{v}]_{\hat{\E}_{\max}} \\
&= [(\hat{\g}+i\tfrac{m}{2})\varrho^{-m/2}\hat{\kappa}_{\varrho}\hat{u},\varrho^{-m/2}\hat{\kappa}_{\varrho}\hat{v}]_{\hat{\E}_{\max}} -
[\varrho^{-m/2}\hat{\kappa}_{\varrho}\hat{u},(\hat{\g}+i\tfrac{m}{2})\varrho^{-m/2}\hat{\kappa}_{\varrho}\hat{v}]_{\hat{\E}_{\max}}.
\end{align*}
Evaluation at $\varrho=1$ proves the second claim.
\end{proof}

\noindent
By Lemma~\ref{GroupInvariance}, 
$$
\hat{\h} = \hat{\g}+i\tfrac{m}{2} : \hat{\E}_{\max} \to \hat{\E}_{\max}
$$
is selfadjoint with respect to the (indefinite) inner product $[\cdot,\cdot]_{\hat{\E}_{\max}}$ on $\hat{\E}_{\max}$, and consequently the triple $\bigl(\hat{\E}_{\max},[\cdot,\cdot]_{\hat{\E}_{\max}},\hat{\h}\bigr)$ has a canonical form (see \cite[Theorem~5.1.1]{GohbergLancasterRodman}) that furnishes a complete description of the adjoint pairing in terms of the generalized eigenspaces of $\hat{\g}$ as stated below. We will refer to these statements as the \emph{Canonical Form Theorem} in the sequel, and it plays the role of an abstract Green formula for the adjoint pairing:

\begin{enumerate}
\item The eigenvalues of $\hat{\g}$ are symmetric about the line $\Im(\sigma)=-\frac{m}{2}$. Let
\begin{equation}\label{Reflection}
\C \ni \sigma \mapsto \sigma^{\star} = \overline{\sigma} - im \in \C
\end{equation}
denote the reflection about the line $\Im(\sigma)=-\frac{m}{2}$. Then
$$
\spec(\hat{\g}) \ni \sigma \mapsto \sigma^{\star} \in \spec(\hat{\g}),
$$
see Figure~\ref{CriticalStrip}.

\begin{figure}
\begin{center}
\begin{tikzpicture}\scriptsize
\fill[color=lightgray,opacity=0.5] (-3,0.5) -- (3,0.5) -- (3,2.5) -- (-3,2.5) -- (-3,0.5);
\draw[->] (0,0) -- (0,3.5) node[right] {$i\R$};
\draw[thick,->] (-3,2.5) -- (3,2.5) node[right] {$\R$};
\draw[thick] (-3,0.5) -- (3,0.5) node[right] {$\Im(\sigma) = -m$};
\draw[dashed] (-3,1.5) -- (3,1.5) node[right] {$\Im(\sigma) = -\frac{m}{2}$};
\fill (-2,2.1) circle (1.5pt) node[right] {$\sigma^\star$};
\fill (-2,0.9) circle (1.5pt) node[right] {$\sigma$};
\draw[dotted] (-2,0.9) -- (-2,2.1);
\fill (-1,1.8) circle (1.5pt);
\fill (-1,1.2) circle (1.5pt);
\fill (-.5,1.9) circle (1.5pt);
\fill (-.5,1.1) circle (1.5pt);
\fill (.5,1.9) circle (1.5pt);
\fill (.5,1.1) circle (1.5pt);
\fill (.5,2.2) circle (1.5pt);
\fill (.5,0.8) circle (1.5pt);
\fill (1.2,2.3) circle (1.5pt);
\fill (1.2,0.7) circle (1.5pt);
\fill (2.1,1.7) circle (1.5pt);
\fill (2.1,1.3) circle (1.5pt);
\fill (0.8,1.5) circle (1.5pt);
\fill (1.6,1.5) circle (1.5pt);
\end{tikzpicture}
\end{center}
\caption{Eigenvalues of $\hat{\g}$}\label{CriticalStrip}
\end{figure}
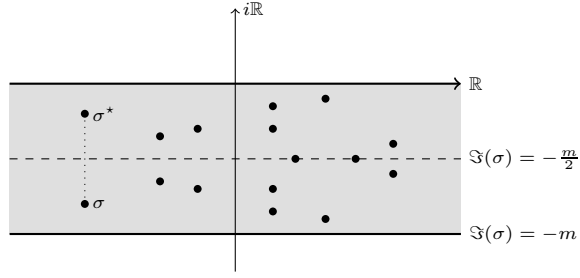

\item For $\sigma \in \spec(\hat{\g})$ let $\hat{\E}_{\sigma}$ be the generalized eigenspace for the eigenvalue $\sigma$. Then
$$
\hat{\E}_{\max} = \bigoplus\limits_{\substack{\sigma\in\spec(\hat{\g}) \\ \Im(\sigma)=-\frac{m}{2}}}\hat{\E}_{\sigma} \; \oplus 
\bigoplus\limits_{\substack{\sigma\in\spec(\hat{\g}) \\ \Im(\sigma)<-\frac{m}{2}}}\bigl[\hat{\E}_{\sigma}\oplus\hat{\E}_{\sigma^\star}\bigr],
$$
and this direct sum is orthogonal with respect to $[\cdot,\cdot]_{\hat{\E}_{\max}}$. More precisely, if $\hat{u} \in \hat{\E}_{\sigma_0}$ and $\hat{v} \in \hat{\E}_{\sigma_1}$ with $\sigma_1 \neq \sigma_0^{\star}$ then $[\hat{u},\hat{v}]_{\hat{\E}_{\max}} = 0$, and
$$
[\cdot,\cdot]_{\hat{\E}_{\max}} : \hat{\E}_{\sigma_0} \times \hat{\E}_{\sigma_0^{\star}} \to \C
$$
is nondegenerate.
\item For every $\sigma \in \spec(\hat{\g})$ with $\Im(\sigma) < -\frac{m}{2}$ there exist decompositions
$$
\hat{\E}_{\sigma} = \bigoplus\limits_{j=1}^{m_{\sigma}}\hat{\E}_{\sigma,j} \textup{ and }
\hat{\E}_{\sigma^\star} = \bigoplus\limits_{j=1}^{m_{\sigma}}\hat{\E}_{\sigma^\star,j}
$$
such that for every $\hat{u}_j \in \hat{\E}_{\sigma,j}$ and $\hat{v}_k \in \hat{\E}_{\sigma^{\star},k}$ we have $[\hat{u}_j,\hat{v}_k]_{\hat{\E}_{\max}} = 0$ whenever $j \neq k$.
Furthermore, for each $j \in \{1,\ldots,m_{\sigma}\}$, both spaces $\hat{\E}_{\sigma,j}$ and $\hat{\E}_{\sigma^{\star},j}$ are invariant under $\hat{\g}$, and there exist ordered bases $\{\hat{u}_1,\ldots,\hat{u}_{n_j}\}$ of $\hat{\E}_{\sigma,j}$ and $\{\hat{v}_1,\ldots,\hat{v}_{n_j}\}$ of $\hat{\E}_{\sigma^{\star},j}$ such that the matrix representation of $\hat{\g}$ in these bases is given by a single Jordan block of size $n_j\times n_j$ with eigenvalue $\sigma$ and $\sigma^{\star}$, respectively, and we have\footnote{The matrix $\begin{pmatrix} \delta_{\mu,N+1-\nu} \end{pmatrix}_{\mu,\nu=1}^N$ is called SIP matrix in \cite{GohbergLancasterRodman}, where the acronym stands for standard involutary permutation.}
$$
[\hat{u}_{\mu},\hat{v}_{\nu}]_{\hat{\E}_{\max}} = \delta_{\mu,n_j+1-\nu}
= \begin{cases}
1 & \mu+\nu=n_j+1, \\
0 & \textup{otherwise}.
\end{cases}
$$
\item For every $\sigma \in \spec(\hat{\g})$ with $\Im(\sigma) = -\frac{m}{2}$ there exists a decomposition
$$
\hat{\E}_{\sigma} = \bigoplus\limits_{j=1}^{m_{\sigma}}\hat{\E}_{\sigma,j}
$$
such that for every $\hat{u}_j \in \hat{\E}_{\sigma,j}$ and $\hat{v}_k \in \hat{\E}_{\sigma,k}$ we have $[\hat{u}_j,\hat{v}_k]_{\hat{\E}_{\max}} = 0$ whenever $j \neq k$.
Furthermore, for each $j \in \{1,\ldots,m_{\sigma}\}$, $\hat{\E}_{\sigma,j}$ is invariant under $\hat{\g}$, and there exists an ordered basis $\{\hat{u}_1,\ldots,\hat{u}_{n_j}\}$ of $\hat{\E}_{\sigma,j}$ such that the matrix representation of $\hat{\g}$ is given by a single Jordan block of size $n_j\times n_j$ with eigenvalue $\sigma$, and we either have
\begin{equation}\label{PlusEins}
[\hat{u}_{\mu},\hat{u}_{\nu}]_{\hat{\E}_{\max}} = \delta_{\mu,n_j+1-\nu}
= \begin{cases}
1 & \mu+\nu=n_j+1, \\
0 & \textup{otherwise},
\end{cases}
\end{equation}
or
\begin{equation}\label{MinusEins}
[\hat{u}_{\mu},\hat{u}_{\nu}]_{\hat{\E}_{\max}} = -\delta_{\mu,n_j+1-\nu}
= \begin{cases}
-1 & \mu+\nu=n_j+1, \\
0 & \textup{otherwise}.
\end{cases} \\
\end{equation}
The resulting collection of signs $+1$, associated to the canonical form \eqref{PlusEins}, or $-1$, associated to the canonical form \eqref{MinusEins}, for $\hat{\E}_{\sigma}$ is unique in the sense that different decompositions of $\hat{\E}_{\sigma}$ into subspaces $\hat{\E}_{\sigma,j}$ as above that result in Jordan blocks for $\hat{\g}$, are orthogonal with respect to $[\cdot,\cdot]_{\hat{\E}_{\max}}$, and furnish the canonical forms \eqref{PlusEins} or \eqref{MinusEins} for $[\cdot,\cdot]_{\hat{\E}_{\max}}$ associated with each Jordan block, yield the same total collection of plus and minus signs for Jordan blocks of equal sizes. Following \cite{GohbergLancasterRodman}, this collection of signs is referred to as the \emph{sign characteristic} of $\hat{\g}$ with respect to $[\cdot,\cdot]_{\hat{\E}_{\max}}$. Note that the sign characteristic is associated to eigenvalues with $\Im(\sigma) = -\frac{m}{2}$ only.
\end{enumerate}

\noindent
We further elaborate on the sign characteristic for eigenvalues located on the line $\Im(\sigma)=-\frac{m}{2}$ and proceed to describe it invariantly: Let $n_{\sigma} \in \N$ be such that
$$
\hat{\E}_{\sigma} = \ker\bigl(\hat{\g}-\sigma\bigr)^{n_{\sigma}} \supsetneq \ker\bigl(\hat{\g}-\sigma\bigr)^{n_{\sigma}-1} \supsetneq \ldots \supsetneq \ker\bigl(\hat{\g}-\sigma\bigr) \supsetneq \{0\}.
$$
For each $\ell \in \{1,\ldots,n_{\sigma}\}$ consider
\begin{equation}\label{EllForm}
\begin{gathered}{}
[\cdot,\cdot]_{\ell} : \ker\bigl(\hat{\g}-\sigma\bigr)^{\ell}\times\ker\bigl(\hat{\g}-\sigma\bigr)^{\ell} \to \C, \\
[\hat{u},\hat{v}]_{\ell} = [(\hat{\g}-\sigma)^{\ell-1}\hat{u},\hat{v}]_{\hat{\E}_{\max}}.
\end{gathered}
\end{equation}
This is a Hermitian sesquilinear form. Let $(m_0,m_+,m_-)$ be its invariants.

\begin{proposition}\label{SignCharacteristic}
We have
$$
m_0 = \dim\ker\bigl(\hat{\g}-\sigma\bigr)^{\ell+1} - \dim\ker\bigl(\hat{\g}-\sigma\bigr)^{\ell} + \dim\ker\bigl(\hat{\g}-\sigma\bigr)^{\ell-1},
$$
and
$$
m_+ + m_- = \textup{Number of Jordan blocks of size $\ell\times\ell$ for the eigenvalue $\sigma$.}
$$
The sign characteristic associated with the eigenvalue $\sigma$ and Jordan blocks of size $\ell\times\ell$ is then given by $(m_+,m_-)$, where $m_+$ is the number of $+1$ signs, and $m_-$ is the number of $-1$ signs.
\end{proposition}

\begin{proof}
The proposition can be deduced from \cite[Theorem~5.8.1]{GohbergLancasterRodman}, but we choose to give a direct proof here.

We first show the claim regarding $m_0$. To that end, we are going to prove that if $\hat{v} \in \ker\bigl(\hat{\g}-\sigma\bigr)^{\ell}$ such that $[\hat{u},\hat{v}]_{\ell}=0$ for all $\hat{u} \in \ker\bigl(\hat{\g}-\sigma\bigr)^{\ell}$ then $\hat{v} \in \ker\bigl(\hat{\g}-\sigma\bigr)^{\ell-1} + \ran\bigl(\hat{\g}-\sigma\bigr)\cap\ker\bigl(\hat{\g}-\sigma\bigr)^{\ell}$, and vice versa. Indeed, if
$$
[\hat{u},\hat{v}]_{\ell} = [(\hat{\g}-\sigma)^{\ell-1}\hat{u},\hat{v}]_{\hat{\E}_{\max}} = [\hat{u},(\hat{\g}-\sigma)^{\ell-1}\hat{v}]_{\hat{\E}_{\max}} = 0
$$
for all $\hat{u} \in \ker\bigl(\hat{\g}-\sigma\bigr)^{\ell}$, then $(\hat{\g}-\sigma)^{\ell-1}\hat{v} \in \bigl[\ker\bigl(\hat{\g}-\sigma\bigr)^{\ell}\bigr]^{[\perp]} = \ran\bigl(\hat{\g}-\sigma\bigr)^{\ell}$. Consequently, there exists $\hat{w} \in \hat{\E}_{\max}$ such that $(\hat{\g}-\sigma)^{\ell-1}\hat{v} = (\hat{\g}-\sigma)^{\ell}\hat{w}$, so $\hat{v} - (\hat{\g}-\sigma)\hat{w} = \hat{u}_0 \in \ker\bigl(\hat{\g}-\sigma\bigr)^{\ell-1}$. Thus
$$
\hat{v} = \hat{u}_0 + (\hat{\g}-\sigma)\hat{w} \in \ker\bigl(\hat{\g}-\sigma\bigr)^{\ell-1} + \ran\bigl(\hat{\g}-\sigma\bigr)\cap\ker\bigl(\hat{\g}-\sigma\bigr)^{\ell}.
$$
Conversely, if
$$
\hat{v} = \hat{u}_0 + (\hat{\g}-\sigma)\hat{w} \in \ker\bigl(\hat{\g}-\sigma\bigr)^{\ell-1} + \ran\bigl(\hat{\g}-\sigma\bigr)\cap\ker\bigl(\hat{\g}-\sigma\bigr)^{\ell},
$$
then $(\hat{\g}-\sigma)^{\ell-1}\hat{v} = (\hat{\g}-\sigma)^{\ell}\hat{w}$ and so
$$
[\hat{u},\hat{v}]_{\ell} = [\hat{u},(\hat{\g}-\sigma)^{\ell-1}\hat{v}]_{\hat{\E}_{\max}}
= [\hat{u},(\hat{\g}-\sigma)^{\ell}\hat{w}]_{\hat{\E}_{\max}} = [(\hat{\g}-\sigma)^{\ell}\hat{u},\hat{w}]_{\hat{\E}_{\max}} = 0
$$
for all $\hat{u} \in \ker\bigl(\hat{\g}-\sigma\bigr)^{\ell}$. Now
$$
\dim\bigl[\ran\bigl(\hat{\g}-\sigma\bigr)\cap\ker\bigl(\hat{\g}-\sigma\bigr)^{\ell}\bigr] = \dim\ker\bigl(\hat{\g}-\sigma\bigr)^{\ell+1} - \dim\ker(\hat{\g}-\sigma)
$$
as $\hat{\g}-\sigma : \ker\bigl(\hat{\g}-\sigma\bigr)^{\ell+1} \to \ran\bigl(\hat{\g}-\sigma\bigr)\cap\ker\bigl(\hat{\g}-\sigma\bigr)^{\ell}$ is surjective. So
\begin{align*}
m_0 &= \dim\bigl[\ker\bigl(\hat{\g}-\sigma\bigr)^{\ell-1} + \ran\bigl(\hat{\g}-\sigma\bigr)\cap\ker\bigl(\hat{\g}-\sigma\bigr)^{\ell}\bigr] \\
&= \dim\ker\bigl(\hat{\g}-\sigma\bigr)^{\ell-1} + \dim\bigl[\ran\bigl(\hat{\g}-\sigma\bigr)\cap\ker\bigl(\hat{\g}-\sigma\bigr)^{\ell}\bigr] \\
&\qquad\qquad - \dim\bigl[\ran\bigl(\hat{\g}-\sigma\bigr)\cap\ker\bigl(\hat{\g}-\sigma\bigr)^{\ell-1}\bigr] \\
&= \dim\ker\bigl(\hat{\g}-\sigma\bigr)^{\ell+1} - \dim\ker\bigl(\hat{\g}-\sigma\bigr)^{\ell} + \dim\ker\bigl(\hat{\g}-\sigma\bigr)^{\ell-1}.
\end{align*}
Furthermore,
\begin{align*}
m_+ + m_- &= \dim\ker\bigl(\hat{\g}-\sigma\bigr)^{\ell} - m_0 \\
&= \bigl[\dim\ker\bigl(\hat{\g}-\sigma\bigr)^{\ell} - \dim\ker\bigl(\hat{\g}-\sigma\bigr)^{\ell-1}\bigr] \\
&\qquad\qquad - \bigl[\dim\ker\bigl(\hat{\g}-\sigma\bigr)^{\ell+1} - \dim\ker\bigl(\hat{\g}-\sigma\bigr)^{\ell}\bigr] \\
&= \textup{Number of Jordan blocks of size $\ell\times\ell$ for the eigenvalue $\sigma$.}
\end{align*}
Now apply the Canonical Form Theorem from \cite[Theorem~5.1.1]{GohbergLancasterRodman} as previously described, and write
$$
\bigoplus\limits_{j=1}^{m_+ + m_-}\hat{\E}^{\ell}_{\sigma,j} \subset \ker\bigl(\hat{\g}-\sigma\bigr)^{\ell},
$$
where this direct sum is orthogonal with respect to $[\cdot,\cdot]_{\hat{\E}_{\max}}$, and the $\hat{\E}^{\ell}_{\sigma,j}$ are associated to $\ell\times\ell$ Jordan blocks for $\hat{\g}$ with eigenvalue $\sigma$. Moreover, let $\hat{u}_{j,1},\ldots,\hat{u}_{j,\ell}$ be a Jordan basis of $\hat{\E}^{\ell}_{\sigma,j}$ for each $j=1,\ldots,(m_+ + m_-)$ as described in the canonical form, such that for each $j$ we either have $[\hat{u}_{j,\ell},\hat{u}_{j,\ell}]_{\ell} = 1$ or $[\hat{u}_{j,\ell},\hat{u}_{j,\ell}]_{\ell} = -1$. Let
\begin{align*}
U_+ &= \Span\{\hat{u}_{j,\ell};\; [\hat{u}_{j,\ell},\hat{u}_{j,\ell}]_{\ell} = 1\}, \\
U_- &= \Span\{\hat{u}_{j,\ell};\; [\hat{u}_{j,\ell},\hat{u}_{j,\ell}]_{\ell} = -1\}.
\end{align*}
Because $[\hat{u}_{j,\ell},\hat{u}_{k,\ell}]_{\ell} = 0$ for $j \neq k$ we then have $[\hat{u},\hat{u}]_{\ell} > 0$ for all $0 \neq \hat{u} \in U_+$, and $[\hat{u},\hat{u}]_{\ell} < 0$ for all $0 \neq \hat{u} \in U_-$. Thus $m_{\pm} \geq \dim U_{\pm}$, but because $m_+ + m_- = \dim U_{+} + \dim U_{-}$ we have $m_{\pm} = \dim U_{\pm}$, proving the claim about the sign characteristic associated to Jordan blocks of size $\ell\times\ell$ for the eigenvalue $\sigma$.
\end{proof}

\noindent
In the sequel we will also write $m_{0}(\sigma,\ell)$ and $m_{\pm}(\sigma,\ell)$ for the invariants of the Hermitian sesquilinear form \eqref{EllForm} when the context warrants specific reference to the eigenvalue $\sigma$ and the size $\ell$. The Canonical Form Theorem implies the following algebraic signature formula and results about the existence of selfadjoint extensions of the operator $A$  (see also \cite[Corollary~5.2.1]{GohbergLancasterRodman}).

\begin{theorem}\label{SelfadjointExtensions}
\begin{enumerate}
\item The signature of $\bigl(\hat{\E}_{\max},[\cdot,\cdot]_{\hat{\E}_{\max}}\bigr)$ is given by
$$
\sig\bigl(\hat{\E}_{\max},[\cdot,\cdot]_{\hat{\E}_{\max}}\bigr) = \sum\limits_{\substack{\sigma \in \spec(\hat{\g}) \\ \Im(\sigma) = -\frac{m}{2}}}\sum\limits_{\ell \; \textup{odd}}\bigl(m_+(\sigma,\ell) - m_-(\sigma,\ell)\bigr).
$$
In particular, the operator $A$ admits selfadjoint extensions $A_{\Dom} = A_{\Dom}^*$ if and only if
$$
\sum\limits_{\substack{\sigma \in \spec(\hat{\g}) \\ \Im(\sigma) = -\frac{m}{2}}}\sum\limits_{\ell \; \textup{odd}}\bigl(m_+(\sigma,\ell) - m_-(\sigma,\ell)\bigr) = 0.
$$
\item The operator $A$ admits an invariant selfadjoint extensions if and only if
$$
\sum\limits_{\ell \; \textup{odd}}\bigl(m_+(\sigma,\ell) - m_-(\sigma,\ell)\bigr) = 0
$$
for every $\sigma \in \spec(\hat{\g})$ with $\Im(\sigma) = -\frac{m}{2}$.
\end{enumerate}
\end{theorem}
\begin{proof}
Let $(0,M_+,M_-)$ be the invariants of the form $[\cdot,\cdot]_{\hat{\E}_{\max}}$. We have
$$
\hat{\E}_{\max} = \bigoplus\limits_{\substack{\sigma\in\spec(\hat{\g}) \\ \Im(\sigma)=-\frac{m}{2}}}\hat{\E}_{\sigma} \; \oplus 
\bigoplus\limits_{\substack{\sigma\in\spec(\hat{\g}) \\ \Im(\sigma)<-\frac{m}{2}}}\bigl[\hat{\E}_{\sigma}\oplus\hat{\E}_{\sigma^\star}\bigr],
$$
where this direct sum is orthogonal with respect to $[\cdot,\cdot]_{\hat{\E}_{\max}}$, and the restriction of $[\cdot,\cdot]_{\hat{\E}_{\max}}$ to each of the direct summands is nondegenerate by the Canonical Form Theorem \cite[Theorem~5.1.1]{GohbergLancasterRodman}. Consequently,
$$
M_+ - M_- = \sum\limits_{\substack{\sigma \in \spec(\hat{\g}) \\ \Im(\sigma) = -\frac{m}{2}}}\bigl(M_+(\hat{\E}_{\sigma}) -
M_-(\hat{\E}_{\sigma})\bigr) +
\sum\limits_{\substack{\sigma \in \spec(\hat{\g}) \\ \Im(\sigma) < -\frac{m}{2}}}\bigl(M_+(\hat{\E}_{\sigma}\oplus\hat{\E}_{\sigma^\star}) - M_-(\hat{\E}_{\sigma}\oplus\hat{\E}_{\sigma^\star})\bigr),
$$
where $M_{\pm}(\hat{U})$ are the invariants of the restriction $[\cdot,\cdot]_{\hat{\E}_{\max}} : \hat{U} \times \hat{U} \to \C$ to the indicated subspace $\hat{U}$. For each $\sigma \in \spec(\hat{\g})$ with $\Im(\sigma) < -\frac{m}{2}$ the Canonical Form Theorem implies that $\hat{\E}_{\sigma} \subset \bigl(\hat{\E}_{\sigma}\oplus\hat{\E}_{\sigma^\star},[\cdot,\cdot]_{\hat{\E}_{\max}}\bigr)$ is Lagrangian, and thus
$$
M_+(\hat{\E}_{\sigma}\oplus\hat{\E}_{\sigma^\star}) - M_-(\hat{\E}_{\sigma}\oplus\hat{\E}_{\sigma^\star}) = 0.
$$
For $\sigma \in \spec(\hat{\g})$ with $\Im(\sigma) = -\frac{m}{2}$ we prove next that
$$
M_+(\hat{\E}_{\sigma}) - M_-(\hat{\E}_{\sigma}) = \sum\limits_{\ell \; \textup{odd}}\bigl(m_+(\sigma,\ell) - m_-(\sigma,\ell)\bigr).
$$
By the Canonical Form Theorem we have
$$
\hat{\E}_{\sigma} = \bigoplus\limits_{j=1}^{m_{\sigma}}\hat{\E}_{\sigma,j}
$$
which is orthogonal with respect to $[\cdot,\cdot]_{\hat{\E}_{\max}}$, the restriction of $[\cdot,\cdot]_{\hat{\E}_{\max}}$ to each of the direct summands is nondegenerate, and each $\hat{\E}_{\sigma,j}$ is associated to a Jordan block that either contributes a $+1$ or a $-1$ to the sign characteristic. We have
$$
M_+(\hat{\E}_{\sigma}) - M_-(\hat{\E}_{\sigma}) = \sum\limits_{j=1}^{m_{\sigma}}\bigl(M_+(\hat{\E}_{\sigma,j}) - M_-(\hat{\E}_{\sigma,j})\bigr).
$$
If $\dim\hat{\E}_{\sigma,j}$ is even and spanned by $\{\hat{u}_1,\ldots,\hat{u}_{2k}\}$ with either \eqref{PlusEins} or \eqref{MinusEins}, then
$$
\hat{U} = \Span\{\hat{u}_1,\ldots,\hat{u}_{k}\} \subset \bigl(\hat{\E}_{\sigma,j},[\cdot,\cdot]_{\hat{\E}_{\max}}\bigr)
$$
is Lagrangian and consequently $M_+(\hat{\E}_{\sigma,j}) - M_-(\hat{\E}_{\sigma,j}) = 0$. If $\dim\hat{\E}_{\sigma,j}$ is odd and spanned by $\{\hat{u}_1,\ldots,\hat{u}_{2k+1}\}$ with either \eqref{PlusEins} or \eqref{MinusEins}, then
$$
M_+(\hat{\E}_{\sigma,j}) - M_-(\hat{\E}_{\sigma,j}) = \begin{cases}
1  & \textup{in case of \eqref{PlusEins}}, \\
-1 & \textup{in case of \eqref{MinusEins}}.
\end{cases}
$$
This is because the matrix $\bigl(\delta_{\mu,2k+2-\nu}\bigr)_{\mu,\nu=1}^{2k+1}$ has $(k+1)$-times the eigenvalue $+1$ and $k$-times the eigenvalue $-1$. Consequently,
$$
M_+(\hat{\E}_{\sigma}) - M_-(\hat{\E}_{\sigma}) = \sum\limits_{\ell \; \textup{odd}}\bigl(m_+(\sigma,\ell) - m_-(\sigma,\ell)\bigr)
$$
as desired, proving the first claim.

We next prove the second claim regarding the existence of invariant selfadjoint extensions. Note that $A_{\Dom} = A_{\Dom}^*$ is an invariant selfadjoint extension if and only if $\hat{U} = \hat{\E}_{\Dom} \subset \bigl(\hat{\E}_{\max},[\cdot,\cdot]_{\hat{\E}_{\max}}\bigr)$ from \eqref{DProjection} is Lagrangian and invariant under $\hat{\g}$. We thus have to show that invariant Lagrangian subspaces exist if and only if
$$
\sum\limits_{\ell \; \textup{odd}}\bigl(m_+(\sigma,\ell) - m_-(\sigma,\ell)\bigr) = 0
$$
holds for every $\sigma \in \spec(\hat{\g})$ with $\Im(\sigma) = -\frac{m}{2}$.

Suppose first that $\hat{U} \subset \hat{\E}_{\max}$ is Lagrangian and invariant under $\hat{\g}$. Then
\begin{align*}
\hat{U} &= \bigoplus\limits_{\sigma\in\spec(\hat{\g})}\bigl[\hat{\E}_{\sigma}\cap \hat{U}\bigr] \\
&= \bigoplus\limits_{\substack{\sigma\in\spec(\hat{\g}) \\ \Im(\sigma)=-\frac{m}{2}}}\bigl[\hat{\E}_{\sigma}\cap\hat{U}\bigr] \; \oplus 
\bigoplus\limits_{\substack{\sigma\in\spec(\hat{\g}) \\ \Im(\sigma)<-\frac{m}{2}}}\bigl(\bigl[\hat{\E}_{\sigma}\cap\hat{U}\bigr]\oplus\bigl[\hat{\E}_{\sigma^\star}\cap\hat{U}\bigr]\bigr).
\end{align*}
Every subspace
\begin{align*}
\bigl[\hat{\E}_{\sigma}\cap \hat{U}\bigr] &\subset \bigl(\hat{\E}_{\sigma},[\cdot,\cdot]_{\hat{\E}_{\max}}\bigr), \quad \Im(\sigma) = -\frac{m}{2}, \\
\bigl[\hat{\E}_{\sigma}\cap\hat{U}\bigr]\oplus\bigl[\hat{\E}_{\sigma^\star}\cap\hat{U}\bigr] &\subset
\bigl(\hat{\E}_{\sigma}\oplus\hat{\E}_{\sigma^\star},[\cdot,\cdot]_{\hat{\E}_{\max}}\bigr), \quad \Im(\sigma) < -\frac{m}{2},
\end{align*}
is isotropic, and for dimensional reasons must be Lagrangian. In particular, for every $\sigma \in \spec(\hat{\g})$ with $\Im(\sigma) = -\frac{m}{2}$ we must have
$$
M_+(\hat{\E}_{\sigma}) - M_-(\hat{\E}_{\sigma}) = \sum\limits_{\ell \; \textup{odd}}\bigl(m_+(\sigma,\ell) - m_-(\sigma,\ell)\bigr) = 0.
$$
Conversely, assume that $M_+(\hat{\E}_{\sigma}) - M_-(\hat{\E}_{\sigma}) = 0$ for every $\sigma \in \spec(\hat{\g})$ with $\Im(\sigma) = -\frac{m}{2}$. We proceed to prove that each space admits a Lagrangian subspace $\hat{U}_{\sigma} \subset \hat{\E}_{\sigma}$ that is invariant under $\hat{\g}$. Then
$$
\hat{U} = \bigoplus\limits_{\substack{\sigma\in\spec(\hat{\g}) \\ \Im(\sigma)=-\frac{m}{2}}} \hat{U}_{\sigma} \; \oplus
\bigoplus\limits_{\substack{\sigma\in\spec(\hat{\g}) \\ \Im(\sigma)<-\frac{m}{2}}}\hat{\E}_{\sigma}
$$
is a Lagrangian subspace of $\bigl(\hat{\E}_{\max},[\cdot,\cdot]_{\hat{\E}_{\max}}\bigr)$ that is invariant under $\hat{\g}$ as desired. Write
$$
\hat{\E}_{\sigma} = \bigoplus\limits_{j=1}^{m_{\sigma}}\hat{\E}_{\sigma,j}
$$
according to the Canonical Form Theorem with mutually $[\cdot,\cdot]_{\hat{\E}_{\max}}$-orthogonal direct summands, where each $\hat{\E}_{\sigma,j}$ is associated to a single Jordan block, and pick bases for each $\hat{\E}_{\sigma,j}$ with either \eqref{PlusEins} or \eqref{MinusEins}. By assumption, the odd-dimensional blocks (if any) equally distribute among the signs $+1$ and $-1$. Now divide the odd-dimensional spaces $\hat{\E}_{\sigma,j}$ up in pairs of the form $(\hat{V},\hat{W})$, where $\hat{V}$ contributes $+1$ and $\hat{W}$ contributes $-1$, and set
$$
\hat{U}_{(\hat{V},\hat{W})} = \Span\Bigl\{\hat{v}_{\frac{\dim\hat{V}+1}{2}} + \hat{w}_{\frac{\dim\hat{W}+1}{2}}\Bigr\},
$$
where the two vectors $\hat{v}_{\frac{\dim\hat{V}+1}{2}}$ and $\hat{w}_{\frac{\dim\hat{W}+1}{2}}$ are the middle vectors in the Jordan bases for $\hat{V}$ and $\hat{W}$, respectively, and define
$$
\hat{U}_{\sigma} = \bigoplus\limits_{j=1}^{m_{\sigma}}\Bigl[\hat{\E}_{\sigma,j}\cap\ker\bigl(\hat{\g}-\sigma\bigr)^{\bigl\lfloor \frac{\dim \hat{\E}_{\sigma,j}}{2}\bigr\rfloor} \Bigr] \; \oplus \bigoplus\limits_{(\hat{V},\hat{W})}\hat{U}_{(\hat{V},\hat{W})}.
$$
Then $\hat{U}_{\sigma}$ has the desired properties, and the proof is complete.
\end{proof}


\section{Semibounded operators and the Friedrichs extension}\label{FriedrichsAbstract}

\noindent
Suppose that $A_{\min}$ is semibounded from below in the sense that there exists a constant $c \in \R$ such that
$$
\langle Au,u \rangle \geq c\langle u,u \rangle
$$
holds for all $u \in \Dom_{\min}$. By Lemma~\ref{Amaxmininv} and Lemma~\ref{BoundStationary} below we in fact must have $A_{\min} \geq 0$ which we are going to assume for the remainder of this section. Since we assume $A_{\min}$ to have finite deficiency indices, we note that every selfadjoint extension $A_{\Dom} = A_{\Dom}^*$ is semibounded from below, see \cite[Theorem~9.3.7]{BirmanSolomjak}.

\begin{lemma}\label{BoundStationary}
Let $A_{\Dom} : \Dom \subset H \to H$ be invariant and bounded from below, i.e.,
$$
\langle Au,u \rangle \geq c\langle u,u \rangle
$$
for all $u \in \Dom$ for some $c \in {\mathbb R}$. Then
$$
\inf\limits_{0 \neq u \in \Dom} \frac{\langle Au,u \rangle}{\langle u,u \rangle} = 0,
$$
i.e., $A_{\Dom}$ has lower bound $0$.
\end{lemma}
\begin{proof}
Let
$$
L = \inf\limits_{0 \neq u \in \Dom} \frac{\langle Au,u \rangle}{\langle u,u \rangle}.
$$
For $0 \neq u \in \Dom$ we have
$$
\langle Au,u \rangle = \varrho^m\langle \kappa_{\varrho}A_{\Dom}\kappa_{\varrho}^{-1}u,u\rangle = \varrho^m\langle A_{\Dom}\kappa_{\varrho}^{-1}u,\kappa_{\varrho}^{-1}u\rangle,
$$
and thus
$$
\frac{\langle A_{\Dom}u,u \rangle}{\langle u,u \rangle} = \varrho^m\frac{\langle A_{\Dom}\kappa_{\varrho}^{-1}u,\kappa_{\varrho}^{-1}u\rangle}{\langle \kappa_{\varrho}^{-1}u,\kappa_{\varrho}^{-1}u\rangle},
$$
using the fact that $\kappa_{\varrho}$ is unitary. Consequently, passing to the infimum over all $0 \neq u \in \Dom$ on both sides, we get $L = \varrho^m L$ for all $\varrho > 0$. This leaves only $L=0$ or $L=-\infty$, and since $L > -\infty$ by assumption we must have $L=0$.
\end{proof}

\begin{proposition}\label{FriedrichsStationary}
Let ${\mathcal H}_F \hookrightarrow H$ be the completion of $\Dom_{\min}$ with respect to the norm
$$
|u|_F^2 = \langle u,u \rangle + \langle Au,u \rangle, \; u \in \Dom_{\min}.
$$
Then $\kappa_{\varrho}$ restricts to a strongly continuous group action $\kappa_{\varrho} \in \L({\mathcal H}_F)$ with
$$
\|\kappa_{\varrho}\|_{\L({\mathcal H}_F)} \leq \max\{1,\varrho^{m/2}\}, \; \varrho > 0.
$$
In particular, the Friedrichs extension
\begin{align*}
A_F &= A^*\Big|_{\Dom_F} : \Dom_F \subset H \to H, \\
\Dom_F &= \Dom_{\max} \cap {\mathcal H}_F,
\end{align*}
is invariant.
\end{proposition}
\begin{proof}
We show that
$$
\kappa_{\varrho} : \bigl(\Dom_{\min},|\cdot|_F\bigr) \to \bigl(\Dom_{\min},|\cdot|_F\bigr)
$$
is continuous, and consequently extends by continuity to a bounded operator in $\L({\mathcal H}_F)$. Because ${\mathcal H}_F \hookrightarrow H$ and $\kappa_{\varrho} \in \L(H)$ this bounded extension is necessarily the restriction of $\kappa_{\varrho}$ to ${\mathcal H}_F$. Indeed, for $u \in \Dom_{\min}$ we have
\begin{align*}
|\kappa_{\varrho}u|_F^2 &= \langle \kappa_{\varrho}u,\kappa_{\varrho}u \rangle + \langle A\kappa_{\varrho}u,\kappa_{\varrho}u \rangle = \langle u,u \rangle + \varrho^m\langle\kappa_{\varrho}Au,\kappa_{\varrho}u\rangle \\
&= \langle u,u \rangle + \varrho^m\langle Au,u\rangle \leq \max\{1,\varrho^m\}\cdot |u|_F^2,
\end{align*}
proving the desired continuity as well as the asserted norm estimate. This implies that $A_F$ is invariant.

It remains to show the strong continuity of $\kappa_{\varrho}$ on the Hilbert space ${\mathcal H}_F$. Because $\kappa_{\varrho}$ is a group and $\{\kappa_{\varrho};\; \frac{1}{2} \leq \varrho \leq 2\} \subset \L({\mathcal H}_F)$ is bounded, it suffices to show that $\kappa_{\varrho}u \to u$ with respect to $|\cdot|_F$ as $\varrho \to 1$ just for $u \in \Dom_{\min}$ (see \cite[I.5.3]{EngelNagel}). Because $\kappa_{\varrho}u \to u$ with respect to $\|\cdot\|_H$ we thus have to prove that
$$
\langle A(\kappa_{\varrho}u - u),\kappa_{\varrho}u - u \rangle_H \to 0 \textup{ as } \varrho \to 1.
$$
But this follows because $A(\kappa_{\varrho}u - u) = \varrho^m\kappa_{\varrho}Au - Au \to 0$ as $\varrho \to 1$ with respect to $\|\cdot\|_H$.
\end{proof}

\noindent
The restriction $\kappa_{\varrho} : {\mathcal H}_F \to {\mathcal H}_F$ of the group action is generated by the part of the generator $\g$ in ${\mathcal H}_F$, i.e., the operator that acts like $\g$ with domain
$$
\{u \in \Dom(\g)\cap {\mathcal H}_F;\; \g u \in {\mathcal H}_F\},
$$
see \cite[II.2.3]{EngelNagel}. The norm estimate for $\kappa_{\varrho}$ in Proposition~\ref{FriedrichsStationary}, both for $1 \leq \varrho < \infty$ and for $0 < \varrho \leq 1$, in conjunction with the Hille-Yosida Generation Theorem \cite[II.3]{EngelNagel} now implies that for every $\sigma \in \C$ with $\Im(\sigma) \notin [-m/2,0]$ and every $v \in {\mathcal H}_F$ the unique solution $u \in \Dom(\g)$ to the equation $(\g - \sigma)u = v$ belongs to ${\mathcal H}_F$. This is in addition to the properties for $\g$ previously discussed in Remark~\ref{GeneratorResolvent}.

Our next goal is to give an explicit description of the domain of the Friedrichs extension $A_F$ in terms of the generalized eigenspaces of $\hat{\g}$. In this context, consider the following definition derived from \cite{GohbergLancasterRodman}.

\begin{definition}\label{SignCondition}
We say that $A$ satisfies the \emph{sign condition} if for every $\sigma \in \spec(\hat{\g})$ with $\Im(\sigma) = -\frac{m}{2}$ the following two conditions hold:
\begin{itemize}
\item $m_+(\sigma,\ell) = m_-(\sigma,\ell) = 0$ for all odd $\ell \in \N$. Thus $\hat{\g}$ does not have odd-sized Jordan blocks associated with the eigenvalue $\sigma$.
\item \emph{Either} $m_+(\sigma,\ell) = 0$ \emph{or} $m_-(\sigma,\ell) = 0$ for all even $\ell \in \N$. Thus the sign characteristic associated with the eigenvalue $\sigma$ is entirely negative or positive for all Jordan blocks of $\hat{\g}$.
\end{itemize}
Here $m_{\pm}(\sigma,\ell)$ are the invariants of the Hermitian sesquilinear form \eqref{EllForm}, see Proposition~\ref{SignCharacteristic}. Note that the sign condition is trivially fulfilled if $\spec(\hat{\g})\cap\{\sigma \in \C;\; \Im(\sigma)=-\frac{m}{2}\} = \emptyset$.
\end{definition}

We note that the sign condition holds for semibounded indicial operators considered in the later sections, see Theorem~\ref{SignConditionIndicial}.

\bigskip

\noindent
If the sign condition holds, \cite[Theorem~5.12.4]{GohbergLancasterRodman} implies that for each $\sigma \in \spec(\hat{\g})$ with $\Im(\sigma)=-\frac{m}{2}$ there exists a unique Langrangian subspace $\hat{\E}_{\sigma,\frac{1}{2}} \subset \bigl(\hat{\E}_{\sigma},[\cdot,\cdot]_{\hat{\E}_{\max}}\bigr)$ that is invariant under $\hat{\g}$.
More precisely, write
$$
\hat{\E}_{\sigma} = \bigoplus\limits_{j=1}^{m_{\sigma}}\hat{\E}_{\sigma,j}
$$
according to the Canonical Form Theorem with mutually $[\cdot,\cdot]_{\hat{\E}_{\max}}$-orthogonal direct summands, where each $\hat{\E}_{\sigma,j}$ is associated to a single Jordan block of size $(2n_j)\times(2n_j)$. Then
$$
\hat{\E}_{\sigma,\frac{1}{2}} = \bigoplus\limits_{j=1}^{m_{\sigma}}\bigl[\hat{\E}_{\sigma,j} \cap \ker(\hat{\g}-\sigma)^{n_j} \bigr],
$$
so $\hat{\E}_{\sigma,\frac{1}{2}}$ is spanned by the first halves of the Jordan basis elements associated to each Jordan block in the canonical form for the triple $\bigl(\hat{\E}_{\sigma},[\cdot,\cdot]_{\hat{\E}_{\max}},\hat{\g} - \sigma\bigr)$.

\begin{theorem}\label{FriedrichsLowerHalf}
Let $\hat{\E}_F = \{\hat{u} \in \hat{\E}_{\max}; \; \hat{u} = u + \Dom_{\min},\; u \in \Dom_F\}$. Then
$$
\hat{\E}_F = \bigoplus\limits_{\substack{\sigma\in\spec(\hat{\g}) \\ \Im(\sigma) = -\frac{m}{2}}}\bigl[\hat{\E}_{\sigma}\cap \hat{\E}_F\bigr] \oplus
\bigoplus\limits_{\substack{\sigma\in\spec(\hat{\g}) \\ \Im(\sigma) < -\frac{m}{2}}}\hat{\E}_{\sigma}.
$$
If $A$ satisfies the sign condition, then $\hat{\E}_{\sigma}\cap \hat{\E}_F = \hat{\E}_{\sigma,\frac{1}{2}}$ for all $\sigma \in \spec(\hat{\g})$ with $\Im(\sigma)=-\frac{m}{2}$.
\end{theorem}
\begin{proof}
By Proposition~\ref{FriedrichsStationary} the Friedrichs extension is invariant. This implies that $\hat{\E}_F$ is invariant under $\hat{\kappa}_{\varrho}$, and consequently $\hat{\E}_F$ is also invariant under the generator $\hat{\g}$. Hence
$$
\hat{\E}_F = \bigoplus\limits_{\sigma\in\spec(\hat{\g})}\bigl[\hat{\E}_{\sigma}\cap \hat{\E}_F\bigr].
$$
We will next prove that if $\Im(\sigma) < -\frac{m}{2}$, and $\hat{u} \in \hat{\E}_{\max}$ with $\bigl(\hat{\g} - \sigma\bigr)\hat{u} \in \hat{\E}_F$, then $\hat{u} \in \hat{\E}_F$. Indeed, by assumption there exists $u \in \Dom_{\max}\cap\Dom(\g)$ with $\hat{u} = u + \Dom_{\min}$, and $v \in \Dom_F$ such that $\bigl(\g - \sigma\bigr)u = v$. Because $v \in {\mathcal H}_F$ we have $u \in {\mathcal H}_F$, and thus $u \in {\mathcal H}_F\cap\Dom_{\max} = \Dom_F$, so $\hat{u} \in \hat{\E}_F$ as stated. Consequently, if $\bigl(\hat{\g}-\sigma\bigr)^k\hat{u} = 0$ for some $k \in \N$, we have $\hat{u} \in \hat{\E}_F$, showing that $\hat{\E}_{\sigma} \subset \hat{\E}_F$ for $\Im(\sigma) < -\frac{m}{2}$. Because the adjoint pairing
$$
[\cdot,\cdot]_{\hat{\E}_{\max}} : \hat{\E}_{\sigma}\times\hat{\E}_{\sigma^{\star}} \to \C
$$
is nondegenerate and $\hat{\E}_F$ is Lagrangian we must necessarily have $\hat{\E}_{\sigma^{\star}}\cap\hat{\E}_F = \{0\}$ for $\Im(\sigma) < -\frac{m}{2}$.

Finally, note that $\hat{\E}_{\sigma}\cap\hat{\E}_F \subset \bigl(\hat{\E}_{\sigma},[\cdot,\cdot]_{\hat{\E}_{\max}}\bigr)$ is Lagrangian and invariant under $\hat{\g}$ for $\Im(\sigma)=-\frac{m}{2}$. Consequently, if $A$ satisfies the sign condition, we must necessarily have $\hat{\E}_{\sigma}\cap \hat{\E}_F = \hat{\E}_{\sigma,\frac{1}{2}}$. The theorem is proved.
\end{proof}


\section{The Krein extension and the order relation for invariant selfadjoint extensions}\label{KreinAbstract}

\noindent
We continue to assume that $A_{\min} \geq 0$. The Friedrichs extension $A_F$ and the Krein extension $A_K$ are distinguished in the sense that $A_K \leq A_{\Dom} \leq A_F$ for all nonnegative selfadjoint extensions $A_{\Dom} = A_{\Dom}^* \geq 0$. Recall that for two selfadjoint operators $T_j = T_j^* \geq 0$ the order relation $T_1 \leq T_2$ holds if and only if $\Dom(T_2^{\frac{1}{2}}) \subseteq \Dom(T_1^{\frac{1}{2}})$ and $\|T_1^{\frac{1}{2}}u\|_H \leq \|T_2^{\frac{1}{2}}u\|_H$ for $u \in \Dom(T_2^{\frac{1}{2}})$. The latter is an inequality for the quadratic forms associated with the $T_j$, while the domain of the nonnegative square root in each case coincides with the domain of the quadratic form.
We refer to \cite{AlonsoSimon,Ashbaughetal,AndoNishio} for information on the Krein-von Neumann extension (but note that $A_{\min}$ is not strictly positive by Lemma~\ref{BoundStationary}, which is compensated for by the invariance of $A$ under scaling in our investigation).

Our goal in this section is to prove the following two theorems.

\begin{theorem}\label{KreinExtension}
Let $\hat{\E}_K = \{\hat{u} \in \hat{\E}_{\max}; \; \hat{u} = u + \Dom_{\min},\; u \in \Dom_K = \Dom(A_K)\}$. Then
$$
\hat{\E}_K = \bigoplus\limits_{\substack{\sigma\in\spec(\hat{\g}) \\ \Im(\sigma) = -\frac{m}{2}}}\bigl[\hat{\E}_{\sigma}\cap \hat{\E}_K\bigr] \oplus
\bigoplus\limits_{\substack{\sigma\in\spec(\hat{\g}) \\ \Im(\sigma) > -\frac{m}{2}}}\hat{\E}_{\sigma}.
$$
If $A$ satisfies the sign condition, then $\hat{\E}_{\sigma}\cap \hat{\E}_K = \hat{\E}_{\sigma,\frac{1}{2}}$ for all $\sigma \in \spec(\hat{\g})$ with $\Im(\sigma)=-\frac{m}{2}$.
\end{theorem}

\begin{theorem}\label{InvariantOrder}
Suppose $A$ satisfies the sign condition, and let
$$
\hat{\E}_- = \bigoplus\limits_{\substack{\sigma\in\spec(\hat{\g}) \\ \Im(\sigma) < -\frac{m}{2}}}\hat{\E}_{\sigma}.
$$
\begin{enumerate}
\item For every subspace $\hat{U} \subset \hat{\E}_-$ that is invariant under $\hat{\kappa}_{\varrho}$ there exists a unique selfadjoint invariant extension $A_{\Dom} : \Dom \subset H \to H$ of $A$ such that $\hat{\E}_{\Dom}\cap\hat{\E}_- = \hat{U}$.
\item Let $A_{\Dom_j} = A_{\Dom_j}^*$, $j = 1,2$, be invariant selfadjoint extensions. Then $A_{\Dom_1} \leq A_{\Dom_2}$ if and only if $\hat{\E}_{\Dom_1}\cap\hat{\E}_- \subseteq \hat{\E}_{\Dom_2}\cap\hat{\E}_-$, or equivalently if and only if $\Dom_1 \cap \Dom_F \subseteq \Dom_2 \cap \Dom_F$.
\end{enumerate}
\end{theorem}

We note that the first part of Theorem~\ref{InvariantOrder} is just \cite[Theorem~5.12.4]{GohbergLancasterRodman} when applied to the present context.

Given $A_{\Dom} = A_{\Dom}^* \geq 0$ we let ${\mathcal H}_{\Dom} \hookrightarrow H$ be the domain of the quadratic form associated with $A_{\Dom}$. We have ${\mathcal H}_{\Dom} = \Dom(A_{\Dom}^{\frac{1}{2}})$, which equivalently can be described as the completion of $\Dom$ with respect to the norm
$$
|u|_{\Dom}^2 = \langle u,u \rangle_H + \langle A_{\Dom}u, u\rangle_H, \; u \in \Dom.
$$
Because $|u|_{\Dom} = |u|_F$ for $u \in \Dom_{\min}$ the embedding $\bigl({\mathcal H}_F,|\cdot|_F\bigr) \hookrightarrow \bigl({\mathcal H}_{\Dom},|\cdot|_{\Dom}\bigr)$ is an isometry. Moreover, the codimension of ${\mathcal H}_F$ in ${\mathcal H}_{\Dom}$ is finite, where more precisely $\dim{\mathcal H}_{\Dom}/{\mathcal H}_F \leq \frac{1}{2}\dim\hat{\E}_{\max}$, see \cite[Theorem~10.3.7]{BirmanSolomjak}. We write ${\mathcal H}_K$ for the domain of the quadratic form associated with the Krein extension $A_K = A_K^*$.

\begin{lemma}\label{HDRepresentation}
Let $A_{\Dom} = A_{\Dom}^* \geq 0$. Then ${\mathcal H}_{\Dom} = \Dom + {\mathcal H}_F$, and this sum is direct modulo $\Dom\cap \Dom_F$.
\end{lemma}
\begin{proof}
Because $\Dom \subset \bigl({\mathcal H}_{\Dom},|\cdot|_{\Dom}\bigr)$ is dense we have that $\bigl(\Dom + {\mathcal H}_F\bigr)/{\mathcal H}_F \subset {\mathcal H}_{\Dom}/{\mathcal H}_F$ is dense, and by finite dimensionality we must have $\bigl(\Dom + {\mathcal H}_F\bigr)/{\mathcal H}_F = {\mathcal H}_{\Dom}/{\mathcal H}_F$. This shows that ${\mathcal H}_{\Dom} = \Dom + {\mathcal H}_F$, and because $\Dom_{\max}\cap{\mathcal H}_F = \Dom_F$ the sum is direct modulo $\Dom\cap\Dom_F$.
\end{proof}

\begin{lemma}\label{KreinInvariant}
The Krein extension $A_K : \Dom_K \subset H \to H$ is invariant.
\end{lemma}
\begin{proof}
This follows from a more general result by Makarov and Tsekanovskii \cite{MakarovTsekanovskii}, but we give an independent proof here. We have $A_K = A^*\Big|_{\Dom_K}$, where the domain of the Krein extension is given by
$$
\Dom_K = \{u \in \Dom_{\max};\; \exists u_k \in \Dom_{\min} : (A_F^{\frac{1}{2}}u_k)_k \subset H \textup{ is Cauchy},\; \lim\limits_{k\to\infty}A_{\min}u_k = A_{\max}u\},
$$
see Ando and Nishio \cite{AndoNishio}. We need to show that if $u \in \Dom_K$, then also $\kappa_{\varrho}u \in \Dom_K$. Let $u_k \in \Dom_{\min}$ such that $\lim\limits_{k \to \infty}A_{\min}u_k = A_{\max}u$, and such that $(A_F^{\frac{1}{2}}u_k)_k$ is a Cauchy sequence in $H$. Now $\kappa_{\varrho}u_k \in \Dom_{\min}$, and we have
$$
\lim\limits_{k\to\infty}A_{\min}\kappa_{\varrho}u_k = \lim\limits_{k\to\infty}\varrho^m\kappa_{\varrho}A_{\min}u_k = \varrho^m\kappa_{\varrho}A_{\max}u = A_{\max}\kappa_{\varrho}u.
$$
Moreover, we have
\begin{align*}
\|A_F^{\frac{1}{2}}\kappa_{\varrho}u_k - A_F^{\frac{1}{2}}\kappa_{\varrho}u_l\|_H^2 &= \langle A_F^{\frac{1}{2}}\kappa_{\varrho}(u_k-u_l),A_F^{\frac{1}{2}}\kappa_{\varrho}(u_k-u_l) \rangle_H \\
&= \langle A_{\min}\kappa_{\varrho}(u_k-u_l),\kappa_{\varrho}(u_k-u_l)\rangle_H \\
&= \langle \varrho^m\kappa_{\varrho}A_{\min}(u_k-u_l),\kappa_{\varrho}(u_k-u_l)\rangle_H \\
&= \varrho^m \langle A_{\min}(u_k-u_l),u_k-u_l \rangle_H \\
&= \varrho^m\|A_F^{\frac{1}{2}}u_k - A_F^{\frac{1}{2}}u_l\|_H^2,
\end{align*}
and thus $(A_F^{\frac{1}{2}}\kappa_{\varrho}u_k) \subset H$ is Cauchy. This shows that $\kappa_{\varrho}u \in \Dom_K$ and finishes the proof of the lemma.
\end{proof}

\begin{lemma}\label{HKModHF}
Suppose $\spec(\hat{\g})\cap\{\sigma\in\C;\;\Im(\sigma)=-\frac{m}{2}\}=\emptyset$. Then the inclusion map $\Dom_{\max} \hookrightarrow {\mathcal H}_K$ is well-defined and continuous. We have $\Dom_{\max} = \Dom_K + \Dom_F$, and this sum is direct modulo $\Dom_{\min}$.
\end{lemma}
\begin{proof}
We have $\Dom_F \subset {\mathcal H}_F \subset {\mathcal H}_K$. Moreover, let $\Dom_{\min} \subset \Dom \subset \Dom_{\max}$ be such that
$$
\hat{\E}_{\Dom} = \bigoplus\limits_{\substack{\sigma\in\spec(\hat{\g}) \\ \Im(\sigma) > -\frac{m}{2}}}\hat{\E}_{\sigma}.
$$
Then $A_{\Dom} = A_{\Dom}^*$ is selfadjoint by the Canonical Form Theorem and our present assumption that $\hat{\g}$ does not have eigenvalues on $\Im(\sigma)=-\frac{m}{2}$. Moreover, $A_{\Dom}$ is also invariant, and therefore $A_{\Dom} \geq 0$ by Lemma~\ref{BoundStationary}. Thus $A_K \leq A_{\Dom}$ by the fundamental property of the Krein extension which shows that $\Dom \subset {\mathcal H}_{\Dom} \subset {\mathcal H}_K$. Consequently, $\Dom_F + \Dom = \Dom_{\max} \subset {\mathcal H}_K$. Continuity of the embedding $\Dom_{\max} \hookrightarrow {\mathcal H}_K$ follows from the Closed Graph Theorem.

By Lemma~\ref{HDRepresentation}, every $u \in \Dom_{\max}$ now has a representation $u = u_K + \tilde{u}$ with $u_K \in \Dom_K$ and $\tilde{u} \in {\mathcal H}_F$. But $\tilde{u} = u - u_K \in \Dom_{\max}\cap{\mathcal H}_F = \Dom_F$, which shows that $\Dom_{\max} = \Dom_K + \Dom_F$. We thus have $\hat{\E}_{\max} = \hat{\E}_K + \hat{\E}_F$, and for dimensional reasons this sum must be direct. The lemma is proved.
\end{proof}

\begin{proof}[Proof of Theorem~\ref{KreinExtension}]
We first consider the case that
$$
\spec(\hat{\g})\cap\{\sigma\in\C;\;\Im(\sigma)=-\tfrac{m}{2}\}=\emptyset.
$$
By Lemma~\ref{HKModHF} we have a direct sum $\hat{\E}_{\max} = \hat{\E}_K \oplus \hat{\E}_F$. Now let
$$
\hat{u} \in \bigoplus\limits_{\substack{\sigma\in\spec(\hat{\g}) \\ \Im(\sigma) > -\frac{m}{2}}}\hat{\E}_{\sigma} \subset \hat{\E}_{\max}
$$
be arbitrary, and write $\hat{u} = \hat{u}_K + \hat{u}_F$ with $\hat{u}_K \in \hat{\E}_K$ and $\hat{u}_F \in \hat{\E}_F$. Then
$$
\hat{u}_K = \hat{u} + (-\hat{u}_F) \in \Biggl[\bigoplus\limits_{\substack{\sigma\in\spec(\hat{\g}) \\ \Im(\sigma) > -\frac{m}{2}}}\hat{\E}_{\sigma}\Biggr] \oplus
\Biggl[\bigoplus\limits_{\substack{\sigma\in\spec(\hat{\g}) \\ \Im(\sigma) < -\frac{m}{2}}}\hat{\E}_{\sigma}\Biggr],
$$
see Theorem~\ref{FriedrichsLowerHalf}, but because $\hat{\E}_K$ is invariant under $\hat{\kappa}_{\varrho}$ by Lemma~\ref{KreinInvariant} we have
$$
\hat{\E}_K = 
\Biggl[\bigoplus\limits_{\substack{\sigma\in\spec(\hat{\g}) \\ \Im(\sigma) > -\frac{m}{2}}}\bigl[\hat{\E}_{\sigma}\cap\hat{\E}_K\bigr]\Biggr] \oplus
\Biggl[\bigoplus\limits_{\substack{\sigma\in\spec(\hat{\g}) \\ \Im(\sigma) < -\frac{m}{2}}}\bigl[\hat{\E}_{\sigma}\cap\hat{\E}_K\bigr]\Biggr].
$$
Consequently both $\hat{u},\,\hat{u}_F \in \hat{\E}_K$, and so necessarily $\hat{u}_F = 0$ and $\hat{u} = \hat{u}_K \in \hat{\E}_K$. Thus
$$
\bigoplus\limits_{\substack{\sigma\in\spec(\hat{\g}) \\ \Im(\sigma) > -\frac{m}{2}}}\hat{\E}_{\sigma} \subset \hat{\E}_K,
$$
and for dimensional reasons these two spaces must be equal which proves the theorem in the present case.

We next consider the general case. Let $\Dom_{\min} \subset \Dom_0 \subset \Dom_{\max}$ be such that
$$
\hat{\E}_{\Dom_0} = \bigoplus\limits_{\substack{\sigma\in\spec(\hat{\g}) \\ \Im(\sigma) = -\frac{m}{2}}}\bigl[\hat{\E}_{\sigma}\cap\hat{\E}_K\bigr],
$$
and consider the operator $B_{\min} = A_{\Dom_0} : \Dom_0 \subset H \to H$. Note that $\Dom_0 = \Dom(B_{\min})$ is invariant under $\kappa_{\varrho}$, $B_{\min}$ is symmetric, and satisfies all the general assumptions imposed on $A_{\min}$. We have $B_{\max} = A_{\Dom_0}^*$, where
$$
\Dom(A_{\Dom_0}^*)/\Dom_{\min}(A) = \Biggl[\bigoplus\limits_{\substack{\sigma\in\spec(\hat{\g}) \\ \Im(\sigma) > -\frac{m}{2}}}\hat{\E}_{\sigma}\Biggr] \oplus
\Biggl[\bigoplus\limits_{\substack{\sigma\in\spec(\hat{\g}) \\ \Im(\sigma) < -\frac{m}{2}}}\hat{\E}_{\sigma}\Biggr] \oplus \hat{\E}_{\Dom_0}
$$
by the Canonical Form Theorem, and consequently
$$
\Dom(B_{\max})/\Dom(B_{\min}) \cong \Biggl[\bigoplus\limits_{\substack{\sigma\in\spec(\hat{\g}) \\ \Im(\sigma) > -\frac{m}{2}}}\hat{\E}_{\sigma}\Biggr] \oplus
\Biggl[\bigoplus\limits_{\substack{\sigma\in\spec(\hat{\g}) \\ \Im(\sigma) < -\frac{m}{2}}}\hat{\E}_{\sigma}\Biggr]
$$
is the spectral decomposition of the quotient $\Dom(B_{\max})/\Dom(B_{\min})$ associated with the generator of the induced group action. By what we have shown above, the Krein extension $B_K$ of $B_{\min}$ has domain given by
$$
\Dom(B_K)/\Dom(B_{\min}) \cong \bigoplus\limits_{\substack{\sigma\in\spec(\hat{\g}) \\ \Im(\sigma) > -\frac{m}{2}}}\hat{\E}_{\sigma},
$$
and we have $A_K = B_K$ as both extend $A_{\min}$ and $B_{\min}$, and are each minimal among the nonnegative selfadjoint extensions of these operators. In conclusion,
$$
\hat{\E}_K = \bigoplus\limits_{\substack{\sigma\in\spec(\hat{\g}) \\ \Im(\sigma) = -\frac{m}{2}}}\bigl[\hat{\E}_{\sigma}\cap \hat{\E}_K\bigr] \oplus
\bigoplus\limits_{\substack{\sigma\in\spec(\hat{\g}) \\ \Im(\sigma) > -\frac{m}{2}}}\hat{\E}_{\sigma}
$$
as desired. Lastly, if $A$ satisfies the sign condition, we necessarily must have $\hat{\E}_{\sigma}\cap \hat{\E}_K = \hat{\E}_{\sigma,\frac{1}{2}}$ for $\Im(\sigma)=-\frac{m}{2}$ because $\hat{\E}_{\sigma,\frac{1}{2}} \subset \bigl(\hat{\E}_{\sigma},[\cdot,\cdot]_{\hat{\E}_{\max}}\bigr)$ is the unique Lagrangian subspace that is invariant under $\hat{\kappa}_{\varrho}$.
\end{proof}

\begin{proof}[Proof of Theorem~\ref{InvariantOrder}]
For every invariant extension $A_{\Dom} = A_{\Dom}^*$ we have
$$
\hat{\E}_{\Dom} = \Biggl[\bigoplus\limits_{\substack{\sigma\in\spec(\hat{\g}) \\ \Im(\sigma) > -\frac{m}{2}}}\bigl[\hat{\E}_{\sigma}\cap\hat{\E}_{\Dom}\bigr]\Biggr] \oplus
\Biggl[\bigoplus\limits_{\substack{\sigma\in\spec(\hat{\g}) \\ \Im(\sigma) < -\frac{m}{2}}}\bigl[\hat{\E}_{\sigma}\cap\hat{\E}_{\Dom}\bigr]\Biggr] 
\oplus
\Biggl[\bigoplus\limits_{\substack{\sigma\in\spec(\hat{\g}) \\ \Im(\sigma) = -\frac{m}{2}}}\bigl[\hat{\E}_{\sigma}\cap\hat{\E}_{\Dom}\bigr]\Biggr],
$$
and because $A$ satisfies the sign condition we have $\hat{\E}_{\sigma}\cap\hat{\E}_{\Dom} = \hat{\E}_{\sigma,\frac{1}{2}}$ for every $\sigma \in \spec(\hat{\g})$ with $\Im(\sigma)=-\frac{m}{2}$. Thus $\bigoplus\limits_{\substack{\sigma\in\spec(\hat{\g}) \\ \Im(\sigma) = -\frac{m}{2}}}\hat{\E}_{\sigma,\frac{1}{2}}$ is part of every invariant selfadjoint extension. Consequently, by replacing the minimal extension $A_{\min}$ with $A_{\Dom_0}$ as in the proof of Theorem~\ref{KreinExtension} if necessary, where $\hat{\E}_{\Dom_0} = \bigoplus\limits_{\substack{\sigma\in\spec(\hat{\g}) \\ \Im(\sigma) = -\frac{m}{2}}}\hat{\E}_{\sigma,\frac{1}{2}}$, we may without loss of generality assume that
$$
\spec(\hat{\g})\cap\{\sigma\in\C;\;\Im(\sigma)=-\tfrac{m}{2}\}=\emptyset.
$$
As previously mentioned, the first part of Theorem~\ref{InvariantOrder} follows from \cite[Theorem~5.12.4]{GohbergLancasterRodman}. It is in fact an immediate consequence of the Canonical Form Theorem. Given an invariant subspace $\hat{U} \subset \hat{\E}_-$ as stated in the assumptions, the domain $\Dom$ of the unique selfadjoint extension with $\hat{\E}_{\Dom}\cap\hat{\E}_- = \hat{U}$ must necessarily be
$$
\hat{\E}_{\Dom} = \Biggl[\bigoplus\limits_{\substack{\sigma\in\spec(\hat{\g}) \\ \Im(\sigma) > -\frac{m}{2}}}\bigl[\hat{\E}_{\sigma}\cap\hat{U}^{[\perp]}\bigr]\Biggr] \oplus
\Biggl[\bigoplus\limits_{\substack{\sigma\in\spec(\hat{\g}) \\ \Im(\sigma) < -\frac{m}{2}}}\bigl[\hat{\E}_{\sigma}\cap\hat{U}\bigr]\Biggr].
$$

We now proceed to prove the second part of the theorem. Let $A_{\Dom} = A_{\Dom}^*$ be invariant. According to Lemma~\ref{HDRepresentation} we have ${\mathcal H}_{\Dom} = \Dom + {\mathcal H}_F$, and because $\Dom$ is invariant we conclude from Theorem~\ref{FriedrichsLowerHalf} and Theorem~\ref{KreinExtension} that ${\mathcal H}_{\Dom} = \Dom\cap\Dom_K + {\mathcal H}_F$, and this sum is direct modulo $\Dom_{\min}$. We also note that $\Dom\cap\Dom_K \subset \bigl({\mathcal H}_{\Dom},|\cdot|_{\Dom}\bigr)$ is dense. Consequently, for invariant extensions $A_{\Dom_j} = A_{\Dom_j}^*$, $j=1,2$, we have ${\mathcal H}_{\Dom_2} \subset {\mathcal H}_{\Dom_1}$ if and only if $\Dom_2\cap\Dom_K \subset \Dom_1\cap\Dom_K$, which by the Canonical Form Theorem is equivalent to $\Dom_1\cap\Dom_F \subset \Dom_2\cap\Dom_F$. Moreover, whenever ${\mathcal H}_{\Dom_2} \subset {\mathcal H}_{\Dom_1}$ the embedding $\bigl({\mathcal H}_{\Dom_2},|\cdot|_{\Dom_2}\bigr) \hookrightarrow \bigl({\mathcal H}_{\Dom_1},|\cdot|_{\Dom_1}\bigr)$ is continuous by the Closed Graph Theorem, and $\|A_{\Dom_1}^{\frac{1}{2}}u\|_H = \|A_{\Dom_2}^{\frac{1}{2}}u\|_H$ for $u \in {\mathcal H}_{\Dom_2}$. To see the latter norm identity note that it is true for $u \in \Dom_1\cap\Dom_2 \supset \Dom_2\cap\Dom_K$, which is dense in ${\mathcal H}_{\Dom_2}$, and so it extends by continuity to all of ${\mathcal H}_{\Dom_2}$. Consequently, $A_{\Dom_1} \leq A_{\Dom_2}$ if and only if ${\mathcal H}_{\Dom_2} \subset {\mathcal H}_{\Dom_1}$ if and only if $\Dom_1\cap\Dom_F \subset \Dom_2\cap\Dom_F$. The theorem is proved.
\end{proof}


\section{Indicial operators}\label{IndicialOperators}

\noindent
In the remaining sections we will discuss applications of the previous results to indicial operators. As discussed in the introduction, indicial operators arise as model operators associated with singularities of corner type. Taking an operator theoretical point of view, we consider indicial operators of the form
\begin{equation}\label{Aindicial}
A = x^{-m}\sum\limits_{j=0}^{\mu} a_j(xD_x)^j : C_c^{\infty}(\R_+;E_1) \subset L^2_b(\R_+;E_0) \to L^2_b(\R_+;E_0),
\end{equation}
where $m,\mu \in \N$ and $E_0$ and $E_1$ are Hilbert spaces such that $E_1 \hookrightarrow E_0$ is continuous and dense, and the operators $a_j : E_1 \to E_0$ are continuous for $j=0,\ldots,\mu$. The (vector-valued) space $L^2_b$ is the $L^2$-space with respect to the Haar measure $\frac{dx}{x}$ on the half-line. The \emph{indicial family} of $A$
$$
p(\sigma) = \sum\limits_{j=0}^{\mu} a_j \sigma^j : E_1 \to E_0, \quad \sigma \in \C
$$
is a holomorphic family of Fredholm operators in $\L(E_1,E_0)$ and satisfies suitable ellipticity assumptions, detailed below, and when considered an unbounded operator acting in $E_0$ with domain $E_1$ we require
$$
p(\sigma^{\star})^* = p(\sigma) : E_1 \subset E_0 \to E_0
$$
to hold for the Hilbert space adjoints for each $\sigma \in \C$, where $\sigma^{\star} = \overline{\sigma} - im$ denotes reflection about the line $\Im(\sigma)=-\frac{m}{2}$, see also \eqref{Reflection}. The latter assumption ensures that the operator $A$ as given in \eqref{Aindicial} is symmetric in $H = L^2_b(\R_+;E_0)$ with dense domain $\Dom_c = C_c^{\infty}(\R_+;E_1)$. On $H = L^2_b(\R_+;E_0)$ we consider the $\R_+$-action given by
\begin{equation}\label{dilation}
\kappa_{\varrho}u(x) = u(\varrho x), \; \varrho > 0, \; u \in L^2_b.
\end{equation}
This is a strongly continuous and unitary group action on $L^2_b$ that leaves $C_c^{\infty}$ invariant, and we have $A = \varrho^m\kappa_{\varrho}A\kappa_{\varrho}^{-1} : \Dom_c \to H$ for all $\varrho > 0$. The infinitesimal generator of this group action is $\g = xD_x$.

For indicial operators associated with cone or corner singularities we typically have $\mu = m$ in \eqref{Aindicial}, but we allow these parameters to be decoupled here to include more general singular behavior; this requires anisotropic estimates. We fix the following notation:
\begin{equation}\label{ellchoices}
\begin{gathered}
\vec{\ell} = (\ell_1,\ell_2) = (\mu,m) \\
\langle \lambda,\sigma \rangle_{\vec{\ell}} = \bigl(1 + \lambda^{2\ell_2} + \sigma^{2\ell_1}\bigr)^{\frac{1}{2\ell_1\ell_2}} = \bigl(1 + \lambda^{2m} + \sigma^{2\mu}\bigr)^{\frac{1}{2m\mu}}
\end{gathered}
\end{equation}
for $(\lambda,\sigma) \in \R^2$, see also Appendix~\ref{PseudoCalculus}.

The following standing assumptions are imposed on the indicial family $p(\sigma)$:
\begin{enumerate}
\renewcommand{\theenumi}{A-\arabic{enumi}}
\item $p(\sigma) : E_1 \subset E_0 \to E_0$ is closed, densely defined, and Fredholm for $\sigma \in \C$, and the map $\C \ni \sigma \mapsto p(\sigma) \in \L(E_1,E_0)$ is holomorphic. \label{Fredholm}
\item We have $p(\sigma^{\star})^* = p(\sigma) : E_1 \subset E_0 \to E_0$ as unbounded operators in $E_0$, where $\sigma^{\star} = \overline{\sigma} - im$. \label{Symmetry}
\item For $(\lambda,\sigma) \in \R^2$ and $|\lambda,\sigma| \geq R \gg 0$ sufficiently large $p(\sigma) \pm i\lambda^m : E_1 \to E_0$ is invertible, and the inverse satisfies \label{InverseBounds}
$$
\sup\limits_{|\lambda,\sigma| \geq R}\Bigl\{\langle \lambda,\sigma \rangle_{\vec{\ell}}^{m\mu}\bigl\|\bigl(p(\sigma) \pm i\lambda^m\bigr)^{-1}\bigr\|_{\L(E_0)} + \bigl\|\bigl(p(\sigma) \pm i\lambda^m\bigr)^{-1}\bigr\|_{\L(E_0,E_1)}\Bigr\} < \infty.
$$
\item For $(\lambda,\sigma) \in \R^2$ and $R \gg 0$ as in \eqref{InverseBounds} and every $k \in \{1,\ldots,\mu\}$ we have \label{InverseDeriv}
$$
\sup\limits_{|\lambda,\sigma| \geq R}\langle \lambda,\sigma \rangle_{\vec{\ell}}^{mk}\bigl\|\bigl[\partial_{\sigma}^kp(\sigma)\bigr]\bigl(p(\sigma) \pm i\lambda^m\bigr)^{-1}\bigr\|_{\L(E_0)} < \infty.
$$
\end{enumerate}
Note that we are not imposing assumptions regarding the deficiency indices of the operator \eqref{Aindicial}, and neither do we place any assumptions on the invertibility of $p(\sigma)$ along a line $\Im(\sigma) = \gamma$ for any specific value of $\gamma \in \R$, nor on the embedding of the spaces $E_1 \hookrightarrow E_0$ other than continuity and density.

Assumptions \eqref{InverseBounds} and \eqref{InverseDeriv} are ellipticity conditions on $p(\sigma)$. To illustrate this we briefly discuss conical singularities in Example~\ref{conicexample} below.

\begin{example}\label{conicexample}
Let $Y$ be a closed compact Riemannian manifold and consider an indicial family of the form
$$
p(\sigma) = \sum\limits_{j=0}^{\mu}a_j(y,D_y)\sigma^j : C^{\infty}(Y) \to C^{\infty}(Y),
$$
where $a_j(y,D_y) \in \Diff^{\mu-j}(Y)$ for $j=0,\ldots,\mu$. We assume that the parameter-dependent principal symbol
$$
\sym(p)(y,\eta;\sigma) = \sum\limits_{j=0}^{\mu}\sym(a_j)(y,\eta)\sigma^j
$$
is invertible for $(y,\eta;\sigma) \in \bigl(T^*Y\times\R\bigr) \setminus 0$, and that $\sym(p)(y,\eta;\sigma) = \sym(p)(y,\eta;\sigma)^*$. Then the family $p(\sigma)$ satisfies both assumptions \eqref{InverseBounds} and \eqref{InverseDeriv} with $E_0 = L^2(Y)$ and $E_1 = H^{\mu}(Y)$.

To see this note that $a(y,\eta;\lambda,\sigma):= \sym(p)(y,\eta;\sigma) \pm i\lambda^m$ satisfies
$$
a(y,\varrho^{m}\eta;\varrho^{\mu}\lambda,\varrho^{m}\sigma) = \varrho^{m\mu}a(y,\eta;\lambda,\sigma), \quad \varrho > 0,
$$
and is invertible for $(y,\eta;\lambda,\sigma) \in \bigl(T^*Y\times\R^2\bigr)\setminus 0$ by the symmetry and invertibility assumptions on $\sym(p)(y,\eta;\sigma)$. Thus $p(\sigma) \pm i\lambda^m$ is an elliptic family of order $m\mu$ that depends anisotropically on the parameters $(\lambda,\sigma) \in \R^2$, and the standard parametrix construction with anisotropic parameters furnishes a parametrix $b(\lambda,\sigma)$ of order $-m\mu$ such that
\begin{equation}\label{ppmlparam}
\bigl(p(\sigma)\pm i\lambda^m\bigr)b(\lambda,\sigma) - 1,\, b(\lambda,\sigma)\bigl(p(\sigma)\pm i\lambda^m\bigr) - 1 \in \S(\R^2;\Psi^{-\infty}(Y)).
\end{equation}
Locally, $b(\lambda,\sigma)$ is quantized from symbols $b(y,\eta;\lambda,\sigma)$ that satisfy anisotropic symbol estimates of order $t = -m\mu$ of the form
\begin{equation}\label{symbestimateslocal}
|D_y^{\alpha}\partial_{(\eta,\lambda,\sigma)}^{\beta}b(y,\eta;\lambda,\sigma)| \lesssim \bigl[(1 + |\eta|^{2\mu} + \lambda^{2m} + \sigma^{2\mu})^{\frac{1}{2m\mu}}\bigr]^{t -m|\beta_1| - \mu\beta_2 -m\beta_3}
\end{equation}
for $\alpha \in \N_0^{\dim Y}$ and $\beta = (\beta_1,\beta_2,\beta_3) \in \N_0^{\dim Y}\times\N_0\times\N_0$, while $p(\sigma)\pm i\lambda^m$ is based on such symbols with $t = m\mu$ and principal symbol $a(y,\eta;\lambda,\sigma)$ above.

The symbol estimates \eqref{symbestimateslocal} for $b(y,\eta;\lambda,\sigma)$ with $t = -m\mu$ imply
\begin{align*}
\langle \eta \rangle^{|\beta_1|}|D_y^{\alpha}\partial_{\eta}^{\beta_1}b(y,\eta;\lambda,\sigma)| &\lesssim \langle \lambda,\sigma \rangle_{\vec{\ell}}^{-m\mu}, \\
\langle \eta \rangle^{\mu+|\beta_1|}|D_y^{\alpha}\partial_{\eta}^{\beta_1}b(y,\eta;\lambda,\sigma)| &\lesssim 1,
\end{align*}
and so 
$$
\sup\Bigl\{\langle \lambda,\sigma \rangle_{\vec{\ell}}^{m\mu}\bigl\|b(\lambda,\sigma)\bigr\|_{\L(H^s(Y))} + \bigl\|b(\lambda,\sigma)\bigr\|_{\L(H^s(Y),H^{s+\mu}(Y))}\Bigr\} < \infty
$$
for every $s \in \R$, in particular for $s = 0$, and because of \eqref{ppmlparam} we can replace $b(\lambda,\sigma)$ by $\bigl(p(\sigma) \pm i\lambda^m\bigr)^{-1}$ for $|\lambda,\sigma| \gg 0$ large enough in these estimates, showing \eqref{InverseBounds}.

The composition
$$
\bigl[\partial_{\lambda}^l\partial_{\sigma}^k \bigl(p(\sigma)\pm i\lambda^m\bigr)\bigr]\bigl(p(\sigma)\pm i\lambda^m\bigr)^{-1} : C^{\infty}(Y) \to C^{\infty}(Y)
$$
is an anisotropic family of pseudodifferential operators depending on $(\lambda,\sigma)$ of order $t = (m\mu-\mu l- mk) + (-m\mu) = -\mu l - mk$ by the composition theorem, and so locally is based on symbols that satisfy the estimates \eqref{symbestimateslocal} with this value of $t$. Consequently,
$$
\sup\limits_{|\lambda,\sigma| \geq R \gg 0}\Bigl\{\langle \lambda,\sigma \rangle_{\vec{\ell}}^{\mu l + mk}\bigl\|\bigl[\partial_{\lambda}^l\partial_{\sigma}^k \bigl(p(\sigma)\pm i\lambda^m\bigr)\bigr]\bigl(p(\sigma)\pm i\lambda^m\bigr)^{-1}\bigr\|_{\L(H^s(Y))}\Bigr\} < \infty
$$
for every $s \in \R$, showing \eqref{InverseDeriv}.
\end{example}

We proceed to further discuss the assumptions in the abstract setting:

First observe that \eqref{InverseBounds} and \eqref{InverseDeriv} imply that
$$
\sup\limits_{|\lambda,\sigma| \geq R \gg 0}\Bigl\{\langle \lambda,\sigma \rangle_{\vec{\ell}}^{\mu l + mk}\bigl\|\bigl[\partial_{\lambda}^l\partial_{\sigma}^k \bigl(p(\sigma)\pm i\lambda^m\bigr)\bigr]\bigl(p(\sigma)\pm i\lambda^m\bigr)^{-1}\bigr\|_{\L(E_0)}\Bigr\} < \infty
$$
holds for all $k,l \in \N_0$, which at first sight may appear stronger than \eqref{InverseDeriv}. In the abstract setting, assumption \eqref{InverseDeriv} serves a dual purpose as an ellipticity assumption and as a replacement for the notion of order for the coefficients $a_j$ in \eqref{Aindicial}.

Next note that for every $\gamma \in \R$ we can conclude that $p(\sigma + i\gamma) \pm i\lambda^m : E_1 \to E_0$ is invertible for large enough $|\lambda,\sigma| \gg 0$, $(\lambda,\sigma) \in \R^2$, and the estimates stated for $p(\sigma)$ in \eqref{InverseBounds} and \eqref{InverseDeriv} are equally valid for $p(\sigma+i\gamma)$. This means that the ellipticity conditions \eqref{InverseBounds} and \eqref{InverseDeriv} are invariant with respect to shifting $p(\sigma)$ in the complex plane, or, on the operator level, that these conditions are invariant with respect to conjugating the operator $A$ in \eqref{Aindicial} by weights $x^{\gamma}Ax^{-\gamma}$ for any $\gamma \in \R$. Only the symmetry assumption \eqref{Symmetry} depends on weights. Moreover, the lower bound $R \gg 0$ for the invertibility, and the estimates in \eqref{InverseBounds} and \eqref{InverseDeriv} for $p(\sigma+i\gamma)$ in place of $p(\sigma)$, are uniform as $\gamma \in \R$ varies in compact intervals\footnote{These statements are not immediately obvious. For a proof we refer to Lemma~\ref{SymbolInverseProps} in Appendix~\ref{PseudoCalculus}. In the terminology of this appendix, assumptions \eqref{InverseBounds} and \eqref{InverseDeriv} imply that $p(\sigma) \pm ix^m \in S^{m\mu;\vec{\ell}}_O(\R\times\C;\L(E_1,E_0))$ is right-hypoelliptic of order $(m\mu,0)$ in the sense of Definition~\ref{hypoellipticity}. The symbolic parametrix $q(x,\sigma)$ defined there via the kernel cut-off construction belongs to $S^{0;\vec{\ell}}_O(\R\times\C;\L(E_0,E_1))\cap S^{-m\mu;\vec{\ell}}_O(\R\times\C;\L(E_0))$. We have $(p(\sigma) \pm ix^m)q(x,\sigma) - 1 \in S^{-\infty}_O(\R\times\C;\L(E_0))$ and $q(x,\sigma)(p(\sigma) \pm ix^m) - 1 \in S^{-\infty}_O(\R\times\C;\L(E_1))$, and combined with the estimates for the derivatives in Part (b) of Lemma~\ref{SymbolInverseProps} we obtain all stated properties.}.

In particular, for $\lambda = 0$, we see that the indicial family $p(\sigma) : E_1 \to E_0$ is invertible for sufficiently large $|\Re(\sigma)|$, and the lower bounds for invertibility can be chosen locally uniformly with respect to $\Im(\sigma)$. Let as usual
$$
\spec_b(A) = \{\sigma \in \C;\; p(\sigma): E_1 \to E_0 \textup{ is not invertible}\}
$$
be the set of indicial roots of $A$. Each strip in the complex plane of the form $-T \leq \Im(\sigma) \leq T$ then contains only finitely many indicial roots by \eqref{Fredholm} and analytic Fredholm theory. We also note that $\spec_b(A)$ is symmetrical about the line $\Im(\sigma) = -\frac{m}{2}$ by \eqref{Symmetry}.

In order to describe the extensions of $A$  we will use von Neumann's formulas
\begin{equation}\label{vonNeumannFormula}
\Dom_{\max} = \Dom_{\min} \oplus \ker(A_{\max} + i) \oplus \ker(A_{\max}-i)
\end{equation}
as a starting point. One of the main ingredients needed for proving the Classification Theorem ~\ref{DmaxThm} for extensions is showing that the functions in $\ker(A_{\max} \pm i)$ decay rapidly as $x \to \infty$, which will be accomplished via the construction of parametrices for both $A \pm i$ and is largely relegated to Appendix~\ref{PseudoCalculus}. In particular, we will need to use Mellin pseudodifferential operators. To fix notation we are going to write
\begin{align*}
\bigr(M u\bigl)(\sigma) &= \int_0^{\infty}x^{-i\sigma}u(x)\,\frac{dx}{x}, \quad \sigma \in \R, \\
\bigr(M^{-1}v\bigl)(x) &= \frac{1}{2\pi}\int_{\R}x^{i\sigma}v(\sigma)\,d\sigma, \quad x > 0,
\end{align*}
for the Mellin transform and its inverse. Mellin pseudodifferential operators, in most general form, are written as
$$
\bigl[\opm(a)u\bigr](x) = \frac{1}{2\pi}\int_{\R}\int_0^{\infty}\Bigl(\frac{x}{x'}\Bigr)^{i\sigma}a(x,x',\sigma)u(x')\,\frac{dx'}{x'}\,d\sigma
$$
for a compound symbol $a(x,x',\sigma)$. 


\section{The space ${\mathscr H}(\R_+;E_1)$ and the minimal domain}\label{MinimalDomain}

\begin{lemma}\label{BasicHolom}
\begin{enumerate}[(a)]
\item Along every line $\Im(\sigma)=\gamma$ in the complex plane the operator $p(\sigma+i\gamma) : E_1 \to E_0$, $\sigma \in \R$, is invertible for $|\sigma| \gg 0$ sufficiently large, and the corresponding lower bounds for invertibility can be chosen locally uniformly with respect to $\gamma \in \R$. We have
$$
p(\sigma+i\gamma)^{-1} \in S^{-\mu}(\R_{\sigma};\L(E_0))\cap S^0(\R_{\sigma};\L(E_0,E_1))
$$
for $|\sigma| \gg 0$ with locally uniform symbol estimates with respect to $\gamma \in \R$.
\item For $\gamma_j \in \R$, $j=1,2$, we have
$$
\bigl(\partial_{\sigma}^kp\bigr)(\sigma+i\gamma_2)p^{-1}(\sigma+i\gamma_1) \in S^{-k}(\R_{\sigma};\L(E_0))
$$
 for $|\sigma| \gg 0$, $k \in \N_0$.
\item There is a holomorphic family $q(\sigma) : \C \to \L(E_0,E_1)$ such that
$$
q(\sigma+i\gamma) \in S^{-\mu}(\R_{\sigma};\L(E_0))\cap S^0(\R_{\sigma};\L(E_0,E_1))
$$
with locally uniform symbol estimates with respect to $\gamma \in \R$ such that
$$
r(\sigma) = p(\sigma)^{-1} - q(\sigma) : \C \to \L(E_0,E_1)
$$
is meromorphic on $\C$ with at most finitely many poles in each strip of the form $\{\sigma \in \C;\; -T \leq \Im(\sigma) \leq T\}$ for every $T \in \N$, and if $\chi \in C^{\infty}(\C)$ is an excision function for the poles we have $(\chi r)(\sigma+i\gamma) \in \S(\R_{\sigma};\L(E_0,E_1))$ with locally uniform estimates with respect to $\gamma \in \R$.
\end{enumerate}
\end{lemma}
\begin{proof}
This follows from Lemma~\ref{SymbolInverseProps} in Appendix~\ref{PseudoCalculus}. To apply that lemma observe that
$$
a(x,\sigma) = p(\sigma) \pm ix^m \in S^{m\mu;\vec{\ell}}_O(\R\times\C;\L(E_1,E_0))
$$
is right-hypoelliptic of order $(m\mu,0)$ by our standing assumptions, and
$$
q(x,\sigma) \in S^{0;\vec{\ell}}_O(\R\times\C;\L(E_0,E_1))\cap S^{-m\mu;\vec{\ell}}_O(\R\times\C;\L(E_0)).
$$
All assertions follow when restricting to $x=0$.
\end{proof}

\begin{definition}
Fix $\gamma_0 \in \R$ such that $p(\sigma + i\gamma_0) : E_1 \to E_0$ is invertible for all $\sigma \in \R$ and define
$$
{\mathscr H}(\R_+;E_1) = \textup{Range of } \opm(p(\sigma+i\gamma_0)^{-1}) : L_b^2(\R_+;E_0) \to L_b^2(\R_+;E_1).
$$
\end{definition}

Note that $p(\sigma+i\gamma_0)^{-1} \in S^0(\R_{\sigma};\L(E_0,E_1))$, so
$$
\opm(p(\sigma+i\gamma_0)^{-1}) : L_b^2(\R_+;E_0) \to L_b^2(\R_+;E_1)
$$
is bounded. This operator is also injective, and consequently we can describe ${\mathscr H}(\R_+;E_1)$ equivalently as
$$
u \in {\mathscr H}(\R_+;E_1) \Longleftrightarrow u \in L^2_b(\R_+;E_1) \textup{ and } \opm(p(\sigma+i\gamma_0))u \in L^2_b(\R_+;E_0).
$$
Note that $\opm(p(\sigma+i\gamma_0)) : H^{-\infty}_b(\R_+;E_1) \to H^{-\infty}_b(\R_+;E_0)$ is bijective with inverse $\opm(p(\sigma+i\gamma_0)^{-1})$.

\begin{proposition}\label{BaseSpaceProps}
\begin{enumerate}[(a)]
\item ${\mathscr H}(\R_+;E_1)$ is a Hilbert space with respect to the inner product
$$
\langle u,v \rangle_{{\mathscr H}} = \langle \opm(p(\sigma+i\gamma_0))u,\opm(p(\sigma+i\gamma_0))v \rangle_{L^2_b(\R_+;E_0)}.
$$
\item Let $\gamma \in \R$ be any other value such that $p(\sigma+i\gamma) : E_1 \to E_0$ is invertible for all $\sigma \in \R$. Then ${\mathscr H}(\R_+;E_1)$ can equivalently be defined as
$$
{\mathscr H}(\R_+;E_1) = \textup{Range of } \opm(p(\sigma + i\gamma)^{-1}) : L_b^2(\R_+;E_0) \to L_b^2(\R_+;E_1),
$$
and the inner product
$$
\langle u,v \rangle'_{{\mathscr H}} = \langle \opm(p(\sigma+i\gamma))u,\opm(p(\sigma+i\gamma))v \rangle_{L^2_b(\R_+;E_0)}
$$
induces an equivalent Hilbert space structure on ${\mathscr H}(\R_+;E_1)$ (i.e. the norms $\|\cdot\|_{{\mathscr H}}$ and $\|\cdot\|'_{{\mathscr H}}$ are equivalent).
\item $C_c^{\infty}(\R_+;E_1) \subset {\mathscr H}(\R_+;E_1)$ is dense and continuously embedded.
\item $H^{\mu}_b(\R_+;E_1) \hookrightarrow {\mathscr H}(\R_+;E_1) \hookrightarrow H^{\mu}_b(\R_+;E_0)\cap L^2_b(\R_+;E_1)$, and
$$
{\mathscr H}(\R_+;E_1) = H^{\mu}_b(\R_+;E_0)\cap L^2_b(\R_+;E_1)
$$
if and only if $\opm(p(\sigma+i\gamma_0)) : H^{\mu}_b(\R_+;E_0)\cap L^2_b(\R_+;E_1) \to L^2_b(\R_+;E_0)$.
\end{enumerate}
\end{proposition}
\begin{proof}
Assertion (a) follows at once from the boundedness and injectivity of
$$
\opm(p(\sigma+i\gamma_0)^{-1}) : L_b^2(\R_+;E_0) \to L_b^2(\R_+;E_1) \hookrightarrow L_b^2(\R_+;E_0).
$$
As both
$$
p(\sigma + i\gamma)p(\sigma+i\gamma_0)^{-1},\; p(\sigma + i\gamma_0)p(\sigma+i\gamma)^{-1} \in S^0(\R_{\sigma};\L(E_0))
$$
by Lemma~\ref{BasicHolom} we have that
$$
\opm(p(\sigma + i\gamma)p(\sigma+i\gamma_0)^{-1}) : L^2_b(\R_+;E_0) \to L^2_b(\R_+;E_0)
$$
is bounded and invertible with inverse $\opm(p(\sigma+i\gamma_0)p(\sigma + i\gamma)^{-1})$. Consequently, the operators
$\opm(p(\sigma + i\gamma)^{-1}) : L^2_b(\R_+;E_0) \to L^2_b(\R_+;E_1)$ and
\begin{align*}
\opm(p(\sigma+i\gamma_0)^{-1}) &=
\opm(p(\sigma + i\gamma)^{-1}) \opm(p(\sigma + i\gamma)p(\sigma+i\gamma_0)^{-1}) \\
&: L^2_b(\R_+;E_0) \to L^2_b(\R_+;E_1)
\end{align*}
have the same range, proving (b).

For (c) note that
$$
\opm(p(\sigma+i\gamma_0)) : \T(\R_+;E_1) \to \T(\R_+;E_0)
$$
is continuous and invertible with inverse $\opm(p(\sigma+i\gamma_0)^{-1})$.  Here $\T$ denotes the space of functions $u(x)$ on $\R_+$ such that $u(e^t)$ is a Schwartz function on $\R$. Thus $\T(\R_+;E_1) \hookrightarrow {\mathscr H}(\R_+;E_1)$, and the density of $\T(\R_+;E_0) \subset L^2_b(\R_+;E_0)$ implies the density of $\T(\R_+;E_1)$ in ${\mathscr H}(\R_+;E_1)$. Because $C_c^{\infty}(\R_+;E_1) \hookrightarrow \T(\R_+;E_1)$ is dense (c) follows.

Lastly, (d) follows because
\begin{gather*}
\opm(p(\sigma+i\gamma_0)) : H^{\mu}_b(\R_+;E_1) \to L^2_b(\R_+;E_0), \\
\opm(p(\sigma+i\gamma_0)^{-1}) : L^2_b(\R_+;E_0) \to H^{\mu}_b(\R_+;E_0)\cap L^2_b(\R_+;E_1)
\end{gather*}
are bounded by Lemma~\ref{BasicHolom}.
\end{proof}

\begin{remark}
A typical situation in applications is that
\begin{equation}\label{standardspace}
{\mathscr H}(\R_+;E_1) = H^{\mu}_b(\R_+;E_0)\cap L^2_b(\R_+;E_1).
\end{equation}
Let $\Lambda : E_1 \subset E_0 \to E_0$ be selfadjoint and invertible\footnote{For example, we may choose $\sigma_0 \in \C$ with $\Im(\sigma_0) = -\frac{m}{2}$ such that $\Lambda = p(\sigma_0) : E_1 \to E_0$ is invertible and selfadjoint by \eqref{Symmetry}.}, and define for $\sigma \in \R$
$$
{\mathfrak r}(\sigma) = \langle \sigma \rangle^{\mu} + i\Lambda : E_1 \to E_0.
$$
Then ${\mathfrak r}(\sigma) \in S^{\mu}(\R;\L(E_1,E_0))$ is invertible for all $\sigma \in \R$, and by the Spectral Theorem we have that ${\mathfrak r}(\sigma)^{-1} \in S^{-\mu}(\R;\L(E_0))\cap S^0(\R;\L(E_0,E_1))$. Thus
$$
\opm({\mathfrak r}) : H^{\mu}_b(\R_+;E_0)\cap L^2_b(\R_+;E_1) \to L^2_b(\R_+;E_0)
$$
is an isomorphism with inverse $\opm({\mathfrak r}^{-1})$. Our assumptions on $p(\sigma)$ and Lemma~\ref{BasicHolom} imply that
$$
{\mathfrak r}(\sigma)p(\sigma+i\gamma_0)^{-1} = \langle\sigma\rangle^{\mu}p(\sigma+i\gamma_0)^{-1} + i\Lambda p(\sigma+i\gamma_0)^{-1} \in S^0(\R;\L(E_0)).
$$
If the latter happens to be an elliptic operator valued symbol then $p(\sigma+i\gamma_0){\mathfrak r}(\sigma)^{-1} \in S^0(\R;\L(E_0))$, and therefore the identity \eqref{standardspace} holds in this case.
\end{remark}

The space ${\mathscr H}(\R_+;E_1)$ naturally arises from an unbounded operator perspective.

\begin{proposition}\label{BaseSpaceUnbdd}
For every $\gamma \in \R$ the operator
$$
\opm(p(\sigma+i\gamma)) : C_c^{\infty}(\R_+;E_1) \subset L^2_b(\R_+;E_0) \to L^2_b(\R_+;E_0)
$$
is closable, and the closure is $\opm(p(\sigma+i\gamma)) : {\mathscr H}(\R_+;E_1) \to L^2_b(\R_+;E_0)$. Moreover, the adjoint of this operator as an unbounded operator acting in $L^2_b(\R_+;E_0)$ is $\opm(p(\sigma-i(\gamma+m))) : {\mathscr H}(\R_+;E_1) \to L^2_b(\R_+;E_0)$.
\end{proposition}
\begin{proof}
We note first that $\opm(p(\sigma+i\gamma)) : {\mathscr H}(\R_+;E_1) \to L^2_b(\R_+;E_0)$ is well-defined and continuous because $p(\sigma+i\gamma)p(\sigma+i\gamma_0)^{-1} \in S^0(\R;\L(E_0))$ by Lemma~\ref{BasicHolom}. Writing $H = L^2_b(\R_+;E_0)$ we have
$$
\|u\|_H + \|\opm(p(\sigma+i\gamma))u\|_H \lesssim \|\opm(p(\sigma+i\gamma_0))u\|_H = \|u\|_{{\mathscr H}}
$$
for $u \in {\mathscr H}(\R_+;E_1)$ in view of $u = \opm(p(\sigma+i\gamma_0)^{-1})\opm(p(\sigma+i\gamma_0))u$. On the other hand, with $q(\sigma)$ from Lemma~\ref{BasicHolom} we have
$$
p(\sigma+i\gamma_0) = p(\sigma+i\gamma_0)q(\sigma+i\gamma)p(\sigma+i\gamma) - p(\sigma+i\gamma_0)l(\sigma+i\gamma),
$$
where $p(\sigma+i\gamma_0)q(\sigma+i\gamma) \in S^0(\R;\L(E_0))$ and $p(\sigma+i\gamma_0)l(\sigma+i\gamma) \in \S(\R;\L(E_0))$, and therefore
$$
\|u\|_{{\mathscr H}} = \|\opm(p(\sigma+i\gamma_0))u\|_H \lesssim \|\opm(p(\sigma+i\gamma))u\|_H + \|u\|_H
$$
for $u \in {\mathscr H}(\R_+;E_1)$. This proves the first part of the proposition.

To prove the claim about the adjoint we consider
$$
{\mathscr P}_{\gamma}(\sigma) = \begin{bmatrix} 0 & p(\sigma+i\gamma) \\ p(\sigma-i(\gamma+m)) & 0 \end{bmatrix} :
\begin{array}{c} E_1 \\ \oplus \\ E_1 \end{array} \subset \begin{array}{c} E_0 \\ \oplus \\ E_0 \end{array} \to \begin{array}{c} E_0 \\ \oplus \\ E_0 \end{array}.
$$
Then ${\mathscr P}_{\gamma}(\sigma) = {\mathscr P}_{\gamma}(\overline{\sigma})^{\star}$ by \eqref{Symmetry}, and
$$
\opm({\mathscr P}_{\gamma}) : \begin{array}{c} {\mathscr H}(\R_+;E_1) \\ \oplus \\ {\mathscr H}(\R_+;E_1) \end{array} \subset
L^2_b\biggl(\R_+;\begin{array}{c} E_0 \\ \oplus \\ E_0 \end{array}\biggr) \to L^2_b\biggl(\R_+;\begin{array}{c} E_0 \\ \oplus \\ E_0 \end{array}\biggr)
$$
is symmetric and closed. We need to show that this operator is selfadjoint, which follows if we show that both
$$
\opm({\mathscr P}_{\gamma}) \pm i : \begin{array}{c} {\mathscr H}(\R_+;E_1) \\ \oplus \\ {\mathscr H}(\R_+;E_1) \end{array} \to L^2_b\biggl(\R_+;\begin{array}{c} E_0 \\ \oplus \\ E_0 \end{array}\biggr)
$$
are invertible. In case $\gamma = \gamma_0$ the operator $\opm({\mathscr P}_{\gamma_0})$ is invertible by choice of $\gamma_0$ and \eqref{Fredholm} with inverse $\opm({\mathscr P}_{\gamma_0}^{-1})$, and so $\opm({\mathscr P}_{\gamma_0})$ is selfadjoint. For general $\gamma$ the operators ${\mathscr P}_{\gamma}(\sigma) \pm i : E_1^2 \to E_0^2$ are invertible for all $\sigma \in \R$, and $\bigl\|\bigl[{\mathscr P}_{\gamma}(\sigma) \pm i\bigr]^{-1}\bigr\|_{\L(E_0^2)} \leq 1$ by the Spectral Theorem. Writing
$$
\bigl[{\mathscr P}_{\gamma}(\sigma) \pm i\bigr]^{-1} = {\mathscr P}_{\gamma}(\sigma)^{-1}\Bigl(1 \mp i\bigl[{\mathscr P}_{\gamma}(\sigma) \pm i\bigr]^{-1}\Bigr)
$$
for large $|\sigma| \gg 0$ and Lemma~\ref{BasicHolom} show that
\begin{gather*}
\bigl[{\mathscr P}_{\gamma}(\sigma) \pm i\bigr]^{-1} \in S^0(\R;\L(E_0^2,E_1^2))\cap S^{-\mu}(\R;\L(E_0^2)), \\
{\mathscr P}_{\gamma_0}(\sigma)\bigl[{\mathscr P}_{\gamma}(\sigma) \pm i\bigr]^{-1} \in S^0(\R;\L(E_0^2)).
\end{gather*}
Thus
$$
\opm(\bigl[{\mathscr P}_{\gamma}(\sigma) \pm i\bigr]^{-1}) : L^2_b(\R_+;E_0^2) \to {\mathscr H}(\R_+;E_1^2)
$$
inverts $\opm({\mathscr P}_{\gamma}) \pm i$ and the proposition is proved.
\end{proof}

We next prove that the space ${\mathscr H}(\R_+;E_1)$ allows localization and microlocalization and has good approximation properties. The latter is based on the following familiar continuity criterion for the strong operator topology in $L^2_b$-spaces:

Let $S^0(\R_+\times\R;\L(F,G))$ denote the space of global operator-valued Mellin symbols of order $0$, where $F$ and $G$ and Hilbert spaces, i.e., $a(x,\sigma) \in C^{\infty}(\R_+\times\R;\L(F,G))$ belongs to $S^0(\R_+\times\R;\L(F,G))$ if and only if for all $\alpha,\beta \in \N_0$ there exists a constant $C_{\alpha,\beta} \geq 0$ independent of $(x,\sigma)$ such that
$$
\|(xD_x)^{\alpha}\partial_{\sigma}^{\beta}a(x,\sigma)\|_{\L(F,G)} \leq C_{\alpha,\beta} \langle \sigma \rangle^{-\beta}.
$$
The best constants in these estimates give rise to the seminorms for the symbol topology on $S^0(\R_+\times\R;\L(F,G))$.

If $\{a_j(x,\sigma)\}_{j=1}^{\infty} \subset S^0(\R_+\times\R;\L(F,G))$ is a sequence that is bounded in the symbol topology and converges pointwise on $\R_+\times\R$ to $a(x,\sigma) \in S^0(\R_+\times\R;\L(F,G))$, then $\opm(a_j) \to \opm(a)$ strongly as $j \to \infty$, i.e., $\opm(a_j)u \to \opm(a)u$ in $L^2_b(\R_+;G)$ as $j \to \infty$ for every $u \in L^2_b(\R_+;F)$.

We quickly review the steps of the proof of this criterion (without loss of generality assume that $a(x,\sigma) \equiv 0$):
\begin{itemize}
\item Let $u \in \T(\R_+;F)$ (recall that this means that $u(e^t) \in \S(\R;F)$):
\begin{itemize}
\item From the boundedness of the $\{a_j(x,\sigma)\}$ and dominated convergence we first get that $[\opm(a_j)u](x) \to 0$ in $G$ pointwise on $\R_+$.
\item As $S^0(\R_+\times\R;\L(F,G)) \ni a \mapsto \opm(a)u \in \T(\R_+;G)$ is continuous we can bound any continuous seminorm of $\opm(a_j)u$ on $\T(\R_+;G)$ by a constant independent of $j$. This observation gives rise to an integrable majorant of the form $F(x) = K\langle \log x \rangle^{-2} \in L^1(\R_+;\frac{dx}{x})$ such that $\|\opm(a_j)u(x)\|^2_G \leq F(x)$, and in conclusion $\opm(a_j)u \to 0$ in $L^2_b(\R_+;G)$ as $j \to \infty$ by dominated convergence.
\end{itemize}
\item Continuity of the map
$$
S^0(\R_+\times\R;\L(F,G)) \ni a \mapsto \opm(a) \in \L(L^2_b(\R_+;F),L^2_b(\R_+;G))
$$
shows that $\{\opm(a_j)\}_{j=1}^{\infty}$ is bounded in $\L(L^2_b(\R_+;F),L^2_b(\R_+;G))$, and $\opm(a_j)u \to 0$ as $j \to \infty$ for $u$ in the dense subspace $\T(\R_+;F)$. Thus $\opm(a_j) \to 0$ strongly as $j \to \infty$.
\end{itemize}

The relevant application of this criterion for us are mollifiers in the case $F=G$ and scalar symbols $a_j(x,\sigma)$, where we specifically choose
\begin{equation}\label{basicmollifier}
a_j(x,\sigma) = \phi(\sqrt[j]{x})\exp\bigl(-\tfrac{\sigma^2}{2j^2}\bigr)
\end{equation}
with $\phi \in C_c^{\infty}(\R_+)$ such that $\phi \equiv 1$ near $x = 1$. Note that $a_j(x,\sigma)$ is holomorphic in $\sigma$, $\{a_j(x,\sigma+i\gamma);\; j \in \N,\; |\gamma| \leq K\} \subset S^0(\R_+\times\R_{\sigma})$ is bounded for every $K > 0$, and $a_j(x,\sigma) \to 1$ as $j \to \infty$ pointwise on $\R_+\times\C$.
An application of the above continuity criterion combined with the analyticity and estimates on the symbols with respect to $\sigma$ yield that for every $\alpha \in \R$ and every $u \in x^{\alpha}L^2_b(\R;F)$ we have $\opm(a_j)u \to u$ as $j \to \infty$ in $x^{\alpha}L^2_b(\R;F)$. Note that the extensions by continuity of $\opm(a_j)$ to weighted $L^2_b$-spaces for different values of $\alpha, \alpha' \in \R$ are consistent on the intersection $x^{\alpha}L^2_b\cap x^{\alpha'}L^2_b$, and thus $\opm(a_j)$ is unambiguously defined. Observe also that $\opm(a_j)u \in C_c^{\infty}(\R_+;F)$ for every $j \in \N$.

\begin{proposition}\label{MicrolocalCutoff}
Let $a(x,\sigma) \in S^0(\R_+\times\R)$ be a global scalar Mellin symbol. Then
$$
\opm(a) : {\mathscr H}(\R_+;E_1) \to {\mathscr H}(\R_+;E_1)
$$
is continuous. Moreover, if $\{a_j(x,\sigma)\} \subset S^0(\R_+\times\R)$ is a sequence of such symbols that is bounded in the global symbol topology, and such that $a_j(x,\sigma)$ converges pointwise to the global Mellin symbol $a(x,\sigma) \in S^0(\R_+\times\R)$, then $\opm(a_j) \to \opm(a)$ strongly in $\L({\mathscr H}(\R_+;E_1))$, i.e., $\opm(a_j)u \to \opm(a)u$ in ${\mathscr H}(\R_+;E_1)$ for each $u \in {\mathscr H}(\R_+;E_1)$.
\end{proposition}
\begin{proof}
To prove the first claim we need to show by definition of ${\mathscr H}(\R_+;E_1)$ that
$$
\opm(p(\sigma+i\gamma_0)\opm(a)\opm(p(\sigma+i\gamma_0)^{-1}) : L_b^2(\R_+;E_0) \to L_b^2(\R_+;E_0)
$$
is bounded. Now
$$
\opm(p(\sigma+i\gamma_0)\opm(a)\opm(p(\sigma+i\gamma_0)^{-1}) = \opm([p(\sigma+i\gamma_0]\# [ap(\sigma+i\gamma_0)^{-1}]),
$$
where
\begin{equation}\label{MollAux}
\begin{aligned}
\relax[p(\sigma+i\gamma_0)] &\# [ap(\sigma+i\gamma_0)^{-1}] \\
&= \sum\limits_{k=0}^{\mu}\frac{1}{k!}\bigl(\partial_{\sigma}^kp(\sigma+i\gamma_0\bigr)\bigl([(xD_x)^ka]p(\sigma+i\gamma_0)^{-1}\bigr) \\
&= \sum\limits_{k=0}^{\mu}\frac{1}{k!}\bigl([\partial_{\sigma}^kp(\sigma+i\gamma_0)]p(\sigma+i\gamma_0)^{-1}\bigr)\bigl((xD_x)^ka(x,\sigma)\bigr).
\end{aligned}
\end{equation}
Lemma~\ref{BasicHolom} now implies that
$$
[p(\sigma+i\gamma_0)]\# [ap(\sigma+i\gamma_0)^{-1}] \in S^0(\R_+\times\R;\L(E_0)),
$$
and consequently the Mellin pseudodifferential operator
$$
\opm\bigl([p(\sigma+i\gamma_0)]\# [ap(\sigma+i\gamma_0)^{-1}]\bigr) : L_b^2(\R_+;E_0) \to L_b^2(\R_+;E_0)
$$
is bounded.

To prove the convergence statement note that boundedness of the $a_j(x,\sigma)$ in the symbol topology implies, in particular, that the $\{a_j\}$ form a bounded sequence with respect to the standard Fr{\'e}chet topology of $C^{\infty}(\R_+\times\R)$. The Montel property of $C^{\infty}(\R_+\times\R)$ together with the pointwise convergence of the $a_j$ then implies that $a_j \to a$ in $C^{\infty}(\R_+\times\R)$. Then
$$
\{[p(\sigma+i\gamma_0)]\# [a_jp(\sigma+i\gamma_0)^{-1}]\}_{j=1}^{\infty} \subset S^0(\R_+\times\R;\L(E_0))
$$
is bounded and
$$
[p(\sigma+i\gamma_0)]\# [a_jp(\sigma+i\gamma_0)^{-1}] \to [p(\sigma+i\gamma_0)]\# [a p(\sigma+i\gamma_0)^{-1}]
$$
pointwise by \eqref{MollAux}, and so
$$
\opm\bigl([p(\sigma+i\gamma_0)]\# [a_jp(\sigma+i\gamma_0)^{-1}]\bigr)\to \opm\bigl([p(\sigma+i\gamma_0)]\# [a p(\sigma+i\gamma_0)^{-1}]\bigr)
$$
strongly in $\L(L^2_b(\R_+;E_0))$, thus proving that $\opm(a_j) \to \opm(a)$ strongly in $\L({\mathscr H}(\R_+;E_1))$.
\end{proof}

\begin{corollary}\label{densityintersectionspaces}
$C_c^{\infty}(\R_+;E_1)$ is dense in $x^{\gamma}{\mathscr H}(\R_+;E_1) \cap x^{\alpha}L^2_b(\R_+;E_0)$ for all $\alpha,\gamma \in \R$.
\end{corollary}
\begin{proof}
Let $a_j(x,\sigma)$ be the sequence of mollifying symbols from \eqref{basicmollifier}. For every $u \in x^{\gamma}{\mathscr H}(\R_+;E_1) \cap x^{\alpha}L^2_b(\R_+;E_0)$ we have $\opm(a_j)u \to u$ as $j \to \infty$ in $x^{\alpha}L^2_b(\R_+;E_0)$, and $\opm(a_j)u = x^{\gamma}\opm(a_j(x,\sigma-i\gamma))(x^{-\gamma}u) \to u$ as $j \to \infty$ in $x^{\gamma}{\mathscr H}(\R_+;E_1)$ by Proposition~\ref{MicrolocalCutoff}.
\end{proof}

\begin{theorem}\label{MinimalDomainThm}
We have
$$
x^m{\mathscr H}(\R_+;E_1)\cap L^2_b(\R_+;E_0) \hookrightarrow \Dom_{\min},
$$
and $\Dom_{\min} = x^m{\mathscr H}(\R_+;E_1)\cap L^2_b(\R_+;E_0)$ if and only if $p(\sigma-im) : E_1 \to E_0$ is invertible for all $\sigma \in \R$.
\end{theorem}
\begin{proof}
The operator $A : C_c^{\infty}(\R_+;E_1) \subset L^2_b(\R_+;E_0) \to L^2_b(\R_+;E_0)$ extends to a bounded operator $A : x^m{\mathscr H}(\R_+;E_1) \to L^2_b(\R_+;E_0)$ via
\begin{equation}\label{AonCc}
Au = \opm(p(\sigma - im))(x^{-m}u)
\end{equation}
for $u \in x^m{\mathscr H}(\R_+;E_1)$ by Proposition~\ref{BaseSpaceUnbdd}. Consequently, writing $H = L^2_b(\R_+;E_0)$, we have
$$
\|u\|_H + \|Au\|_H \lesssim \|u\|_H + \|u\|_{x^m{\mathscr H}}
$$
for $u \in C_c^{\infty}(\R_+;E_1)$ which proves that $x^m{\mathscr H}(\R_+;E_1)\cap L^2_b(\R_+;E_0) \hookrightarrow \Dom_{\min}$ by Corollary~\ref{densityintersectionspaces}.

If $p(\sigma-im) : E_1 \to E_0$ is invertible for all $\sigma \in \R$, Proposition~\ref{BaseSpaceProps} and \eqref{AonCc} imply that $\|Au\|_{H}$ is an equivalent norm on $x^m{\mathscr H}(\R_+;E_1)$. Consequently, the graph norm $\|u\|_H + \|Au\|_H$ and the norm $\|u\|_H + \|u\|_{x^m{\mathscr H}}$ are equivalent on $C_c^{\infty}(\R_+;E_1)$, and therefore $\Dom_{\min} = x^m{\mathscr H}(\R_+;E_1)\cap L^2_b(\R_+;E_0)$ by the density of $C_c^{\infty}(\R_+;E_1)$ in both spaces.

Conversely, if there exists $\sigma_0 \in \C$ with $\Im(\sigma_0) = -m$ such that $p(\sigma_0) : E_1 \to E_0$ is not invertible, we will show in the proof of Theorem~\ref{DmaxThm} in Section~\ref{MaximalDomain} that $\Dom_{\min}$ contains elements of the form $u = \omega e_1x^{i\sigma_0}$, where $0 \neq e_1 \in E_1$ and $\omega \in C_c^{\infty}(\overline{\R}_+)$ is a function with $\omega \equiv 1$ near $x = 0$. Such functions $u$ are not contained in $x^mL^2_b(\R_+;E_0)$ and thus $u \notin x^m{\mathscr H}(\R_+;E_1)$, and so $\Dom_{\min} \neq x^m{\mathscr H}(\R_+;E_1)\cap L^2_b(\R_+;E_0)$.
\end{proof}


\section{The maximal domain}\label{MaximalDomain}

\noindent
Fix $\omega \in C_c^{\infty}(\overline{\R}_+)$ with $\omega \equiv 1$ near $x = 0$. For each $\sigma_0 \in \spec_b(A)$ let
\begin{equation}\label{Esigma0def}
\begin{aligned}
\E_{\sigma_0} = \{u = \omega\sum\limits_{j=0}^k e_j&\log^j(x)x^{i\sigma_0};\; k \in \N_0 \textup{ and } e_j \in E_1, \\
&\textup{and } p(\sigma)(Mu)(\sigma) \textup{ is holomorphic at $\sigma=\sigma_0$}\}.
\end{aligned}
\end{equation}
By analytic Fredholm theory this space is finite-dimensional. Theorem~\ref{DmaxThm} describes the structure of the maximal domain.

\begin{theorem}\label{DmaxThm}
We have
$$
\Dom_{\max} = \Dom_{\min} \oplus \bigoplus\limits_{\substack{\sigma_0 \in \spec_b(A) \\ -m < \Im(\sigma_0) < 0}}\E_{\sigma_0}.
$$
In particular, $\dim\hat{\E}_{\max} < \infty$. For each $\sigma_0 \in \spec_b(A)$ with $-m < \Im(\sigma_0) < 0$ the space $\hat{\E}_{\sigma_0} = \E_{\sigma_0} + \Dom_{\min}$ is the generalized eigenspace to the eigenvalue $\sigma_0$ of the generator $\hat{\g}$ of the induced action $\hat{\kappa}_{\varrho}$ from \eqref{dilation} on $\hat{\E}_{\max}$.

If $u_{\sigma_j} \in \E_{\sigma_j}$ for $\sigma_j \in \spec_b(A)$ with $-m < \Im(\sigma_j) < 0$ for $j=0,1$, then the adjoint pairing between these functions is given by
\begin{equation}\label{ResidueFormulaPairing}
[u_{\sigma_0},u_{\sigma_1}]_A = \res_{\sigma=\sigma_0}\langle p(\sigma)[Mu_{\sigma_0}](\sigma),[Mu_{\sigma_1}](\sigma^{\star})\rangle_{E_0},
\end{equation}
where $\sigma^{\star}=\overline{\sigma}-im$ is reflection about the line $\Im(\sigma)=-\frac{m}{2}$.
\end{theorem}

This means that every $u \in \Dom_{\max}$ has an asymptotic expansion of the form
$$
u \sim \sum\limits_{\substack{\sigma_0 \in \spec_b(A) \\ -m < \Im(\sigma_0) < 0}}\sum\limits_{j=0}^{k_{\sigma_0}}e_{\sigma_0,j}\log^j(x)x^{i\sigma_0} \quad \textup{as } x \to 0 \mod \Dom_{\min},
$$
and vanishing conditions placed upon these asymptotic terms are boundary conditions as $x \to 0$ that determine extensions. In particular, the abstract theory of the first part of this paper is applicable to the indicial operator \eqref{Aindicial}.

The proof of Theorem \ref{DmaxThm} requires some auxiliary results.

\begin{lemma}\label{DmaxAux1}
We have
$$
{\mathscr H}(\R_+;E_1) \cap \Dom_{\max} = \{u \in {\mathscr H}(\R_+;E_1);\; \opm(p)u \in x^mL^2_b(\R_+;E_0)\cap L^2_b(\R_+;E_0)\},
$$
and $A_{\max}u = x^{-m}\opm(p)u$ for $u \in {\mathscr H}(\R_+;E_1) \cap \Dom_{\max}$.
\end{lemma}
\begin{proof}
Let $u \in {\mathscr H}(\R_+;E_1)$ and $v \in C_c^{\infty}(\R_+;E_1)$ be arbitrary. Using Plancherel and our standing assumptions we get
\begin{align*}
\langle Av,u \rangle &= \langle \opm\bigl(p(\cdot-im)\bigr)(x^{-m}v),u \rangle \\
&= \frac{1}{2\pi}\int_{\R}\langle p(\sigma-im)[M(x^{-m}v)](\sigma),[Mu](\sigma)\rangle_{E_0}\,d\sigma \\
&= \frac{1}{2\pi}\int_{\R}\langle [M(x^{-m}v)](\sigma),p(\sigma)[Mu](\sigma)\rangle_{E_0}\,d\sigma \\
&= \langle x^{-m}v,\opm(p)u \rangle = \int_0^{\infty} \langle v(x),[x^{-m}\opm(p)u](x) \rangle_{E_0}\,\frac{dx}{x}.
\end{align*}
Consequently, if $x^{-m}\opm(p)u \in L^2_b(\R_+;E_0)$, then $u \in \Dom_{\max}$ and the last integral can be rewritten as the pairing $\langle v,A_{\max}u \rangle$ with $A_{\max}u = x^{-m}\opm(p)u$. Conversely, if $u \in \Dom_{\max}$, we have $\langle Av,u \rangle = \langle v,A_{\max}u \rangle$, and so
$$
\int_0^{\infty} \langle v(x),[x^{-m}\opm(p)u](x) \rangle_{E_0}\,\frac{dx}{x} = \int_0^{\infty} \langle v(x),A_{\max}u(x) \rangle_{E_0}\,\frac{dx}{x}
$$
for all $v \in C_c^{\infty}(\R_+;E_1)$. This shows $x^{-m}\opm(p)u = A_{\max}u \in L^2_{\textup{loc}}(\R_+;E_0)$, and thus $x^{-m}\opm(p)u = A_{\max}u \in L^2_b(\R_+;E_0)$.
\end{proof}

\begin{lemma}\label{DmaxAux2}
For every $\sigma_0 \in \spec_b(A)$ with $\Im(\sigma_0) < 0$ we have $\E_{\sigma_0} \subset {\mathscr H}(\R_+;E_1)\cap\Dom_{\max}$.
\end{lemma}
\begin{proof}
We have $\E_{\sigma_0} \subset H^{\infty}_b(\R_+;E_1) \subset {\mathscr H}(\R_+;E_1)$, and for every $u \in \E_{\sigma_0}$ we have
$$
\opm(p)u \in C_c^{\infty}(\R_+;E_0) \subset x^mL^2_b(\R_+;E_0)\cap L^2_b(\R_+;E_0).
$$
The claim now follows from Lemma~\ref{DmaxAux1}.
\end{proof}

The following proposition is one of the main ingredients for the proof of Theorem~\ref{DmaxThm}. It relies on the pseudodifferential calculus in Appendix~\ref{PseudoCalculus}.

\begin{proposition}\label{kerAmaxpmi}
We have
$$
\ker(A_{\max} \pm i) \subset {\mathscr H}(\R_+;E_1) \cap \bigcap\limits_{\gamma \leq 0}x^{\gamma}H^{\infty}_b(\R_+;E_0).
$$
\end{proposition}
\begin{proof}
We consider $\ker(A_{\max} + i)$, the case of $\ker(A_{\max} - i)$ is analogous. Using the pseudodifferential calculus in Appendix~\ref{PseudoCalculus} we will show that there exists a continuous operator $Q : L^2_b(\R_+;E_0) \to \Dom_{\min} \hookrightarrow L^2_b(\R_+;E_0)$ such that
\begin{equation}\label{DminRightParam}
(A_{\min} - i)Q = 1 + G : L^2_b(\R_+;E_0) \to L^2_b(\R_+;E_0)
\end{equation}
with $G \in \Psi^{-\infty}_O(\R_+;\L(E_0))$, i.e.,
$$
G,\,G^* : x^{\alpha}H^s_b(\R_+;E_0) \to x^{\alpha'}H^{s'}_b(\R_+;E_0)
$$
for all $\alpha,s,s' \in \R$, and all $\alpha' \leq \alpha$, where the adjoints refer to the $L^2_b(\R_+;E_0)$-inner product. Passing to adjoints in \eqref{DminRightParam} then shows that
$$
Q^*(A_{\max} + i) = 1 + G^* : \Dom_{\max} \subset L^2_b(\R_+;E_0) \to L^2_b(\R_+;E_0).
$$
In particular, if $u \in \ker(A_{\max} + i)$, then $u = -G^*u \in \bigcap\limits_{\gamma \leq 0}x^{\gamma}H^{\infty}_b(\R_+;E_0)$ by the mapping properties of $G^*$. We proceed with the construction of $Q$.

Consider
$$
a(x,\sigma) = p(\sigma) - ix^m \in S^{m\mu;\vec{\ell}}_O(\R\times\C;\L(E_1,E_0)),
$$
where the anisotropy vector $\vec{\ell}$ is given by \eqref{ellchoices}. By our standing assumptions, $a(x,\sigma)$ is right-hypoelliptic of order $(m\mu,0)$ in the sense of Definition~\ref{hypoellipticity}. Thus there exists $b(x,\sigma) \in S^{0;\vec{\ell}}_O(\R\times\C;\L(E_0,E_1))$ with $a{\#}b \sim 1$, see Theorem~\ref{ParametrixHypoelliptic}. Now define
$$
Q = \opm(b(x,\sigma))x^m : \dot{C}^{\infty}(\R_+;E_0) \to \dot{C}^{\infty}(\R_+;E_1),
$$
where as in Appendix~\ref{PseudoCalculus} we denote by $\dot{C}^{\infty}$ the space of functions on $\R_+$ that vanish to infinite order at $x=0$ and are rapidly decreasing as $x \to \infty$.
Because $A_{\min} : \Dom_{\min} \subset L^2_b(\R_+;E_0) \to L^2_b(\R_+;E_0)$ is symmetric we have
$$
\|(A_{\min}-i)u\|_{L^2_b(\R_+;E_0)}^2 = \|A_{\min}u\|_{L^2_b(\R_+;E_0)}^2 + \|u\|_{L^2_b(\R_+;E_0)}^2, \quad u \in \Dom_{\min}.
$$
Now $Q : \dot{C}^{\infty}(\R_+;E_0) \to \dot{C}^{\infty}(\R_+;E_1)$ and $\dot{C}^{\infty}(\R_+;E_1) \hookrightarrow \Dom_{\min}$ with $A_{\min}u = x^{-m}\opm(p)u$ for $u \in \dot{C}^{\infty}(\R_+;E_1)$ by Theorem~\ref{MinimalDomainThm}. Thus
\begin{gather*}
(A_{\min}-i)Q = x^{-m}(\opm(a(x,\sigma))\opm(b(x,\sigma))x^m \\
= x^{-m}(1 + \tilde{G})x^m = 1 + G : \dot{C}^{\infty}(\R_+;E_0) \to \dot{C}^{\infty}(\R_+;E_0)
\end{gather*}
by Theorem~\ref{ParametrixHypoelliptic}, and because this operator extends to a continuous operator $L^2_b(\R_+;E_0) \to L^2_b(\R_+;E_0)$ we obtain that $Q : L^2_b(\R_+;E_0) \to \Dom_{\min}$ is bounded with \eqref{DminRightParam} as asserted.

It remains to prove that $\ker(A_{\max} + i) \subset {\mathscr H}(\R_+;E_1)$. Let $u \in \ker(A_{\max} + i)$ be arbitrary. Then
$$
x^mA_{\max}u = -ix^mu \in L^2_b(\R_+;E_0)
$$
by what we have already shown. For all $v \in C_c^{\infty}(\R_+;E_1)$ we have
$$
\langle v,x^mA_{\max}u \rangle = \langle x^mv,A_{\max}u \rangle = \langle A(x^mv),u \rangle = \langle \opm(p(\sigma-im))v,u \rangle.
$$
Consequently, $u$ is in the domain of the adjoint of the closure of $\opm(p(\sigma-im)) : C_c^{\infty}(\R_+;E_1) \subset L^2_b(\R_+;E_0) \to L^2_b(\R_+;E_0)$, and thus $u \in {\mathscr H}(\R_+;E_1)$ by Proposition~\ref{BaseSpaceUnbdd}. The proposition is proved.
\end{proof}

\begin{lemma}\label{KerAsymptotics}
Let $u \in \ker(A_{\max} \pm i)$. Then there exist $u_{\sigma_0} \in \E_{\sigma_0}$, $\sigma_0 \in \spec_b(A)$ with $-m \leq \Im(\sigma_0) < 0$, such that
$$
u - \sum\limits_{\substack{\sigma_0 \in \spec_b(A) \\ -m \leq \Im(\sigma_0) < 0}}u_{\sigma_0} \in \Dom_{\min}.
$$
\end{lemma}
\begin{proof}
This follows via the standard argument to establish asymptotic expansions utilizing the Mellin transform. Without loss of generality we consider $u \in \ker(A_{\max}-i)$ in this proof. Then
$$
\opm(p)u = ix^mu \in \bigcap\limits_{\alpha \geq 0}x^{m-\alpha}L^2_b(\R_+;E_0) \subset \bigcap\limits_{\alpha > 0}x^{m-\alpha}L^1_b(\R_+;E_0)
$$
by Lemma~\ref{DmaxAux1} and Proposition~\ref{kerAmaxpmi}. In particular, the Mellin transform $M(ix^mu)(\sigma)$ extends to an analytic $E_0$-valued function in the half-plane $\Im(\sigma) > -m$, and the function $\R \ni \sigma \mapsto M(ix^mu)(\sigma+i\gamma)$ belongs to $L^2(\R_{\sigma};E_0)\cap C_0(\R_{\sigma};E_0)$ with continuous dependence on $\gamma > -m$. Now
$$
Mu(\sigma) =p(\sigma)^{-1}M(ix^mu)(\sigma),
$$
which is a priori analytic in $\Im(\sigma) > 0$ and via this identity extends meromorphically to $\Im(\sigma) > -m$ with possible locations of poles at points in $\spec_b(A)$, and because $Mu(\sigma) \in L^2(\R;E_0)$ this function cannot have poles on $\R$. Consequently, there exists $\varepsilon > 0$ such that $Mu(\sigma)$ is analytic in $\Im(\sigma) > -\varepsilon$ with $Mu(\sigma+i\gamma) \in L^2(\R_{\sigma};E_0)\cap C_0(\R_{\sigma};E_0)$ for $\gamma > -\varepsilon$, so $u \in \bigcap\limits_{\alpha > 0}x^{\varepsilon - \alpha}L^2_b(\R_+;E_0)$ and therefore
$$
\opm(p)u = ix^mu \in \bigcap\limits_{\alpha > 0}x^{m+\varepsilon-\alpha}L^2_b(\R_+;E_0) \subset \bigcap\limits_{\alpha > 0}x^{m+\varepsilon-\alpha}L^1_b(\R_+;E_0).
$$
This shows that $Mu(\sigma)$ extends meromorphically further to $\Im(\sigma) > -m - \varepsilon$, and by choosing $\varepsilon > 0$ small enough we can assume that $\spec_b(A)\cap\{\sigma \in \C;\; -m - \varepsilon < \Im(\sigma) < -m\} = \emptyset$, so all possible poles for $Mu(\sigma)$ in $\Im(\sigma) > -m-\varepsilon$ are located in $-m \leq \Im(\sigma) < 0$. Note also that $Mu(\sigma)$ is an $E_1$-valued meromorphic function in $\Im(\sigma) > -m - \varepsilon$. Consequently, there exist $u_{\sigma_0} \in \E_{\sigma_0}$, $\sigma_0 \in \spec_b(A)$ with $-m \leq \Im(\sigma_0) < 0$, such that both
$$
M\Bigl[u - \sum\limits_{\substack{\sigma_0 \in \spec_b(A) \\ -m \leq \Im(\sigma_0) < 0}}u_{\sigma_0}\Bigr](\sigma) \quad \textup{and} \quad
p(\sigma)M\Bigl[u - \sum\limits_{\substack{\sigma_0 \in \spec_b(A) \\ -m \leq \Im(\sigma_0) < 0}}u_{\sigma_0}\Bigr](\sigma)
$$
are holomorphic in $\Im(\sigma) > -m-\varepsilon$, and along every line $\Im(\sigma) = \gamma$ these functions are $L^2\cap C_0$ with values in $E_1$ and $E_0$, respectively, with continuous dependence on $\gamma > -m-\varepsilon$. Now choose $-m-\varepsilon < \gamma_1 < -m < 0 < \gamma_2$ such that $p(\sigma) : E_1 \to E_0$ is invertible along $\Im(\sigma) = \gamma_j$, $j=1,2$. From the above we then obtain that
\begin{align*}
u - \sum\limits_{\substack{\sigma_0 \in \spec_b(A) \\ -m \leq \Im(\sigma_0) < 0}}u_{\sigma_0} &\in x^{-\gamma_1}{\mathscr H}(\R_+;E_1)\cap x^{-\gamma_2}{\mathscr H}(\R_+;E_1) \\
&= \bigcap\limits_{-\gamma_2 \leq \gamma \leq -\gamma_1}x^{\gamma}{\mathscr H}(\R_+;E_1).
\end{align*}
Note that the last equality is true because of Proposition~\ref{MicrolocalCutoff}. Thus
$$
u - \sum\limits_{\substack{\sigma_0 \in \spec_b(A) \\ -m \leq \Im(\sigma_0) < 0}}u_{\sigma_0} \in x^m{\mathscr H}(\R_+;E_1)\cap L^2_b(\R_+E_0) \subset \Dom_{\min}
$$
by Theorem~\ref{MinimalDomainThm}. The lemma is proved.
\end{proof}

\begin{proof}[Proof of Theorem~\ref{DmaxThm}]
By von Neumann's formulas \eqref{vonNeumannFormula} and Lemmas~\ref{DmaxAux2} and \ref{KerAsymptotics} we have
$$
\Dom_{\max} = \Dom_{\min} + \sum\limits_{\substack{\sigma_0 \in \spec_b(A) \\ -m \leq \Im(\sigma_0) < 0}}\E_{\sigma_0}.
$$
In particular $\dim\hat{\E}_{\max} < \infty$ and $A$ has finite deficiency indices, so the abstract theory from the first part of the paper is applicable.

Let $\g = xD_x$ be the generator of the scaling action $\kappa_{\varrho}$ from \eqref{dilation}. We have
$$
\g : \E_{\sigma_0} + C_c^{\infty}(\R_+;E_1) \to \E_{\sigma_0} + C_c^{\infty}(\R_+;E_1)
$$
for every $\sigma_0 \in \spec_b(A)$ per the defining relation \eqref{Esigma0def}, and
$$
(\g - \sigma_0)^N : \E_{\sigma_0} + C_c^{\infty}(\R_+;E_1) \to C_c^{\infty}(\R_+;E_1)
$$
for $N$ large enough\footnote{Note that the action of $\g$ corresponds to multiplication by $\sigma$ on the Mellin transform side.}. This shows that $\hat{\E}_{\sigma_0} = \E_{\sigma_0} + \Dom_{\min} \subset \hat{\E}_{\max}$ is the generalized eigenspace associated with the eigenvalue $\sigma_0 \in \spec_b(A)$ for the generator $\hat{\g} : \hat{\E}_{\max} \to \hat{\E}_{\max}$. In particular, $\hat{\g}$ has no real eigenvalues, and by the Canonical Form Theorem from Section~\ref{GreenAbstract} then also has no eigenvalue with $\Im(\sigma) = -m$. This shows that $\hat{\E}_{\sigma_0} \subset \Dom_{\min}$ for every $\sigma_0 \in \spec_b(A)$ with $\Im(\sigma_0) = -m$, and we get
\begin{equation}\label{DmaxSumAux}
\Dom_{\max} = \Dom_{\min} + \sum\limits_{\substack{\sigma_0 \in \spec_b(A) \\ -m < \Im(\sigma_0) < 0}}\E_{\sigma_0}.
\end{equation}
We next show that this sum is direct. The argument is based on formula \eqref{ResidueFormulaPairing} for the adjoint pairing. To prove this formula we use Plancherel and write
\begingroup
\allowdisplaybreaks
\begin{align*}
\frac{1}{i}\langle A_{\max}u_{\sigma_0},u_{\sigma_1} \rangle &= \frac{1}{2\pi i}\int_{\R}\langle p(\sigma-im)[Mu_{\sigma_0}](\sigma-im),[Mu_{\sigma_1}](\sigma)\rangle_{E_0}\,d\sigma \\
&= \frac{1}{2\pi i}\int_{\Im(\sigma)=-m}\langle p(\sigma)[Mu_{\sigma_0}](\sigma),[Mu_{\sigma_1}](\sigma^{\star})\rangle_{E_0}\,d\sigma \\
&= \frac{1}{2\pi i}\int_{\Im(\sigma)=-m}\langle [Mu_{\sigma_0}](\sigma),p(\sigma^{\star})[Mu_{\sigma_1}](\sigma^{\star})\rangle_{E_0}\,d\sigma, \\
\frac{1}{i}\langle u_{\sigma_0},A_{\max}u_{\sigma_1} \rangle &= \frac{1}{2\pi i}\int_{\R}\langle [Mu_{\sigma_0}](\sigma),p(\sigma-im)[Mu_{\sigma_1}](\sigma-im)\rangle_{E_0}\,d\sigma \\
&= \frac{1}{2\pi i}\int_{\R}\langle [Mu_{\sigma_0}](\sigma),p(\sigma^{\star})[Mu_{\sigma_1}](\sigma^{\star})\rangle_{E_0}\,d\sigma.
\end{align*}
\endgroup
The function $\langle [Mu_{\sigma_0}](\sigma),p(\sigma^{\star})[Mu_{\sigma_1}](\sigma^{\star})\rangle_{E_0}$ is meromorphic on $\C$ with a possible pole only at $\sigma = \sigma_0$, and it is rapidly decreasing as $|\Re(\sigma)| \to \infty$ locally uniformly with respect to $\Im(\sigma)$. Consequently,
\begin{align*}
[u_{\sigma_0},u_{\sigma_1}]_A &= \frac{1}{i}\bigl[\langle A_{\max}u_{\sigma_0},u_{\sigma_1} \rangle - \langle u_{\sigma_0},A_{\max}u_{\sigma_1} \rangle\bigr] \\
&= \frac{1}{2\pi i}\oint_{C_{\varepsilon}(\sigma_0)} \langle [Mu_{\sigma_0}](\sigma),p(\sigma^{\star})[Mu_{\sigma_1}](\sigma^{\star})\rangle_{E_0}\,d\sigma \\
&= \frac{1}{2\pi i}\oint_{C_{\varepsilon}(\sigma_0)} \langle p(\sigma)[Mu_{\sigma_0}](\sigma),[Mu_{\sigma_1}](\sigma^{\star})\rangle_{E_0}\,d\sigma \\
&= \res_{\sigma=\sigma_0}\langle p(\sigma)[Mu_{\sigma_0}](\sigma),[Mu_{\sigma_1}](\sigma^{\star})\rangle_{E_0}
\end{align*}
after shifting the integration contour from the two lines $\R$ and $\Im(\sigma)=-m$ to a small positively oriented circle $C_{\varepsilon}(\sigma_0)$ centered at $\sigma_0$, showing \eqref{ResidueFormulaPairing}.  In particular, $[u_{\sigma_0},u_{\sigma_1}]_A = 0$ for $\sigma_1 \neq \sigma_0^{\star}$. In case $\sigma_1 = \sigma_0^{\star}$ consider
$$
\P(\sigma) = \begin{bmatrix} 0 & p(\sigma_0+\sigma) \\ p(\sigma_0^{\star} + \sigma) & 0 \end{bmatrix} : \begin{array}{c} E_1 \\ \oplus \\ E_1 \end{array} \to \begin{array}{c} E_0 \\ \oplus \\ E_0 \end{array}
$$
near $\sigma = 0$. By Theorem~\ref{Nondeg} the pairing on the space $\K(\P)$ of meromorphic germs at $\sigma=0$ that are annihilated by $\P(\sigma)$ modulo holomorphic germs discussed in Appendix~\ref{ModelPairing} is nondegenerate, but we have $\K(\P) \cong \E_{\sigma_0^{\star}} \oplus \E_{\sigma_0}$ and by what we have just shown the pairing on $\K(\P)$ under this isomorphism is expressed by
$$
\bigl[(u_{\sigma_0^{\star}},v_{\sigma_0}),(v_{\sigma_0^{\star}},u_{\sigma_0})\bigr] \mapsto [v_{\sigma_0},v_{\sigma_0^{\star}}]_A + [u_{\sigma_0^{\star}},u_{\sigma_0}]_A.
$$
Consequently, $[\cdot,\cdot]_A : \E_{\sigma_0} \times \E_{\sigma_0^{\star}} \to \C$ is nondegenerate. Using these pairings and their nondegeneracy for $\sigma_1 = \sigma_0^{\star}$ and orthogonality for $\sigma_1 \neq \sigma_0^{\star}$ then shows that the sum \eqref{DmaxSumAux} must be direct. The theorem is proved.
\end{proof}


\section{The signature of the adjoint pairing and the sign condition}\label{GreenIndicial}

\noindent
By Theorem~\ref{DmaxThm} the abstract theory from the first part of the paper applies to indicial operators \eqref{Aindicial}. In particular, the Canonical Form Theorem implies an algebraic Green formula for the pairing $[\cdot,\cdot]_{\hat{\E}_{\max}}$ as discussed in Section~\ref{GreenAbstract}. For indicial operators the Mellin transform and indicial family provide an additional analytic counterpart that is not available in the operator theoretical setting of the first part of this paper that we can harness to prove an analytic formula for the signature of the pairing $[\cdot,\cdot]_{\hat{\E}_{\max}}$ in terms of the spectral flow of the indicial family $p(\sigma) : E_1 \subset E_0 \to E_0$, $\Im(\sigma)=-\frac{m}{2}$. By our standing assumptions this is a family of selfadjoint unbounded Fredholm operators, invertible everywhere except at finitely many indicial roots. It therefore makes sense to consider the spectral flow across each indicial root separately. We refer to the survey by Lesch \cite{LeschSF} for general information about the spectral flow, see also \cite{BoossFurutani} and Appendix~\ref{ModelPairing}. For an indicial operator we therefore obtain both an algebraic formula for the signature of the adjoint pairing as in Theorem~\ref{SelfadjointExtensions} in terms of the invariants of the generator on generalized eigenspaces with $\Im(\sigma_0) = -\frac{m}{2}$, and an analytic formula in terms of the indicial family.

\begin{theorem}
The signature of $\bigl(\hat{\E}_{\max},[\cdot,\cdot]_{\hat{\E}_{\max}}\bigr)$ is given by
$$
\sig\bigl(\hat{\E}_{\max},[\cdot,\cdot]_{\hat{\E}_{\max}}\bigr) = \sum\limits_{\substack{\sigma_0 \in \spec_b(A) \\ \Im(\sigma_0) = -\frac{m}{2}}}\sum\limits_{\ell \; \textup{odd}}\bigl(m_+(\sigma_0,\ell) - m_-(\sigma_0,\ell)\bigr).
$$
For every $\sigma_0 \in \spec_b(A)$ with $\Im(\sigma_0) = -\frac{m}{2}$ we have
$$
\sum\limits_{\ell \; \textup{odd}}\bigl(m_+(\sigma_0,\ell) - m_-(\sigma_0,\ell)\bigr) = \SF_{\sigma=\sigma_0}\bigl[p(\sigma) : E_1 \subset E_0 \to E_0,\; \Im(\sigma)=-\tfrac{m}{2}\bigr].
$$
In particular,
$$
\sig\bigl(\hat{\E}_{\max},[\cdot,\cdot]_{\hat{\E}_{\max}}\bigr) = \SF\bigl[p(\sigma) : E_1 \subset E_0 \to E_0,\; \Im(\sigma)=-\tfrac{m}{2}\bigr].
$$
\end{theorem}
\begin{proof}
The algebraic statement follows from Theorem~\ref{SelfadjointExtensions}. The analytic formula for the signature in terms of the spectral flow is based on formula \eqref{ResidueFormulaPairing} and is discussed in the appendix; specifically, this is Theorem~\ref{Nondeg}. Note that the action of $\hat{\g}$ is transformed to multiplication by $\sigma$ on the Mellin transform side.
\end{proof}

If $\langle Au,u \rangle \geq 0$ for $u \in C_c^{\infty}(\R_+;E_1)$ we have seen in Sections~\ref{FriedrichsAbstract} and \ref{KreinAbstract} that the sign condition is needed to give satisfactory classifications of the Friedrichs and Krein extensions, see Definition~\ref{SignCondition}. For indicial operators this condition is satisfied.

\begin{theorem}\label{SignConditionIndicial}
Suppose $\langle Au,u \rangle \geq 0$ for $u \in C_c^{\infty}(\R_+;E_1)$. Then the invariants $m_{\pm}(\sigma_0,\ell)$ of $\bigl(\hat{\E}_{\max},[\cdot,\cdot]_{\hat{\E}_{\max}},\hat{\g}\bigr)$ for every $\sigma_0 \in \spec_b(A)$ with $\Im(\sigma_0) = -\frac{m}{2}$ satisfy $m_{\pm}(\sigma_0,\ell) = 0$ for $\ell$ odd, and $m_-(\sigma_0,\ell) = 0$ for $\ell$ even.
\end{theorem}
\begin{proof}
The assumption $\langle Au,u \rangle \geq 0$ for $u \in C_c^{\infty}(\R_+;E_1)$ is equivalent to $p(\sigma) \geq 0$ for $\Im(\sigma) = -\frac{m}{2}$. The theorem thus follows from Theorem~\ref{Positivity}.
\end{proof}


\begin{appendix}

\section{A pseudodifferential calculus}\label{PseudoCalculus}

\noindent
Let $\vec{\ell} = (\ell_1,\ell_2) \in \N^2$, and define
\begin{equation}\label{JapaneseBracket}
\langle x,\sigma \rangle_{\vec{\ell}} = (1 + x^{2\ell_2} + \sigma^{2\ell_1})^{\frac{1}{2\ell_1\ell_2}}
\end{equation}
 for $(x,\sigma) \in \R^2$. Note that Peetre's inequality
 $$
 \langle x+x',\sigma+\sigma' \rangle_{\vec{\ell}}^s \leq 2^{|s|}\langle x,\sigma \rangle_{\vec{\ell}}^{s}\,\langle x',\sigma' \rangle_{\vec{\ell}}^{|s|}
 $$
 holds, and there exist $c,C > 0$ such that
 $$
 c\langle x,\sigma \rangle^{\frac{1}{\ell_1+\ell_2}} \leq \langle x,\sigma \rangle_{\vec{\ell}} \leq C \langle x,\sigma \rangle^{\ell_1+\ell_2}.
 $$ 
 We consider anisotropic symbols $a_0(x,\sigma)$ taking valued in the bounded operators between Hilbert spaces that are based on \eqref{JapaneseBracket} such that
 \begin{equation}\label{BasicSymbEstimates}
\| D_x^{\alpha}\partial_{\sigma}^{\beta}a_0(x,\sigma)\| \leq C_{\alpha,\beta}\langle x,\sigma \rangle_{\vec{\ell}}^{\mu-\ell_1\alpha-\ell_2\beta}
 \end{equation}
 for $\alpha,\beta \in \N_0$. Write $S^{\mu;\vec{\ell}}(\R^2)$ for these symbol spaces; the Hilbert spaces are understood from the context and not included in the notation in this section. Note that the usual rules of symbol calculus are valid for these symbol classes. These symbols are such that $C^{\infty}$-functions that are anisotropic homogeneous of degree $\mu$ in the large in the sense that
 $$
 a_0(\varrho^{\ell_1}x,\varrho^{\ell_2}\sigma) = \varrho^{\mu}a_0(x,\sigma)
 $$
for $\varrho \geq 1$ and large $|x,\sigma|$ belong to $S^{\mu;\vec{\ell}}(\R^2)$. In particular, for what follows it will be relevant that the function $x \in S^{\ell_1;\vec{\ell}}(\R^2)$. We also note that $\langle x,\sigma \rangle_{\vec{\ell}}^{\mu} \in S^{\mu;\vec{\ell}}(\R^2)$.

The estimates \eqref{BasicSymbEstimates} for $a_0(x,\sigma)$ imply that
\begin{equation}\label{BasicSymbEstimatesMellin}
\|(xD_x)^{\alpha}\partial_{\sigma}^{\beta}a_0(x,\sigma)\| \lesssim \langle x,\sigma \rangle_{\vec{\ell}}^{\mu-\ell_2\beta} \lesssim \langle x \rangle^{\frac{\mu_+}{\ell_1}}\langle \sigma \rangle^{\frac{\mu}{\ell_2}-\beta},
\end{equation}
where $\mu_+ = \max\{\mu,0\}$. Consequently, restricting to $x > 0$, $\langle x \rangle^{-\frac{\mu_+}{\ell_1}}a_0(x,\sigma)$ is a standard global Mellin symbol of class $S^{\frac{\mu}{\ell_2}}(\R_+\times\R)$.

We will need to work with symbols that depend holomorphically on $\sigma \in \C$. By $S^{\mu;\vec{\ell}}_O(\R\times\C)$ we denote the symbol class of functions $a(x,\sigma) \in C^{\infty}(\R\times\C)$ taking values in the bounded operators between Hilbert spaces such that $a(x,\sigma)$ is holomorphic with respect to $\sigma \in \C$, and such that for every $\gamma \in \R$ the function $\R^2 \ni (x,\sigma) \mapsto a(x,\sigma + i\gamma)$ belongs to $S^{\mu;\vec{\ell}}(\R^2)$ with symbol estimates that are locally uniform with respect to $\gamma$, i.e., for every $\alpha,\beta \in \N_0$ and $R > 0$ there are constants $C_{\alpha,\beta,R} \geq 0$ such that
$$
\|D_{x}^{\alpha}\partial_{\sigma}^{\beta}a(x,\sigma+i\gamma)\| \leq C_{\alpha,\beta,R}\langle x,\sigma \rangle_{\vec{\ell}}^{\mu-\ell_1\alpha-\ell_2\beta}
$$
for all $(x,\sigma) \in \R^2$, and all $\gamma \in \R$ with $|\gamma| \leq R$.

The real symbols obtained by restriction of $a(x,\sigma)$ to different lines $\Im(\sigma) = \gamma_1$ and $\Im(\sigma) = \gamma_2$ are related by asymptotic expansion
\begin{equation}\label{Taylorlines}
a(x,\sigma + i\gamma_2) \sim \sum \limits_{k=0}^{\infty}\frac{(i\gamma_2-i\gamma_1)^k}{k!}\bigl(\partial^k_{\sigma}a\bigr)(x,\sigma + i\gamma_1) \in S^{\mu;\vec{\ell}}(\R^2).
\end{equation}
This is both an exact pointwise representation via Taylor expansion in view of analyticity, and an asymptotic expansion with respect to the filtration by order in the symbol classes. Taylor's formula shows that the symbol estimates for the remainder terms of this asymptotic expansion are locally uniform with respect to $(\gamma_1,\gamma_2) \in \R^2$.

Every symbol $a_0 \in S^{\mu;\vec{\ell}}(\R^2)$ has a representative $a \in S^{\mu;\vec{\ell}}_O(\R\times\C)$ modulo $S^{-\infty}(\R^2)$. This follows by employing the \emph{kernel cut-off construction} (this construction is frequently used in the symbol calculus in Schulze's theory \cite{SchuNH}): Let $\phi \in C_c^{\infty}(\R)$ with $\phi \equiv 1$ near $t = 0$, and $\F$, $\F^{-1}$ be the Fourier transform and its inverse on tempered distributions, respectively. Then $a(x,\sigma) = \F_{t\to\sigma}\phi(t)\F^{-1}_{\sigma\to t}a_0$. The idea of this construction is that the Schwartz kernel of the Kohn-Nirenberg quantized pseudodifferential operator with symbol $a_0(x,\sigma)$ is localized near the diagonal using $\phi$ (which also explains why $a-a_0$ belongs to $S^{-\infty}$), and analyticity of $a(x,\sigma)$ in $\sigma$ then follows from the Paley-Wiener Theorem. To analyze the kernel cut-off operator in more detail it is convenient to rewrite it as an oscillatory integral in the form
\begin{equation}\label{kernelcutoff}
a(x,\sigma+i\gamma) = \frac{1}{2\pi} \iint e^{-it\tau}e^{t\gamma}\phi(t)a_0(x,\sigma-\tau)\,dt d\tau \\
\end{equation}
for real $x$, $\sigma$, and $\gamma$. The standard regularization procedure applied to this integral reveals that $a(x,\sigma+i\gamma)$ depends smoothly on $(x,\sigma + i\gamma) \in \R\times\C$, that the Cauchy-Riemann equations hold with respect to $\sigma+i\gamma \in \C$, and from the symbol estimates for $a_0$ we obtain the resulting estimates for $a$. While kernel cut-off as an operation on symbols only acts on the variable $\sigma$, the representation \eqref{kernelcutoff} shows that if additional parameters are present for a symbol $a_0$ that satisfy suitable joint estimates with $\sigma$, such as the variable $x$ in our case, then these joint estimates are typically preserved for the resulting analytic symbol $a$ along lines parallel to the real axis. Finally, from
\begin{align*}
a_0(x,\sigma) - a(x,\sigma) &= \frac{1}{2\pi} \iint e^{-it\tau}(1-\phi(t))a_0(x,\sigma-\tau)\,dt d\tau \\
&= \frac{i^k}{2\pi} \iint e^{-it\tau}\frac{1-\phi(t)}{t^k}[\partial^k_{\sigma}a_0](x,\sigma-\tau)\,dt d\tau, \; k \in \N_0,
\end{align*}
we obtain that $a - a_0 \in S^{-\infty}(\R^2)$. Detailed technical proofs of these statements, albeit for a different symbol class, can be found in \cite[Section 3]{KrainerVolterra}.

Kernel cut-off allows performing real variable symbolic manipulations along one line $\Im(\sigma) = \gamma$, such as manipulations that utilize excision functions, and then pass to a holomorphic representative of the resulting symbol modulo $S^{-\infty}(\R^2)$. An example for this is a proof that asymptotic expansions exist in the analytic category: Given $\mu_j \to -\infty$ and $a_j \in S^{\mu_j;\vec{\ell}}_O(\R\times\C)$, there exists $a(x,\sigma) \in S^{\mu;\vec{\ell}}_O(\R\times\C)$, $\mu = \max\mu_j$, such that $a \sim \sum\limits_{j=1}^{\infty}a_j$ in the sense that for every $R > 0$ there exists $k_0 \in \N_0$ such that $a - \sum\limits_{j=1}^ka_j \in S^{-R;\vec{\ell}}_O(\R\times\C)$ for $k \geq k_0$. To see this first restrict all $a_j(x,\sigma)$ to $\R^2$. By the standard Borel argument for real symbols there exists $a_0(x,\sigma) \in S^{\mu;\vec{\ell}}(\R^2)$ with $a_0(x,\sigma) \sim \sum\limits_{j=1}^{\infty}a_j(x,\sigma)$. Now use the kernel cut-off construction and define $a \in S^{\mu;\vec{\ell}}_O(\R\times\C)$ via \eqref{kernelcutoff}. Since $a-a_0 \in S^{-\infty}(\R^2)$ we still have the asymptotic expansion $a(x,\sigma) \sim \sum\limits_{j=1}^{\infty}a_j(x,\sigma)$ as real symbols. Consequently,
$$
a(x,\sigma) - \sum\limits_{j=1}^ka_j \in S^{-R;\vec{\ell}}(\R^2)\cap S^{\mu;\vec{\ell}}_O(\R\times\C)
$$
for $k \geq k_0$. But $S^{-R;\vec{\ell}}(\R^2)\cap S^{\mu;\vec{\ell}}_O(\R\times\C) = S^{-R;\vec{\ell}}_O(\R\times\C)$ by \eqref{Taylorlines}, thus establishing the desired result.

The following is another closely related application of the kernel cut-off construction.

\begin{lemma}\label{AnalyticSmoothingApprox}
Let $a \in S^{\mu;\vec{\ell}}_O(\R\times\C)$. Then there exists a sequence $a_j \in S^{-\infty}_O(\R\times\C)$ such that $a_j \to a$ in $S^{\mu';\vec{\ell}}_O(\R\times\C)$ as $j \to \infty$ for every $\mu' > \mu$.
\end{lemma}
\begin{proof}
Let $\chi \in C^{\infty}(\R^2)$ be a function with $\chi \equiv 0$ for $(|x|^{2\ell_2} + |\sigma|^{2\ell_1})^{\frac{1}{2\ell_1\ell_2}} \leq 1$ and $\chi \equiv 1$ for $(|x|^{2\ell_2} + |\sigma|^{2\ell_1})^{\frac{1}{2\ell_1\ell_2}} \geq 2$, and define
$$
b_j(x,\sigma) = \chi\bigl(\tfrac{x}{j^{\ell_1}},\tfrac{\sigma}{j^{\ell_2}}\bigr)a(x,\sigma) \in S^{\mu;\vec{\ell}}(\R^2), \; j \in \N.
$$
Then $b_j \to 0$ in $S^{\mu';\vec{\ell}}(\R^2)$ as $j \to \infty$ for every $\mu' > \mu$, and $a - b_j \in S^{-\infty}(\R^2)$. Now let $c_j \in S^{\mu;\vec{\ell}}_O(\R\times\C)$ be defined by applying the kernel cut-off operator \eqref{kernelcutoff} to $b_j$. Then $b_j - c_j \in S^{-\infty}(\R^2)$, and by the continuity of the kernel cut-off operator we have $c_j \to 0$ in $S^{\mu';\vec{\ell}}_O(\R\times\C)$ for $\mu' > \mu$. The lemma then follows with
$$
a_j = a - c_j \in S^{\mu;\vec{\ell}}_O(\R\times\C)\cap S^{-\infty}(\R^2) = S^{-\infty}_O(\R\times\C)
$$
and \eqref{Taylorlines}.
\end{proof}

Given $a \in S^{\mu;\vec{\ell}}_O(\R\times\C)$ we consider the Mellin pseudodifferential operator
$$
\opm(a) : \dot{C}^{\infty}(\R_+) \to \dot{C}^{\infty}(\R_+),
$$
where $\dot{C}^{\infty}$ is the space of (Hilbert space valued) smooth functions on $\R_+$ that vanish to infinite order at $x = 0$ and are Schwartz functions as $x \to \infty$. Analyticity of the symbol $a(x,\sigma)$ and the estimates \eqref{BasicSymbEstimatesMellin} along lines parallel to the real axis ensure that the action of $\opm(a)$ preserves this space. The map
$$
S^{\mu;\vec{\ell}}_O(\R\times\C) \times \dot{C}^{\infty}(\R_+) \ni (a,u) \mapsto \opm(a)u \in \dot{C}^{\infty}(\R_+)
$$
is continuous.

We assume in the sequel that the reader has some familiarity with standard (global) Mellin pseudodifferential calculus (see \cite{GilSchulzeSeiler,KrainerMellin,LewisParenti,SchuNH}). 

\begin{definition}
By $\Psi^{\mu;\vec{\ell}}_O$ we denote the operators of the form
$$
\opm(a) + G : \dot{C}^{\infty}(\R_+) \to \dot{C}^{\infty}(\R_+)
$$
with $a \in S^{\mu;\vec{\ell}}_O(\R\times\C)$, and the operator $G : \dot{C}^{\infty}(\R_+) \to \dot{C}^{\infty}(\R_+)$ and its formal adjoint $G^* : \dot{C}^{\infty}(\R_+) \to \dot{C}^{\infty}(\R_+)$ with respect to the $L^2_b$ and Hilbert space inner products in the target and range spaces extend to bounded operators
$$
G,\,G^* : x^{\alpha}H^s_b(\R_+) \to x^{\alpha'}H^{s'}_b(\R_+)
$$
for all $\alpha,s,s' \in \R$, and all $\alpha' \leq \alpha$.  Thus $G$ and $G^*$ are smoothing and produce Schwartz behavior as $x \to \infty$ while maintaining the same order of growth as $x \to 0$.
\end{definition}

\begin{lemma}\label{ImprovedMappingProps}
Let $a \in S^{\mu;\vec{\ell}}_O(\R\times\C)$, $\mu \leq 0$. For $j \in \N_0$ with $\mu+\ell_1j \leq 0$ we have
$$
\opm(a),\; \opm(a)^* : x^{\alpha}H^s_b(\R_+) \to x^{\alpha-j}H^{s-\frac{\mu+\ell_1 j}{\ell_2}}_b(\R_+)
$$
for all $s,\alpha \in \R$, where $\opm(a)^* : \dot{C}^{\infty}(\R_+) \to \dot{C}^{\infty}(\R_+)$ is the (formal) adjoint to $\opm(a) : \dot{C}^{\infty}(\R_+) \to \dot{C}^{\infty}(\R_+)$.
\end{lemma}
\begin{proof}
By assumption on $\mu$ and $j$, $x^ja(x,\sigma)$ is a standard global Mellin symbol of class $S^{\frac{\mu+\ell_1 j}{\ell_2}}(\R_+\times\R)$, which shows the asserted mapping property for $\opm(a)$. Now
$$
x^j[\opm(a)]^* = [\opm(a)x^j]^* = [x^j\opm(a(x,\sigma-ij))]^* : \dot{C}^{\infty}(\R_+) \to \dot{C}^{\infty}(\R_+),
$$
where $x^ja(x,\sigma-ij) \in S^{\frac{\mu+\ell_1 j}{\ell_2}}(\R_+\times\R)$. Consequently,
$$
x^j[\opm(a)]^* : x^{\alpha}H^s_b(\R_+) \to x^{\alpha}H^{s-\frac{\mu+\ell_1 j}{\ell_2}}_b(\R_+),
$$
and the lemma is proved.
\end{proof}

\begin{theorem}\label{AdjointsCalculus}
Let $\opm(a) + G \in \Psi^{\mu;\vec{\ell}}_O$. Then the formal adjoint of this operator with respect to the $L^2_b$ inner product belongs to the class $\Psi^{\mu;\vec{\ell}}_O$, where more precisely
$$
\bigl(\opm(a) + G\bigr)^* = \opm(a^{\sharp}) + G' : \dot{C}^{\infty}(\R_+) \to \dot{C}^{\infty}(\R_+),
$$
and $a^{\sharp} \in S^{\mu;\vec{\ell}}_O(\R\times\C)$ has the asymptotic expansion
\begin{equation}\label{AdjointAsymptotics}
a^{\sharp}(x,\sigma) \sim \sum\limits_{k=0}^{\infty}\frac{1}{k!}(xD_x)^k\partial_{\sigma}^k[a(x,\overline{\sigma})^*].
\end{equation}
\end{theorem}
\begin{proof}
Let first $a \in S^{-\infty}_O(\R\times\C)$. In particular, $a \in S^{-\infty}(\R_+\times\R)$ and is analytic, and from the standard global Mellin pseudodifferential calculus we have an exact representation of the formal adjoint $\opm(a)^* = \opm(b) : \dot{C}^{\infty}(\R_+) \to \dot{C}^{\infty}(\R_+)$, where for every $N \in \N_0$
\begin{align*}
b(x,\sigma) &= \frac{1}{2\pi}\iint y^{-i\eta}a(xy,\overline{\sigma}+\eta)^*\,\frac{dy}{y}d\eta \\
&= \sum\limits_{k=0}^{N-1}\frac{1}{k!}(xD_x)^k\partial_{\sigma}^k[a(x,\overline{\sigma})^*] + r_N(x,\sigma) \\
&= b_N(x,\sigma) + r_N(x,\sigma)
\end{align*}
with $$
r_N(x,\sigma) = \int_0^1\frac{(1-\theta)^{N-1}}{2\pi(N-1)!}\iint y^{-i\eta}[(-xD_x)^{N}\partial_{\sigma}^Na](xy,\overline{\sigma}+\theta\eta)^*\,\frac{dy}{y}d\eta d\theta.
$$
For $j \in \N_0$ we take advantage of the analyticity of the symbols to shift the integration into the complex $\eta$-plane and write
\begin{align*}
x^jr_N&(x,\sigma) = \int_0^1\frac{(1-\theta)^{N-1}}{2\pi(N-1)!}\iint y^{-i\eta}x^j[(-xD_x)^{N}\partial_{\sigma}^Na](xy,\overline{\sigma}+\theta\eta)^*\,\frac{dy}{y}d\eta d\theta \\
&= \int_0^1\frac{(1-\theta)^{N-1}}{2\pi(N-1)!}\iint y^{-i\eta}(xy)^j[(-xD_x)^{N}\partial_{\sigma}^Na](xy,\overline{\sigma}+\theta(\eta+ij))^*\,\frac{dy}{y}d\eta d\theta \\
&= \int_0^1\frac{(1-\theta)^{N-1}}{2\pi(N-1)!}\iint y^{-i\eta}[x^j(-xD_x)^{N}\partial_{\sigma}^Na](xy,\overline{\sigma}+\theta(\eta+ij))^*\,\frac{dy}{y}d\eta d\theta.
\end{align*}
From these formulas we see that for any given $\mu' \in \R$, $j_0 \in \N_0$, and $\mu_0 < 0$ there exists $N_0 \in \N$ such that for $N \geq N_0$ and $0 \leq j \leq j_0$ the map $S^{-\infty}_O(\R\times\C) \ni a(x,\sigma) \mapsto x^jr_N(x,\sigma)$ extends to a continuous map $S^{\mu';\vec{\ell}}_O(\R\times\C) \to S^{\mu_0}(\R_+\times\R)$, and the same is true for the map $a \mapsto x^jr_N(x,\sigma-ij)$. For these values of the parameters we then have
\begin{gather*}
\opm(r_N) : x^{\alpha}H^s_b(\R_+) \to x^{\alpha-j}H^{s-\mu_0}_b(\R_+), \\
\opm(r_N)^* = x^{-j}\opm(x^jr_N(x,\sigma-ij))^* : x^{\alpha}H^s_b(\R_+) \to x^{\alpha-j}H^{s-\mu_0}_b(\R_+)
\end{gather*}
for all $\alpha,s \in \R$.

Now let $a \in S^{\mu;\vec{\ell}}_O(\R\times\C)$. By Lemma~\ref{AnalyticSmoothingApprox} there exists a sequence $a_{\nu} \in S^{-\infty}_O(\R\times\C)$ such that $a_{\nu} \to a$ in $S^{\mu';\vec{\ell}}_O(\R\times\C)$ as $\nu \to \infty$ for $\mu' > \mu$. Then
\begin{gather*}
\langle \opm(a)u,v \rangle_{L^2_b} \longleftarrow \langle \opm(a_{\nu})u,v \rangle_{L^2_b} = \langle u,\opm(a_{\nu})^*v \rangle_{L^2_b} \\
= \langle u, (\opm(b_{N,\nu}) + \opm(r_{N,\nu}))v \rangle_{L^2_b} \longrightarrow \langle u, (\opm(b_{N}) + \opm(r_{N}))v \rangle_{L^2_b}
\end{gather*}
for $u,v \in C_c^{\infty}(\R_+)$ and $N \in \N$ large enough, where the extra $\nu$-parameter indicates that we use the previous formulas for $a_{\nu}$, while its absence means that the formulas are applied to the symbol $a(x,\sigma)$. In particular, $\opm(a)^* = \opm(b_N) + \opm(r_N)$ for $N$ large enough with the above mapping properties for $\opm(r_N)$. Consequently, if $a^{\sharp} \in S^{\mu;\vec{\ell}}_O(\R\times\C)$ has the asymptotic expansion \eqref{AdjointAsymptotics}, then $G_0 = \opm(a)^* - \opm(a^{\sharp})$ satisfies
$$
G_0,\,G_0^* : x^{\alpha}H^s_b(\R_+) \to x^{\alpha'}H^{s'}_b(\R_+)
$$
for all $\alpha,s,s' \in \R$, and all $\alpha' \leq \alpha$ by Lemma~\ref{ImprovedMappingProps} and the mapping properties of $\opm(r_N)$ (as $N \to \infty$). The theorem is proved.
\end{proof}

\begin{theorem}\label{CompositionCalculus}
Let $\opm(a_j) + G_j \in \Psi^{\mu_j;\vec{\ell}}_O$, $j = 1,2$. Then
$$
(\opm(a_1)+G_1)\circ (\opm(a_2) + G_2) = \opm(a_1{\#}a_2) + G : \dot{C}^{\infty}(\R_+) \to \dot{C}^{\infty}(\R_+)
$$
with $\opm(a_1{\#}a_2) + G \in \Psi^{\mu_1+\mu_2;\vec{\ell}}_O$, and $a_1{\#}a_2 \in S^{\mu_1+\mu_2;\vec{\ell}}_O(\R\times\C)$ has the asymptotic expansion
\begin{equation}\label{LeibnizAsymptotics}
(a_1{\#}a_2)(x,\sigma) \sim \sum\limits_{k=0}^{\infty}\frac{1}{k!}[\partial_{\sigma}^ka_1](x,\sigma)[(xD_x)^ka_2](x,\sigma).
\end{equation}
\end{theorem}
\begin{proof}
The composition $G_1G_2 \in \Psi^{-\infty}_O$ in view of the defining mapping properties of the $G_j$. We next prove that the compositions $\opm(a_1)G_2$ and $G_1\opm(a_2)$ belong to $\Psi^{-\infty}_O$.  Let $\omega \in C_c^{\infty}(\R)$ with $\omega \equiv 1$ near $x=0$, and decompose
$$
a_j(x,\sigma) = \omega a_j(x,\sigma) + (1-\omega) a_j(x,\sigma) = a_{j,0}(x,\sigma) + a_{j,\infty}(x,\sigma).
$$
The symbols $a_{j,0}(x,\sigma)$ and $a_{j,\infty}(x,\sigma)$ no longer satisfy the joint symbol estimates \eqref{BasicSymbEstimates} in $(x,\sigma)$, but we have
$$
a_{j,0}(x,\sigma),\, x^{-K}a_{j,\infty}(x,\sigma) \in S^{\frac{\mu_j}{\ell_2}}(\R_+\times\R)
$$
for $K$ large enough by \eqref{BasicSymbEstimatesMellin}. Combined with the analyticity in $\sigma$ this shows that
\begin{gather*}
\opm(a_{j,0}) : x^{\alpha}H^s_b(\R_+) \to x^{\alpha}H^{s-\frac{\mu_j}{\ell_2}}_b(\R_+), \\
\opm(a_{j,\infty}) : x^{\alpha}H^s_b(\R_+) \to x^{\alpha+K}H^{s-\frac{\mu_j}{\ell_2}}_b(\R_+).
\end{gather*}
Now $G_1\opm(a_{2,0}) : x^{\alpha}H^s_b(\R_+) \to x^{\alpha'}H^{s'}_b(\R_+)$ for all $s,s',\alpha \in \R$ and $\alpha' \leq \alpha$ by the mapping properties of $G_1$. Likewise, $G_1\opm(a_{2,\infty}) : x^{\alpha}H^s_b(\R_+) \to x^{\alpha'}H^{s'}_b(\R_+)$ for all $s,s',\alpha \in \R$ and $\alpha' \leq \alpha+K$, and so $G_1\opm(a_2) : x^{\alpha}H^s_b(\R_+) \to x^{\alpha'}H^{s'}_b(\R_+)$ for all $s,s',\alpha \in \R$ and $\alpha' \leq \alpha$.

We have $G_2 : x^{\alpha}H^s_b(\R_+) \to x^{\alpha'}H^{s'+\frac{\mu_1}{\ell_2}}_b(\R_+)$, and so $\opm(a_{1,0})G_2 : x^{\alpha}H^s_b(\R_+) \to x^{\alpha'}H^{s'}_b(\R_+)$ for all $s,s',\alpha \in \R$ and $\alpha' \leq \alpha$. We likewise have $G_2 : x^{\alpha}H^s_b(\R_+) \to x^{\alpha'-K}H^{s'+\frac{\mu_1}{\ell_2}}_b(\R_+)$, and thus $\opm(a_{1,\infty})G_2 : x^{\alpha}H^s_b(\R_+) \to x^{\alpha'}H^{s'}_b(\R_+)$, which shows that $\opm(a_1)G_2 : x^{\alpha}H^s_b(\R_+) \to x^{\alpha'}H^{s'}_b(\R_+)$ for all $s,s',\alpha \in \R$ and $\alpha' \leq \alpha$. For the formal adjoints we have
$$
[G_1\opm(a_2)]^* = \opm(a_2)^*G_1^* \textup{ and } [\opm(a_1)G_2]^* = G_2^*\opm(a_1)^*,
$$
and because $\opm(a_j)^* \in \Psi^{\mu_j;\vec{\ell}}_O$ by Theorem~\ref{AdjointsCalculus} we finally obtain with the above that $\opm(a_1)G_2,\;G_1\opm(a_2) \in \Psi^{-\infty}_O$.

It remains to consider the composition $\opm(a_1)\opm(a_2)$. Let first $a_j \in S^{-\infty}_O(\R\times\C)$. Then $\opm(a_1)\opm(a_2) = \opm(c)$ by the standard Mellin pseudodifferential calculus, where for every $N \in \N_0$
\begin{align*}
c(x,\sigma) &= \frac{1}{2\pi}\iint y^{-i\eta}a_1(x,\sigma+\eta)a_2(xy,\sigma)\,\frac{dy}{y}d\eta \\
&= \sum\limits_{k=0}^{N-1}\frac{1}{k!}[\partial_{\sigma}^ka_1](x,\sigma)[(xD_x)^ka_2](x,\sigma) + r_N(x,\sigma) \\
&= c_N(x,\sigma) + r_N(x,\sigma)
\end{align*}
with
$$
r_N(x,\sigma) = \int_0^1\!\frac{(1-\theta)^{N-1}}{2\pi(N-1)!}\iint y^{-i\eta}[\partial_{\sigma}^Na_1](x,\sigma+\theta\eta)[(xD_{x})^Na_2](xy,\sigma)\,\frac{dy}{y}d\eta d\theta.
$$
Into this formula for $r_N(x,\sigma)$ we substitute the splitting $a_2 = a_{2,0} + a_{2,\infty}$ from before and obtain $r_N(x,\sigma) = r_{N,0}(x,\sigma) + r_{N,\infty}(x,\sigma)$ which we process separately. Let $j \in \N_0$ be arbitrary. Then $x^jr_{N,0}(x,\sigma)$ is given by
$$
\int_0^1\!\frac{(1-\theta)^{N-1}}{2\pi(N-1)!}\iint y^{-i\eta}[x^j\partial_{\sigma}^Na_1](x,\sigma+\theta\eta)[(xD_{x})^Na_{2,0}](xy,\sigma)\,\frac{dy}{y}d\eta d\theta,
$$
and from this formula we see that for any given $\mu_1',\mu_2' \in \R$, $j_0 \in \N_0$, and $\mu_0 < 0$ there exists $N_0 \in \N$ such that for $N \geq N_0$ and $0 \leq j \leq j_0$ the map
$$
S^{-\infty}_O(\R\times\C) \times S^{-\infty}_O(\R\times\C) \ni (a_1,a_2) \mapsto x^jr_{N,0}(x,\sigma)
$$
extends to a continuous map
$$
S^{\mu_1';\vec{\ell}}_O(\R\times\C) \times S^{\mu_2';\vec{\ell}}_O(\R\times\C) \to S^{\mu_0}(\R_+\times\R),
$$
and the same is true for the map $(a_1,a_2) \mapsto x^jr_{N,0}(x,\sigma-ij)$. For these values of the parameters we then have
\begin{gather*}
\opm(r_{N,0}) : x^{\alpha}H^s_b(\R_+) \to x^{\alpha-j}H^{s-\mu_0}_b(\R_+), \\
\opm(r_{N,0})^* = x^{-j}\opm(x^jr_{N,0}(x,\sigma-ij))^* : x^{\alpha}H^s_b(\R_+) \to x^{\alpha-j}H^{s-\mu_0}_b(\R_+)
\end{gather*}
for all $\alpha,s \in \R$ just as in the proof of Theorem~\ref{AdjointsCalculus}.

To analyze $x^jr_{N,\infty}(x,\sigma)$ we take advantage of analyticity and shift the integration in $\eta$ into the complex plane and obtain
\begin{align*}
x^j&r_{N,\infty}(x,\sigma) = \int_0^1\!\frac{(1-\theta)^{N-1}}{2\pi(N-1)!} \cdot \\
&\iint y^{-i\eta}[x^{j+K}\partial_{\sigma}^Na_1](x,\sigma+\theta(\eta-iK))[x^{-K}(xD_{x})^Na_{2,\infty}](xy,\sigma)\,\frac{dy}{y}d\eta d\theta.
\end{align*}
Consequently, given $\mu_1',\mu_2' \in \R$, $j_0 \in \N_0$, and $\mu_0 < 0$, we first choose $K \in \N$ large enough such that $x^{-K}(1-\omega)S^{\mu_2';\vec{\ell}}_O(\R\times\C) \hookrightarrow S^{\frac{\mu_2'}{\ell_2}}(\R_+\times\R)$, and can then find $N_0 \in \N$ such that for $N \geq N_0$ and $0 \leq j \leq j_0$ the map $(a_1,a_2) \mapsto x^jr_{N,\infty}(x,\sigma)$ is continuous in
$$
S^{\mu_1';\vec{\ell}}_O(\R\times\C) \times S^{\mu_2';\vec{\ell}}_O(\R\times\C) \to S^{\mu_0}(\R_+\times\R),
$$
and the same property holds for the map $(a_1,a_2) \mapsto x^jr_{N,\infty}(x,\sigma-ij)$. Thus $\opm(r_{N,\infty})$ and $\opm(r_{N,\infty})^*$ have the same mapping properties as previously stated for $\opm(r_{N,0})$ and $\opm(r_{N,0})^*$ for $N$ large enough, and therefore also $\opm(r_N)$ and $\opm(r_N)^*$ have these properties.

Now let $a_j \in S^{\mu_j;\vec{\ell}}_O(\R\times\C)$. By Lemma~\ref{AnalyticSmoothingApprox} there exist sequences $a_{j,\nu} \in S^{-\infty}_O(\R\times\C)$ such that $a_{j,\nu} \to a_j$ in $S^{\mu_j';\vec{\ell}}_O(\R\times\C)$ as $\nu \to \infty$ for $\mu_j' > \mu_j$. For every $u \in \dot{C}^{\infty}(\R_+)$ we then have for $N \in \N$ large enough
\begin{gather*}
\opm(a_1)\opm(a_2)u \longleftarrow \opm(a_{1,\nu})\opm(a_{2,\nu})u \\
= \opm(c_{N,\nu})u + \opm(r_{N,\nu})u \longrightarrow \opm(c_N)u + \opm(r_N)u,
\end{gather*}
where the extra $\nu$-parameter indicates that we use the previous formulas for $a_{j,\nu}$, while its absence means that the formulas are applied to the symbols $a_j(x,\sigma)$. In particular, $\opm(a_1)\opm(a_2) = \opm(c_N) + \opm(r_N)$ for $N$ large enough, and $\opm(r_N)$ has the mapping properties previously shown. Consequently, if $a_1{\#}a_2 \in S^{\mu_1+\mu_2;\vec{\ell}}_O(\R\times\C)$ has the asymptotic expansion \eqref{LeibnizAsymptotics}, then $G_0 = \opm(a_1)\opm(a_2) - \opm(a_1{\#}a_2)$ satisfies
$$
G_0,\,G_0^* : x^{\alpha}H^s_b(\R_+) \to x^{\alpha'}H^{s'}_b(\R_+)
$$
for all $\alpha,s,s' \in \R$, and all $\alpha' \leq \alpha$ by Lemma~\ref{ImprovedMappingProps} and the mapping properties of $\opm(r_N)$ as $N \to \infty$. The theorem is proved.
\end{proof}

\begin{definition}\label{hypoellipticity}
A symbol $a(x,\sigma) \in S^{\mu;\vec{\ell}}(\R^2)$ is \emph{right-hypoelliptic} of order $(\mu,\mu')$ if $a(x,\sigma)$ is invertible for sufficiently large $|x,\sigma| \gg 0$, and the inverse satisfies
\begin{align}
\|a^{-1}(x,\sigma)\| &\lesssim \langle x,\sigma \rangle_{\vec{\ell}}^{-\mu'}, \label{EstimateInverse} \\
\intertext{and for every $\alpha,\beta \in \N_0$ we have}
\|[D_x^{\alpha}\partial_{\sigma}^{\beta}a](x,\sigma)[a(x,\sigma)]^{-1}\| &\lesssim \langle x,\sigma \rangle_{\vec{\ell}}^{-\ell_1\alpha - \ell_2\beta}. \label{DerivativesInverse}
\end{align}
We call $a(x,\sigma) \in S^{\mu;\vec{\ell}}_O(\R\times\C)$ right-hypoelliptic of order $(\mu,\mu')$ if its restriction to $\R^2$ is right-hypoelliptic of order $(\mu,\mu')$.
\end{definition}

Suppose $a(x,\sigma) \in S^{\mu;\vec{\ell}}_O(\R\times\C)$ is right-hypoelliptic of order $(\mu,\mu')$. Let $\chi \in C^{\infty}(\R^2)$ be an excision function such that $\chi \equiv 0$ near $(0,0)$ and $\chi \equiv 1$ for large $|x,\sigma|$ so that $\chi(x,\sigma)a(x,\sigma)^{-1}$ is defined on $\R^2$. The estimates \eqref{EstimateInverse} and \eqref{DerivativesInverse} show that $\chi(x,\sigma)a(x,\sigma)^{-1} \in S^{-\mu';\vec{\ell}}(\R^2)$. We then apply the kernel cut-off operator \eqref{kernelcutoff} to this symbol to obtain a holomorphic symbol $q(x,\sigma) \in S^{-\mu';\vec{\ell}}_O(\R\times\C)$ which has the following properties:

\begin{lemma}\label{SymbolInverseProps}
\begin{enumerate}[(a)]
\item We have $a(x,\sigma)q(x,\sigma) - 1,\, q(x,\sigma)a(x,\sigma) - 1 \in S^{-\infty}_O(\R\times\C)$.
\item For $\gamma_j \in \R$, $j=1,2$, we have
$$
[D_x^{\alpha}\partial_{\sigma}^\beta a](x,\sigma+i\gamma_2)[D_x^{\alpha'}\partial_{\sigma}^{\beta'}q](x,\sigma+i\gamma_1) \in S^{-\ell_1(\alpha+\alpha')-\ell_2(\beta+\beta');\vec{\ell}}_O(\R\times\C)
$$
for all $\alpha,\alpha',\beta,\beta' \in \N_0$.
\item $a(x,\sigma + i\gamma) \in S^{\mu;\vec{\ell}}(\R^2)$ is right-hypoelliptic of order $(\mu,\mu')$ for every $\gamma \in \R$.
\end{enumerate}
\end{lemma}
\begin{proof}
We have $q(x,\sigma) - \chi(x,\sigma)a(x,\sigma)^{-1} \in S^{-\infty}(\R^2)$, and so
$$
a(x,\sigma)q(x,\sigma) - 1 \in S^{(\mu-\mu')_+;\vec{\ell}}_O(\R\times\C)\cap S^{-\infty}(\R^2) = S^{-\infty}_O(\R\times\C)
$$
by \eqref{Taylorlines}; here $(\mu-\mu')_+ = \max\{\mu-\mu',0\}$. For the same reason we also have $q(x,\sigma)a(x,\sigma) - 1 \in S^{-\infty}_O(\R\times\C)$, proving (a).

To prove (b) we first consider $\gamma_1=\gamma_2=0$. By \eqref{DerivativesInverse} we have
$$
[D_x^{\alpha}\partial_{\sigma}^\beta a](x,\sigma)[D_x^{\alpha'}\partial_{\sigma}^{\beta'}(\chi a^{-1})](x,\sigma) \in S^{-\ell_1(\alpha+\alpha')-\ell_2(\beta+\beta');\vec{\ell}}(\R^2),
$$
and because $q - \chi a^{-1} \in S^{-\infty}(\R^2)$ we can replace $\chi a^{-1}$ by $q$. Taylor expansion \eqref{Taylorlines} then shows that
$$
[D_x^{\alpha}\partial_{\sigma}^\beta a](x,\sigma)[D_x^{\alpha'}\partial_{\sigma}^{\beta'}q](x,\sigma) \in S^{-\ell_1(\alpha+\alpha')-\ell_2(\beta+\beta');\vec{\ell}}_O(\R\times\C),
$$
and translation proves (b) for general $\gamma_1=\gamma_2$. For $\gamma_1 \neq \gamma_2$ we first use \eqref{Taylorlines} and get
$$
[D_x^{\alpha}\partial_{\sigma}^\beta a](x,\sigma+i\gamma_2) \sim \sum \limits_{k=0}^{\infty}\frac{(i\gamma_2-i\gamma_1)^k}{k!}\bigl(D_x^{\alpha}\partial_{\sigma}^{\beta+k}a\bigr)(x,\sigma + i\gamma_1).
$$
Now multiply from the right by $[D_x^{\alpha'}\partial_{\sigma}^{\beta'}q](x,\sigma+i\gamma_1)$. The resulting asymptotic expansion then shows that (b) also holds in the case $\gamma_1 \neq \gamma_2$.

Finally, (c) follows from (a) and (b). By (a), $a(x,\sigma+i\gamma)$ is invertible for sufficiently large $|x,\sigma| \gg 0$, $(x,\sigma) \in \R^2$, and the inverse differs from $q(x,\sigma+i\gamma)$ by a rapidly decreasing function in $(x,\sigma)$. Both estimates \eqref{EstimateInverse} and \eqref{DerivativesInverse} hold for $q(x,\sigma+i\gamma)$ instead of $a(x,\sigma+i\gamma)^{-1}$ and are stable with respect to rapidly decreasing perturbations. The lemma is proved.
\end{proof}

\begin{theorem}\label{ParametrixHypoelliptic}
Let $\opm(a) + G \in \Psi^{\mu;\vec{\ell}}_O$, and suppose that $a \in S^{\mu;\vec{\ell}}_O(\R\times\C)$ is right-hypoelliptic of order $(\mu,\mu')$. Then there exists $\opm(b) \in \Psi^{-\mu';\vec{\ell}}_O$ such that
$$
\bigl(\opm(a) + G\bigr)\circ \opm(b) = 1 + \tilde{G} : \dot{C}^{\infty}(\R_+) \to \dot{C}^{\infty}(\R_+)
$$
with $\tilde{G} \in \Psi^{-\infty}_O$.
\end{theorem}
\begin{proof}
With the symbol $q(x,\sigma) \in S^{-\mu';\vec{\ell}}_O(\R\times\C)$ we have
$$
(a{\#}q)(x,\sigma) \sim \sum\limits_{k=0}^{\infty}\frac{1}{k!}[\partial_{\sigma}^ka](x,\sigma)[(xD_x)^kq](x,\sigma) = 1 + r(x,\sigma)
$$
with $r(x,\sigma) \in S^{-\ell_2;\vec{\ell}}_O(\R\times\C)$ by Lemma~\ref{SymbolInverseProps}. Now apply the formal Neumann series argument and get $r'(x,\sigma) \in S^{-\ell_2;\vec{\ell}}_O(\R\times\C)$ such that $(1+r){\#}(1+r') \sim 1$. Let $b(x,\sigma) \in S^{-\mu';\vec{\ell}}_O(\R\times\C)$ with $b(x,\sigma) \sim q{\#}(1+r')$. Then $a{\#}b \sim 1$, and Theorem~\ref{CompositionCalculus} implies the assertion.
\end{proof}

Analogous considerations apply to left-hypoellipticity, but left-hypoellipticity is not used in this work.


\section{An indefinite space arising in analytic Fredholm theory}\label{ModelPairing}

\begin{definition}[\protect{\cite[Section~4.5]{GohbergLancasterRodman}}]
Let $\bigl(V,[\cdot,\cdot]_V\bigr)$ and $\bigr(W,[\cdot,\cdot]_W\bigr)$ be finite-di\-men\-sio\-nal $\C$-vector spaces equipped with Hermitian sesquilinear forms $[\cdot,\cdot]_V$ and $[\cdot,\cdot]_W$, respectively. Let further $g \in \End(V)$ and $h \in \End(W)$. Then the triples $\bigl(V,[\cdot,\cdot]_V,g\bigr)$ and $\bigl(W,[\cdot,\cdot]_W,h\bigr)$ are unitarily equivalent if there exists an isomorphism $T : V \to W$ such that both
$$
[Tu,Tv]_W = [u,v]_V \; \forall u,v \in V \textup{ and } h = TgT^{-1}
$$
hold.
\end{definition}

The Canonical Form Theorem \cite[Theorem~5.1.1]{GohbergLancasterRodman} classifies triples $\bigl(V,[\cdot,\cdot]_V,g\bigr)$ for $g \in \End(V)$ that are selfadjoint with respect to nondegenerate $[\cdot,\cdot]_V$ up to unitary equivalence. The generalized eigenspaces $\E_{\sigma_0} \subset V$ of $g$ associated with real eigenvalues $\sigma_0 \in \R$ are particularly interesting. By the Canonical Form Theorem we can localize to $\bigl(\E_{\sigma_0},[\cdot,\cdot]_V,g-\sigma_0\bigr)$ to further study these spaces. In this appendix and Appendix~\ref{AnalyticCrossingsFiniteDimensional} we discuss an indefinite space from analytic Fredholm theory that arises as unitarily equivalent to $\bigl(\E_{\sigma_0},[\cdot,\cdot]_V,g-\sigma_0\bigr)$ for indicial operators. We refer to \cite{EidamPiccione,FarberLevine,GiamboPiccionePortaluri,GilMendoza} for related investigations and results. Information about the general theory of Fredholm operator pencils and applications to differential equations can be found in \cite{KozlovMazya}.

\bigskip

Suppose $E_0$ and $E_1$ are separable complex Hilbert spaces such that $E_1 \hookrightarrow E_0$ is densely and continuously embedded, and let
$$
\P : B_{\varepsilon}(0) \to \L(E_1,E_0)
$$
be a holomorphic operator function defined on the open disk $B_{\varepsilon}(0) \subset \C$ for some $\varepsilon > 0$. We assume that $\P(\sigma)$ is Fredholm for all $\sigma \in B_{\varepsilon}(0)$ and invertible for all $\sigma \neq 0$. We consider each operator $\P(\sigma) : E_1 \subset E_0 \to E_0$ an unbounded operator acting in $E_0$ with domain $E_1$. Under the stated assumptions, $\P(\sigma)$ is closed and densely defined. We assume that
$$
\P(\overline{\sigma})^* = \P(\sigma) : E_1 \subset E_0 \to E_0
$$
holds. In particular, $\P(\sigma)$ is selfadjoint for real $\sigma$. We denote the set of germs of such operator functions $\P(\sigma)$ by $\Germs(E_1,E_0)$.

With every $\P(\sigma) \in \Germs(E_1,E_0)$ we associate its spectral flow across $\sigma = 0$ as follows: Pick $\varepsilon_0, \delta_0 > 0$ small enough such that $\P(\sigma) - \lambda : E_1 \to E_0$ is Fredholm for $-\delta_0 \leq \sigma \leq \delta_0$ and $-\varepsilon_0 \leq \lambda \leq \varepsilon_0$, and such that $\P(\sigma) \pm \varepsilon_0 : E_1 \to E_0$ is invertible for $-\delta_0 \leq \sigma \leq \delta_0$. Then
$$
\SF_{\sigma=0}[\P(\sigma) : E_1 \subset E_0 \to E_0] = \lim\limits_{\sigma \to 0^+}\Bigl(\tr\bigl[\Pi_{[0,\varepsilon_0)}(\P(\sigma))\bigr] - \tr\bigl[\Pi_{[0,\varepsilon_0)}(\P(-\sigma))\bigr]\Bigr),
$$
where $\Pi_{[0,\varepsilon_0)}$ in each case is the spectral projection onto the spectral subspace of the operator pertaining to the part of the spectrum contained in $[0,\varepsilon_0)$. We refer to the survey by Lesch \cite{LeschSF} for details on the spectral flow.

For any Fr{\'e}chet space $F$ let $\Mero_0(F)$ denote the space of meromorphic germs of $F$-valued functions at $0$, and let $\Hol_0(F)$ denote the holomorphic germs. Then $\bigl(\Mero/\Hol\bigr)_0(F):= \Mero_0(F)/\Hol_0(F)$ can be identified with the space of principal parts of Laurent expansions at $0$ of $F$-valued meromorphic functions. The operator function $\P(\sigma) \in \Germs(E_1,E_0)$ induces a map
\begin{gather*}
\P : \bigl(\Mero/\Hol\bigr)_0(E_1) \to \bigl(\Mero/\Hol\bigr)_0(E_0), \\
\P\bigl[\hat{u}(\sigma) + \Hol_0(E_1)\bigr] = \P(\sigma)\hat{u}(\sigma) + \Hol_0(E_0)
\end{gather*}
for $\hat{u}(\sigma) + \Hol_0(E_1) \in \bigl(\Mero/\Hol\bigr)_0(E_1)$. Define
$$
\K(\P) = \{\hat{u}(\sigma) + \Hol_0(E_1) \in \bigl(\Mero/\Hol\bigr)_0(E_1);\; \P(\sigma)\hat{u}(\sigma) \in \Hol_0(E_0)\}.
$$
Then $\K(\P) = \ker\bigl[\P : \bigl(\Mero/\Hol\bigr)_0(E_1) \to \bigl(\Mero/\Hol\bigr)_0(E_0)\bigr]$, and by analytic Fredholm theory we have $\dim \K(\P) < \infty$. We define a pairing
\begin{gather*}
[\cdot,\cdot]_{\K(\P)} : \K(\P) \times \K(\P) \to \C \\
\intertext{via}
[\hat{u},\hat{v}]_{\K(\P)} = \frac{1}{2\pi i}\oint_C \langle \P(\sigma)\hat{u}(\sigma),\hat{v}(\overline{\sigma})\rangle_{E_0}\,d\sigma = \res\limits_{\sigma=0}\langle \P(\sigma)\hat{u}(\sigma),\hat{v}(\overline{\sigma})\rangle_{E_0},
\end{gather*}
where $C$ is a positively oriented circle of sufficiently small radius centered at the origin. It is easy to see that $[\cdot,\cdot]_{\K(\P)}$ is well-defined (i.e. independent of representatives $\hat{u}$ and $\hat{v}$ modulo holomorphic germs), and that it furnishes a Hermitian sesquilinear form on $\K(\P)$. We then consider the triple $\bigl(\K(\P),[\cdot,\cdot]_{\K(\P)},M_{\sigma}\bigr)$, where
$$
M_{\sigma} : \K(\P) \to \K(\P), \quad \hat{u}(\sigma) + \Hol_0(E_1) \mapsto \sigma\hat{u}(\sigma) + \Hol_0(E_1).
$$
Note that $M_{\sigma}$ is selfadjoint with respect to $[\cdot,\cdot]_{\K(\P)}$.

For $\ell \in \N$ we also consider the Hermitian sesquilinear forms
\begin{equation*}
\begin{gathered}{}
[\cdot,\cdot]_{\K(\P),\ell} : \ker M_{\sigma}^{\ell}\times\ker M_{\sigma}^{\ell} \to \C, \\
[\hat{u},\hat{v}]_{\K(\P),\ell} = [M_{\sigma}^{\ell-1}\hat{u},\hat{v}]_{\K(\P)}.
\end{gathered}
\end{equation*}
Let $(m_0(\ell),m_+(\ell),m_-(\ell))$ be the invariants of $[\cdot,\cdot]_{\K(\P),\ell}$. The numbers $m_+(\ell)$ and $m_-(\ell)$ yield the sign characteristic for $M_{\sigma}$ associated to Jordan blocks of size $\ell\times\ell$ in the Canonical Form Theorem \cite[Theorem~5.1.1]{GohbergLancasterRodman} for $\bigl(\K(\P),[\cdot,\cdot]_{\K(\P)},M_{\sigma}\bigr)$, see \cite[Theorem~5.8.1]{GohbergLancasterRodman} and Proposition~\ref{SignCharacteristic}. The main theorem regarding this triple is the following.

\begin{theorem}\label{Nondeg}
$\bigl(\K(\P),[\cdot,\cdot]_{\K(\P)}\bigr)$ is nondegenerate, and
$$
\sig\bigl(\K(\P),[\cdot,\cdot]_{\K(\P)}\bigr) = \sum\limits_{\ell \; \textup{odd}}\bigl(m_+(\ell)-m_-(\ell)\bigr) = \SF_{\sigma=0}\bigl[\P(\sigma) : E_1 \subset E_0 \to E_0\bigr],
$$
where $m_{\pm}(\ell)$ are the invariants associated with the triple $\bigl(\K(\P),[\cdot,\cdot]_{\K(\P)},M_{\sigma}\bigr)$.
\end{theorem}
\begin{proof}
We will prove the theorem by reducing it to the finite-dimensional case, considered separately in Appendix~\ref{AnalyticCrossingsFiniteDimensional}. To this end, let $N = \ker(\P(0))$ and $R = \ran(\P(0))$. We have the orthogonal decomposition $E_0 = N \oplus R$. Let $\pi_N=\pi_N^* = \pi_N^2 \in \L(E_0)$ be the orthogonal projection onto $N$, and  let $\pi_R = 1-\pi_N$. Because $\dim N < \infty$ both $\pi_N,\pi_R \in \L(E_1)$ by the Closed Graph Theorem, and we get $E_1 = N \oplus [R\cap E_1]$. We decompose
$$
\P(\sigma) = \begin{bmatrix} \P_{11}(\sigma) & \P_{12}(\sigma) \\ \P_{21}(\sigma) & \P_{22}(\sigma) \end{bmatrix} : \begin{array}{c} N \\ \oplus \\ R\cap E_1 \end{array} \to \begin{array}{c} N \\ \oplus \\ R \end{array},
$$
and note that $\P_{22}(\sigma)$ is invertible for $|\sigma| < \varepsilon$ for $\varepsilon > 0$ small enough.

Now, for $0 \leq t \leq 1$, define $\U_t(\sigma) : E_j \to E_j$ via
$$
\U_t(\sigma) = \begin{bmatrix} 1 & 0 \\ -t\P_{22}(\sigma)^{-1}\P_{21}(\sigma) & 1 \end{bmatrix} :
\begin{array}{c} N \\ \oplus \\ R\cap E_j\end{array} \to \begin{array}{c} N \\ \oplus \\ R\cap E_j \end{array}, \quad j=1,2.
$$
Then $\U_t(\sigma) : B_{\varepsilon}(0) \to \L(E_j)$ is holomorphic and invertible, and the same holds for
$$
\U_t(\overline{\sigma})^* = \begin{bmatrix} 1 & -t\P_{12}(\sigma)\P_{22}(\sigma)^{-1} \\ 0 & 1 \end{bmatrix} :
\begin{array}{c} N \\ \oplus \\ R\cap E_j \end{array} \to \begin{array}{c} N \\ \oplus \\ R\cap E_j \end{array},
$$
where the adjoint refers to the base Hilbert space $E_0$. Then
$$
\P_t(\sigma) = \U_t(\overline{\sigma})^*\P(\sigma)\U_t(\sigma) \in \Germs(E_1,E_0)
$$
is a homotopy within $\Germs(E_1,E_0)$ between $\P_0(\sigma) = \P(\sigma)$ and
$$
\P_1(\sigma) = \begin{bmatrix} \P_{11}(\sigma) - \P_{12}(\sigma)\P_{22}(\sigma)^{-1}\P_{21}(\sigma) & 0 \\ 0 & \P_{22}(\sigma) \end{bmatrix}.
$$
By the homotopy invariance of the spectral flow we have
$$
\SF_{\sigma=0}[\P_1(\sigma) : E_1 \subset E_0 \to E_0] = \SF_{\sigma=0}[\P(\sigma) : E_1 \subset E_0 \to E_0].
$$
Moreover, the map
$$
\hat{u}(\sigma) + \Hol_0(E_1) \longmapsto \U_1(\sigma)\hat{u}(\sigma) + \Hol_0(E_1)
$$
furnishes a unitary equivalence
$$
\bigl(\K(\P_1),[\cdot,\cdot]_{\K(\P_1)},M_{\sigma}\bigr) \cong \bigl(\K(\P),[\cdot,\cdot]_{\K(\P)},M_{\sigma}\bigr).
$$
It therefore suffices to prove the theorem for $\P_1(\sigma)$ instead of $\P(\sigma)$. Define
$$
p(\sigma):= \P_{11}(\sigma) - \P_{12}(\sigma)\P_{22}(\sigma)^{-1}\P_{21}(\sigma) \in \Germs(N),
$$
so $\P_1(\sigma) = \begin{bmatrix} p(\sigma) & 0 \\ 0 & \P_{22}(\sigma) \end{bmatrix}$. In view of the invertibility of $\P_{22}(\sigma)$ we have both
\begin{align*}
\SF_{\sigma=0}[\P_1(\sigma) : E_1 \subset E_0 \to E_0] &= \SF_{\sigma=0}[p(\sigma) : N \to N] \\
\intertext{and}
\bigl(\K(\P_1),[\cdot,\cdot]_{\K(\P_1)},M_{\sigma}\bigr) &\cong \bigl(\K(p),[\cdot,\cdot]_{\K(p)},M_{\sigma}\bigr),
\end{align*}
where the latter unitary equivalence is induced by projection onto $N$. The theorem is therefore reduced to considering the finite-dimensional case for $p(\sigma) \in \Germs(N)$, and an application of Theorem~\ref{CrossingsFlow} thus finishes the proof.
\end{proof}

We conclude with the following theorem about semibounded operators.

\begin{theorem}\label{Positivity}
Suppose $\P(\sigma) \in \Germs(E_1,E_0)$ satisfies $\P(\sigma) \geq 0$ for $\sigma$ real. Then the following hold:
\begin{enumerate}
\item $m_+(\ell) = m_-(\ell) = 0$ for $\ell$ odd. The canonical form for $M_{\sigma}$ does not contain any Jordan blocks of odd sizes, and $\sig\bigl(\K(\P),[\cdot,\cdot]_{\K(\P)}\bigr) = 0$.
\item The sign characteristic for the triple $\bigl(\K(\P),[\cdot,\cdot]_{\K(\P)},M_{\sigma}\bigr)$ does not contain any negative terms, i.e., we also have $m_-(\ell) = 0$ for $\ell$ even. 
\item There exists a unique Lagrangian subspace of $\K(\P)$ that is invariant under $M_{\sigma}$, denoted by $\K(\P)_{\frac{1}{2}}$.  Specifically, if
$$
\K(\P) = \bigoplus\limits_{j=1}^N U_j
$$
according to the Canonical Form Theorem with mutually $[\cdot,\cdot]_{\K(\P)}$-ortho\-gonal direct summands, and each $U_j$ is associated to a single Jordan block of $M_{\sigma}$ of size $(2n_j)\times(2n_j)$, then
$$
\K(\P)_{\frac{1}{2}} = \bigoplus\limits_{j=1}^N \Bigl[U_j \cap \ker M_{\sigma}^{n_j}\Bigr].
$$
\end{enumerate}
\end{theorem}
\begin{proof}
Let $\ell \in \N$, and let $\hat{u}(\sigma) + \Hol_0(E_1) \in \ker M_{\sigma}^{\ell}$ be arbitrary. Then
$$
[\hat{u},\hat{u}]_{\K(\P),\ell} = [\hat{u},M_{\sigma}^{\ell-1}\hat{u}]_{\K(\P)} = \res\limits_{\sigma=0}\langle \P(\sigma)\hat{u}(\sigma),\overline{\sigma}^{\ell-1}\hat{u}(\overline{\sigma})\rangle_{E_0}.
$$
We have $\P(\sigma)\hat{u}(\sigma),\, \sigma^{\ell}\hat{u}(\sigma) \in \Hol_0(E_0)$, and consequently
$$
\res\limits_{\sigma=0}\langle \P(\sigma)\hat{u}(\sigma),\overline{\sigma}^{\ell-1}\hat{u}(\overline{\sigma})\rangle_{E_0} = \lim\limits_{\substack{\sigma \to 0 \\ \sigma\, \textup{real}}} \sigma^{\ell}\langle \P(\sigma)\hat{u}(\sigma),\hat{u}(\sigma)\rangle_{E_0}.
$$
Note that $\langle \P(\sigma)\hat{u}(\sigma),\hat{u}(\sigma)\rangle_{E_0} \geq 0$ for real $\sigma \neq 0$ by assumption. For odd $\ell$ we thus get
\begin{align*}
\res\limits_{\sigma=0}\langle \P(\sigma)\hat{u}(\sigma),\overline{\sigma}^{\ell-1}\hat{u}(\overline{\sigma})\rangle_{E_0} &= \lim\limits_{\sigma \to 0^+} \sigma^{\ell}\langle \P(\sigma)\hat{u}(\sigma),\hat{u}(\sigma)\rangle_{E_0} \geq 0, \\
\res\limits_{\sigma=0}\langle \P(\sigma)\hat{u}(\sigma),\overline{\sigma}^{\ell-1}\hat{u}(\overline{\sigma})\rangle_{E_0} &= \lim\limits_{\sigma \to 0^-} \sigma^{\ell}\langle \P(\sigma)\hat{u}(\sigma),\hat{u}(\sigma)\rangle_{E_0} \leq 0,
\end{align*}
and so $[\hat{u},\hat{u}]_{\K(\P),\ell} = 0$, while for even $\ell$ we have
$$
\res\limits_{\sigma=0}\langle \P(\sigma)\hat{u}(\sigma),\overline{\sigma}^{\ell-1}\hat{u}(\overline{\sigma})\rangle_{E_0} = \lim\limits_{\substack{\sigma \to 0 \\ \sigma\, \textup{real}}} \sigma^{\ell}\langle \P(\sigma)\hat{u}(\sigma),\hat{u}(\sigma)\rangle_{E_0} \geq 0,
$$
and so $[\hat{u},\hat{u}]_{\K(\P),\ell} \geq 0$. This proves the first two assertions of the theorem. The third assertion now follows from \cite[Theorem~5.12.4]{GohbergLancasterRodman}.
\end{proof}


\section{Analytic crossings and spectral flow in finite-dimensional spaces}\label{AnalyticCrossingsFiniteDimensional}

\noindent
Let $F$ be a complex finite-dimensional Hilbert space. For a holomorphic operator function $p(\sigma) : B_{\varepsilon}(0) \to \L(F)$, where $\varepsilon > 0$ is sufficiently small, we define its adjoint via $p^{\star}(\sigma):= [p(\overline{\sigma})]^* : B_{\varepsilon}(0) \to \L(F)$ and note that it depends holomorphically on $\sigma$. By $\Germs(F)$ we denote the collection of all germs of holomorphic $\L(F)$-valued functions $p(\sigma)$ defined near $\sigma=0$ such that for some sufficiently small $\varepsilon > 0$ the following two properties hold:
\begin{itemize}
\item $p(\sigma) : B_{\varepsilon}(0) \to \L(F)$ is holomorphic and invertible for $\sigma \neq 0$.
\item $p$ is selfadjoint in the sense that $p^{\star} = p$ as operator functions on $B_{\varepsilon}(0)$.
\end{itemize}
In other words, for real $-\varepsilon < \sigma < \varepsilon$, the operator $p(\sigma) : F \to F$ is selfadjoint and invertible for $\sigma \neq 0$. Thus crossings of negative to positive eigenvalues (or vice versa) may occur at $\sigma = 0$ only, while $p(\sigma)$ is analytic near $\sigma = 0$. The spectral flow of $p(\sigma)$ across $\sigma = 0$ is
$$
\SF_{\sigma=0}[p(\sigma) : F \to F] = \lim\limits_{\sigma \to 0^+}\Bigl(\tr[\Pi_+(p(\sigma))] - \tr[\Pi_+(p(-\sigma))]\Bigr),
$$
where $\Pi_+$ in each case denotes the spectral projection onto the span of the eigenspa\-ces associated with positive eigenvalues of the selfadjoint operator $p(\sigma)$ or $p(-\sigma)$, respectively.

As in Appendix~\ref{ModelPairing}, every $p(\sigma) \in \Germs(F)$ induces a map
\begin{gather*}
p : \bigl(\Mero/\Hol\bigr)_0(F) \to \bigl(\Mero/\Hol\bigr)_0(F), \\
p\bigl[\hat{v}(\sigma) + \Hol_0(F)\bigr] = p(\sigma)\hat{v}(\sigma) + \Hol_0(F)
\end{gather*}
for $\hat{v}(\sigma) + \Hol_0(F) \in \bigl(\Mero/\Hol\bigr)_0(F)$. Define
\begin{align*}
\K(p) &= \ker\bigl[p : \bigl(\Mero/\Hol\bigr)_0(F) \to \bigl(\Mero/\Hol\bigr)_0(F)\bigr] \\
&= \{\hat{v}(\sigma) + \Hol_0(F) \in \bigl(\Mero/\Hol\bigr)_0(F);\; p(\sigma)\hat{v}(\sigma) \in \Hol_0(F)\}.
\end{align*}
Note that $\dim\K(p) < \infty$. We define a pairing $[\cdot,\cdot]_{\K(p)} : \K(p) \times \K(p) \to \C$ via
$$
[\hat{v},\hat{w}]_{\K(p)} = \frac{1}{2\pi i}\oint_C \langle p(\sigma)\hat{v}(\sigma),\hat{w}(\overline{\sigma})\rangle_{F}\,d\sigma = \res\limits_{\sigma=0}\langle p(\sigma)\hat{v}(\sigma),\hat{w}(\overline{\sigma})\rangle_{F},
$$
where $C$ is any positively oriented circle centered at $\sigma = 0$ of sufficiently small radius. This is a Hermitian sesquilinear form on $\K(p)$. We thus associate with every $p(\sigma) \in \Germs(F)$ the triple $\bigl(\K(p),[\cdot,\cdot]_{\K(p)},M_{\sigma}\bigr)$, where
\begin{gather*}
M_{\sigma} : \K(p) \to \K(p),  \\
M_{\sigma} : \hat{v}(\sigma) + \Hol_0(F) \mapsto \sigma\hat{v}(\sigma) + \Hol_0(F).
\end{gather*}
Note that $M_{\sigma} \in \L(\K(p))$ is nilpotent and selfadjoint with respect to $[\cdot,\cdot]_{\K(p)}$.

\begin{definition}
For $p,q \in \Germs(F)$ we write $p \sim_s q$ if there exists $\varepsilon > 0$ and an \emph{invertible} holomorphic operator function $u(\sigma) : B_{\varepsilon}(0) \to \L(F)$ such that $u^{\star}qu = p$ on $B_{\varepsilon}(0)$, i.e., $[u(\overline{\sigma})]^*q(\sigma)u(\sigma) = p(\sigma) : F \to F$ for $\sigma \in B_{\varepsilon}(0)$.
\end{definition}

Note that $\sim_s$ is an equivalence relation on the set $\Germs(F)$.

\begin{proposition}\label{InvarianceCrossingSpace}
For $p,q \in \Germs(F)$ with $p \sim_s q$ the triples $\bigl(\K(p),[\cdot,\cdot]_{\K(p)},M_{\sigma}\bigr)$ and $\bigl(\K(q),[\cdot,\cdot]_{\K(q)},M_{\sigma}\bigr)$ are unitarily equivalent. More precisely, if $u^{\star}qu = p$ with the holomorphic and invertible function $u(\sigma) : B_{\varepsilon}(0) \to \L(F)$, then
$$
T : \K(p) \to \K(q),\quad \hat{v}(\sigma) + \Hol_0(F) \mapsto u(\sigma)\hat{v}(\sigma) + \Hol_0(F)
$$
furnishes a unitary equivalence between these indefinite inner product spaces. Moreover, we have
$$
\SF_{\sigma=0}[p(\sigma) : F \to F] = \SF_{\sigma=0}[q(\sigma) : F \to F].
$$
\end{proposition}
\begin{proof}
The proof of the first statement follows immediately from the definitions. Note that $T^{-1}$ is given by multiplication by $u(\sigma)^{-1}$, which exists and is holomorphic near $\sigma = 0$. Finally, $M_{\sigma}$ commutes with multiplication by holomorphic operator functions, so $T$ is a unitary equivalence of the triples.

Regarding the spectral flow, we note that if $a = a^* \in \L(F)$ is selfadjoint, then $[v,w]_a:= \langle av,w \rangle_F$ is a Hermitian sesquilinear form on $F$. If $(m_0,m_+,m_-)$ are its invariants, then $m_+ = \tr\bigl(\Pi_+(a)\bigr)$. If $b = u^{*}au$ with an invertible $u$, then $u : \bigl(F,[\cdot,\cdot]_b\bigr) \to \bigl(F,[\cdot,\cdot]_a\bigr)$ is a unitary equivalence of these indefinite inner product spaces. Consequently, $[\cdot,\cdot]_a$ and $[\cdot,\cdot]_b$ have the same invariants, i.e., $m_+ =  \tr\bigl(\Pi_+(a)\bigr) =  \tr\bigl(\Pi_+(b)\bigr)$, which implies the invariance of the spectral flow as claimed.
\end{proof}

The following key lemma is based on several results from analytic perturbation theory (see \cite{Kato,Simon1}).

\begin{lemma}\label{BasicLemma}
Suppose $p \in \Germs(F)$ such that $p(0)$ is invertible. Then there exists an orthogonal projection $\pi = \pi^2 = \pi^* \in \L(F)$ such that $p(\sigma) \sim_s \pi - (1-\pi)$.
\end{lemma}
\begin{proof}
Choose $\varepsilon > 0$ and $\varrho \gg 0$ such that with
$$
\spec_{\pm}(p(\sigma)):= \spec(p(\sigma))\cap B_{\varrho-\varepsilon}(\pm\varrho)
$$
we have $\spec(p(\sigma)) = \spec_+(p(\sigma))\cup\spec_-(p(\sigma))$ for $|\sigma| < \varepsilon$. Define
$$
\pi(\sigma) = \frac{1}{2\pi i}\oint_{\partial B_{\varrho-\varepsilon}(\varrho)}(\lambda - p(\sigma))^{-1}\,d\lambda \in \L(F).
$$
Then $\pi(\sigma) = \pi^*(\sigma) = \pi(\sigma)^2$ is the Riesz projection onto the generalized eigenspaces associated with eigenvalues of $p(\sigma)$ that have positive real part for $|\sigma| < \varepsilon$. In particular, for $\sigma$ real, $\pi(\sigma)$ is an orthogonal projection. Note that $\pi(\sigma)$ is holomorphic in $\sigma$ in view of the Dunford integral representation formula. We'll see momentarily that the $\pi$ in the statement of the lemma is going to be $\pi:= \pi(0)$, but first define $w(\sigma)$ via
$$
\frac{1}{2\pi i}\oint_{\partial B_{\varrho-\varepsilon}(\varrho)}\lambda^{-\frac{1}{2}}(\lambda - p(\sigma))^{-1}\,d\lambda + \frac{1}{2\pi i}\oint_{\partial B_{\varrho-\varepsilon}(-\varrho)}(-\lambda)^{-\frac{1}{2}}(\lambda - p(\sigma))^{-1}\,d\lambda.
$$
Holomorphic functional calculus implies that $w(\sigma) = w^{\star}(\sigma)$ is invertible with $w^{\star}pw = \pi(\sigma) - (1-\pi(\sigma))$, so $p(\sigma) \sim_s \pi(\sigma) - (1-\pi(\sigma))$. Now, making $\varepsilon > 0$ smaller if necessary, we may further assume that $\|\pi(\sigma) - \pi(0)\|_{\L(F)} < 1$ for all $|\sigma| < \varepsilon$. Let then
\begin{align*}
u(\sigma) &= \bigl[\pi(\sigma)\pi(0) + [1-\pi(\sigma)][1-\pi(0)]\bigr]\bigl[1-[\pi(\sigma)-\pi(0)]^2\bigr]^{-\frac{1}{2}}, \\
u^{\star}(\sigma) &= \bigl[1-[\pi(\sigma)-\pi(0)]^2\bigr]^{-\frac{1}{2}}\bigl[\pi(0)\pi(\sigma) + [1-\pi(0)][1-\pi(\sigma)]\bigr].
\end{align*}
We have $u^{\star}(\sigma) = u(\sigma)^{-1}$, and $u^{\star}(\sigma)\pi(\sigma)u(\sigma) = \pi(0)$. Thus $\pi(\sigma) - (1-\pi(\sigma)) \sim_s \pi(0) - (1-\pi(0))$. In conclusion,
$$
p(\sigma) \sim_s \pi(\sigma) - (1-\pi(\sigma)) \sim_s \pi(0) - (1-\pi(0)).
$$
\end{proof}

\begin{proposition}\label{reductionprojections}
Let $p \in \Germs(F)$ be arbitrary. Then there exists $N \in \N_0$ and an orthogonal decomposition
$$
F = \bigoplus\limits_{\ell=0}^N\bigl(F_{+,\ell}\oplus F_{-,\ell}\bigr)
$$
such that
$$
p(\sigma) \sim_s \sum\limits_{\ell=0}^{N}\bigl(\pi_{F_{+,\ell}} - \pi_{F_{-,\ell}}\bigr)\sigma^{\ell},
$$
where $\pi_{F_{\pm},\ell}$ is the orthogonal projection onto $F_{\pm,\ell} \subset F$. We note that some of the spaces $F_{\pm,\ell}$ may be $\{0\}$.
\end{proposition}
\begin{proof}
We prove this proposition by induction with respect to $\dim F$.

If $F = \C$ we can write $p(\sigma) = \sigma^N\tilde{p}(\sigma)$ for some $N \in \N_0$, where $\tilde{p}(\sigma) \in \Germs(\C)$ with $\tilde{p}(0) \neq 0$. By Lemma~\ref{BasicLemma} we have $\tilde{p}(\sigma) \sim_s \pm 1$, and so $p(\sigma) \sim_s \pm\sigma^N$, thus proving the result if $\dim F = 1$.

We now assume that the proposition is valid for spaces $F$ up to dimension $k$ for some $k \in \N$. Let then $F$ be $(k+1)$-dimensional. We can write $p(\sigma) = \sigma^{\nu}\tilde{p}(\sigma)$, where $\nu \in \N_0$ and $\tilde{p}(\sigma) \in \Germs(F)$ with $\tilde{p}(0) \neq 0$, and proving the claim for $\tilde{p}(\sigma)$ implies that it holds for $p(\sigma)$ as well. Thus we assume without loss of generality that $\nu = 0$ and $p(\sigma) = \tilde{p}(\sigma)$ in the sequel.

If $p(0)$ is invertible, Lemma~\ref{BasicLemma} implies that $p(\sigma) \sim_s \pi - (1-\pi)$ for some orthogonal projection $\pi$, thus proving the assertion in this case. So suppose now that $p(0)$ is not invertible. Let $F_0 = \ker p(0)$. Then $\{0\} \subsetneq F_0 \subsetneq F$, so $1 \leq \dim F_0 \leq k$. We decompose $p(\sigma)$ as
$$
p(\sigma) = \begin{bmatrix} p_{11}(\sigma) & p_{12}(\sigma) \\ p_{21}(\sigma) & p_{22}(\sigma)\end{bmatrix} : \begin{array}{c} F_0 \\ \oplus \\ F_0^{\perp} \end{array} \to \begin{array}{c} F_0 \\ \oplus \\ F_0^{\perp} \end{array}.
$$
Note that $p_{22}(\sigma) \in \Germs(F_0^{\perp})$ is invertible for $|\sigma|$ small enough. Define
$$
u(\sigma) = \begin{bmatrix} 1 & 0 \\ -p_{22}(\sigma)^{-1}p_{21}(\sigma) & 1 \end{bmatrix} : \begin{array}{c} F_0 \\ \oplus \\ F_0^{\perp} \end{array} \to \begin{array}{c} F_0 \\ \oplus \\ F_0^{\perp} \end{array}.
$$
Then $u(\sigma)$ is invertible, and
$$
u^{\star}(\sigma) = \begin{bmatrix} 1 & -p_{12}(\sigma)p_{22}(\sigma)^{-1} \\ 0 & 1 \end{bmatrix} : \begin{array}{c} F_0 \\ \oplus \\ F_0^{\perp} \end{array} \to \begin{array}{c} F_0 \\ \oplus \\ F_0^{\perp} \end{array}.
$$
We have
$$
u^{\star}(\sigma)p(\sigma)u(\sigma) =
\begin{bmatrix} p_{11}(\sigma) - p_{12}(\sigma)p_{22}(\sigma)^{-1}p_{21}(\sigma) & 0 \\ 0 & p_{22}(\sigma) \end{bmatrix},
$$
where $q = p_{11} - p_{12}p_{22}^{-1}p_{21} \in \Germs(F_0)$. Consequently, the proposition holds for $p$ if it holds for both $q \in \Germs(F_0)$ and $p_{22} \in \Germs(F_0^{\perp})$. But the inductive hypothesis is applicable to $q$, and Lemma~\ref{BasicLemma} applies to $p_{22}$, thus proving the proposition for $p$. This completes the induction and finishes the proof.
\end{proof}

\begin{theorem}\label{CrossingsFlow}
Let $p \in \Germs(F)$ be arbitrary. Then $[\cdot,\cdot]_{\K(p)} : \K(p) \times \K(p) \to \C$ is nondegenerate, and
$$
\sig\bigl(\K(p),[\cdot,\cdot]_{\K(p)}\bigr) = \sum\limits_{\ell \; \textup{odd}}\bigl(m_+(\ell)-m_-(\ell)\bigr) = \SF_{\sigma=0}\bigl[p(\sigma) : F \to F\bigr],
$$
where $m_{\pm}(\ell)$ are the invariants associated with the triple $\bigl(\K(p),[\cdot,\cdot]_{\K(p)},M_{\sigma}\bigr)$.
\end{theorem}
\begin{proof}
Propositions \ref{InvarianceCrossingSpace} and \ref{reductionprojections} imply that it suffices to prove the theorem for the special case $p(\sigma) = \sigma^N\id : F \to F$, where $N \in \N_0$. In this situation
$$
\K(p) = \Bigl\{\sum\limits_{j=1}^{N} f_j\sigma^{-j} + \Hol_0(F);\; f_j \in F\Bigr\},
$$
and for $\hat{v}(\sigma) = \sum\limits_{j=1}^{N} f_j\sigma^{-j}$ and $\hat{w}(\sigma) = \sum\limits_{k=1}^{N} g_k\sigma^{-k}$ we get
$$
[\hat{v},\hat{w}]_{\K(p)} = \sum\limits_{k=1}^N \langle f_{N+1-k},g_k \rangle_F = \biggl\langle J\begin{bmatrix}f_1 \\ \vdots \\ f_N \end{bmatrix},\begin{bmatrix} g_1 \\ \vdots \\ g_N \end{bmatrix} \biggr\rangle_{F^N},
$$
where $J : F^N \to F^N$ is given by $J\begin{bmatrix}f_1 \\ \vdots \\ f_N \end{bmatrix} = \begin{bmatrix}f_N \\ \vdots \\ f_1 \end{bmatrix}$. We have $J^2 = \id$, so $J$ has eigenvalues $\pm 1$, proving the nondegeneracy of the pairing. The dimensions of the eigenspaces of $J$ depend on the parity of $N$. Direct computation shows that for $N$ even we have 
$\dim\ker(J-\id) = \dim\ker(J+\id) = \frac{N}{2}\dim F$, thus proving in this case that
$$
\sig\bigl(\K(p),[\cdot,\cdot]_{\K(p)}\bigr) = 0 = \SF_{\sigma=0}\bigl[p(\sigma) : F \to F\bigr].
$$
For $N$ odd we have $\dim\ker(J-\id) = \frac{N+1}{2}\dim F$ and $\dim\ker(J+\id) = \frac{N-1}{2}\dim F$, and thus
$$
\sig\bigl(\K(p),[\cdot,\cdot]_{\K(p)}\bigr) = \dim F = \SF_{\sigma=0}\bigl[p(\sigma) : F \to F\bigr].
$$
Finally, we note that if $\hat{v}(\sigma)$ and $\hat{w}(\sigma)$ above belong to $\ker[M_{\sigma}^{\ell} : \K(p) \to \K(p)]$, then
$$
[M_{\sigma}^{\ell-1}\hat{v},\hat{w}]_{\K(p)} = \begin{cases}
0 & \ell \neq N, \\
\langle f_N,g_N \rangle_F & \ell = N.
\end{cases}
$$
Thus $m_+(\ell) = m_-(\ell) = 0$ for $\ell \neq N$, while $m_+(N) = \dim F$ and $m_-(N) = 0$. The theorem is proved.
\end{proof}

\end{appendix}


\end{document}